\documentclass[11pt]{article}
\usepackage{graphicx} 
\usepackage{booktabs} 
\usepackage{tikz}
\usetikzlibrary{arrows.meta,positioning}
\usepackage{amsmath,amssymb,amsthm,enumerate}
\usepackage{tabularx,array}
\usepackage{xcolor}
\usepackage{float}
\usepackage{bm}
\usepackage[font=small]{caption}
\usepackage{subcaption}
\usepackage{algorithm}
\usepackage{algorithmic}
\usepackage{mathrsfs}
\usepackage{hyperref}
\usepackage{accents}
\usepackage{comment}

\newtheorem{theorem}{Theorem}[section]     
\newtheorem{proposition}[theorem]{Proposition}

\newtheorem{corollary}[theorem]{Corollary}

\theoremstyle{definition}

\theoremstyle{remark}
\newtheorem{remark}{Remark}
\usepackage[
authoryear, round
]{natbib}
\usepackage[margin=1.0in]{geometry}
\usepackage[normalem]{ulem}

\DeclareMathOperator*{\Var}{Var}
\DeclareMathOperator*{\Std}{Std}
\DeclareMathOperator*{\Cov}{Cov}
\DeclareMathOperator*{\Unif}{Unif}
\newcommand{\hz}[1]{\textcolor{red}{\textsf{[HZ: #1]}}}

\newcommand{\rh}[1]{\textcolor{purple}{\textsf{[RH: #1]}}}

\newcommand{\PP}{\mathbb{P}}
\newcommand{\E}{\mathbb{E}}

\newcommand{\cO}{\mathcal{O}}

\newcommand{\ba}{\mathbf{a}}
\newcommand{\bb}{\mathbf{b}}
\newcommand{\bc}{\mathbf{c}}
\newcommand{\be}{\mathbf{e}}
\newcommand{\bh}{\mathbf{h}}
\newcommand{\bk}{\mathbf{k}}
\newcommand{\bK}{\mathbf{K}}

\newcommand{\bp}{\mathbf{p}}
\newcommand{\br}{\mathbf{r}}
\newcommand{\bs}{{\mathbf{s}}}

\newcommand{\bw}{\mathbf{w}}

\newcommand{\bB}{\mathbf{B}}
\newcommand{\bH}{\mathbf{H}}

\newcommand{\bM}{\mathbf{M}}
\newcommand{\bfm}{\mathbf{m}}
\newcommand{\bI}{\mathbf{I}}
\newcommand{\bJ}{\mathbf{J}}
\newcommand{\bP}{\mathbf{P}}
\newcommand{\bS}{\mathbf{S}}

\newcommand{\bbR}{\mathbb{R}}

\newcommand{\calC}{\mathcal{C}}

\newcommand{\calG}{\mathcal{G}}
\newcommand{\calI}{\mathcal{I}}
\newcommand{\calJ}{\mathcal{J}}

\newcommand{\calM}{\mathcal{M}}
\newcommand{\calP}{\mathcal{P}}

\newcommand{\calR}{\mathcal{R}}
\newcommand{\calS}{\mathcal{S}}

\newcommand{\cR}{\mathcal{R}}

\newcommand{\bW}{\mathbf{W}}
\newcommand{\balpha}{\bm \alpha}

\newcommand{\bd}{\mathbf{d}}

\newcommand{\diag}[1]{\operatorname{diag}\!\left(#1\right)}
\newcommand{\half}{\frac{1}{2}}
\newcommand{\ubarP}{\bar{\mathbf{P}}}

\definecolor{hybridcolor}{RGB}{70,130,180} 
\definecolor{arbicolor}{RGB}{220,20,60}    

\title{\vspace{-1cm}

Modeling Stochastic Multi-Agent Interaction in Intraday Battery Energy Storage Dispatch with Market Power}

\author{ Ruimeng Hu\thanks{Department of Mathematics, and Department of Statistics and Applied Probability,
University of California, Santa Barbara, CA 93106-3080, USA.
\, Email: \texttt{rhu@ucsb.edu}.}\and Mike Ludkovski\thanks{Department of Statistics and Applied Probability,
University of California, Santa Barbara, CA 93106-3110, USA.
\, Email: \texttt{ludkovski@pstat.ucsb.edu}.} \and Hezhong Zhang\thanks{Department of Statistics and Applied Probability,
University of California, Santa Barbara, CA 93106-3110, USA.
\, Email: \texttt{hzhang586@ucsb.edu}.}}

\begin{document}

\maketitle


\begin{abstract}
We develop a stochastic game-theoretic model for intraday dispatch of grid-scale battery energy storage systems (BESSs). We assume that each BESS operator competitively manages her state-of-charge to maximize energy arbitrage revenues, driven by the endogenized electricity price that depends on the sum of the charging rates. We characterize the Nash equilibrium of the resulting finite-player linear-quadratic differential game with a shared stochastic driver, obtaining semi-explicit representations of equilibrium feedback controls and equilibrium prices both in the general heterogeneous and the simplified homogeneous BESS setting, via a system of Riccati equations. We then analyze competitive effects, including the marginal externality of additional BESS entering the market, the benefit of coordination and the corresponding market power of large operators, and supply effects from hybrid-type BESSs. We further study the asymptotic regime as the number of agents grows large. Our model provides a quantitative testbed to study the impact of decentralized BESS deployment on the grid and the resulting reduction in daily price spreads.

\end{abstract}

\noindent\textbf{Keywords:} battery energy storage system (BESS), competitive dispatch, linear-quadratic stochastic differential game

\section{Introduction}

Battery energy storage systems (BESSs) have transformed the power grid over the past decade. Complementing the ongoing deployment of weather-driven renewable energy sources, BESSs allow for smoothing and time-transfer of intermittent energy production. BESS provides an economic solution to the ``duck curve'' challenge of solar energy that creates a large intra-day ramp-down and ramp-up effect that challenges thermal generators, and moreover have proven invaluable for supporting reliability during peak demand periods by providing additional supply.  Deployment of BESS has followed an exponential curve, going from negligible to a major segment of the grid in just a few years. For example, in the California Independent System Operator (CAISO) grid, there was 0.5 GWh of installed BESS capacity in 2018, 1.5 GWh in 2020, 8.5 GWh in 2021, 27.1 GWh in 2023 \cite{caiso} and 59.6 GWh by end of 2025 \cite{modo}, with BESS now routinely providing 20-30\% of supply during the evening peak hours. In turn, widespread adoption of BESS has enabled further penetration by renewable sources. In power grids at the forefront of this change---notably CAISO and ERCOT ---there is the unmistakable pattern of setting records of BESS charging/discharging, record high renewable generation and record low net supply, all feeding off one another.

Within a deregulated power market, the vast majority of grid-scale BESS are merchant-owned. Thus, the grid is highly decentralized—--with multiple agents deploying independent BESS---shifting attention to their strategic interactions. Indeed, BESS operators act competitively and in their own self-interest, maximizing battery profitability. Specifically, BESS operators monetize their physical assets by trading on the short-term power markets, including day-ahead and real-time electricity price markets, as well as ancillary service markets. BESS dispatch decisions are then driven by price signals. The fundamental strategy is to charge the battery when prices are low and discharge it when prices are high. This price-driven behavior is aligned with grid needs through the price formation mechanism: high net load is associated with high prices and low net load is associated with low (or often negative) prices. Consequently, the buy-low sell-high strategy matches the logic of charging when there is excess supply (mid-day due to solar over-production) and discharging when there is high demand (evening when the sun sets, or early morning), broadly aligning with the first-best dispatch strategy that would be picked by a central planner.

Although the market's invisible hand is conceptually attractive, competition creates serious side-effects. BESS operators rightly worry about ``cannibalization'', whereby actions of other operators impact their own profits. Indeed, non-cooperative equilibrium results in over-extension: batteries charge more than is socially optimal when prices are low and discharge more than is optimal when prices are high. The immediate effect is that price variability, the life-blood of BESS operation, is lowered significantly, making everyone less profitable and eventually putting in danger further BESS build-out. 

In this article, we develop a tractable stochastic game model that quantifies these effects, offering a multi-faceted testbed to study the interaction of BESS on the power market. In our model, each operator manages the state-of-charge (SOC) $S^i$ of her own BESS. The corresponding dispatch control $\alpha^i$ affects future flexibility for the respective asset, and also contributes to the overall supply of electricity, impacting wholesale prices. This leads to a finite-player dynamic game in which each agent optimizes a local objective that depends both on her own actions and the aggregate behavior of others.  Notably, the game coupling is through the sum of all the players' controls, as well as through common noise interpreted as the exogenous net load process $Q$ that aggregates the rest of the power generators (thermal, renewable, hydro, etc.) as well as the non-dispatchable loads. 

In line with current grid practice, we consider two types of BESS: ``arbitrageurs'' and ``hybrid''. Hybrid BESS operators jointly dispatch the battery and the energy generated by a renewable source (in practice typically solar). Such configurations are presently the dominant type of new BESS construction, driven by regulatory advantages and the cost of renewable curtailment. Notably, we allow for correlation between the self-generation of the hybrid BESS and the net load $Q$.  Arbitrageur BESS are pure-storage assets that seek to monetize price fluctuations across hours of the day. While historically the primary revenue source for BESS was from ancillary service provision, energy arbitrage is expected to be critical for future BESS operations \citep{butters2025soaking}.

Several studies have examined the application of BESS in hybrid renewable energy systems. We build upon \cite{aung2025intraday} who investigate operation of hybrid assets on operational time scales, focusing on optimal dispatch of the coupled BESS. Their model considers wind-hybrid assets, where the BESS is dynamically controlled to counter the deviations of realized wind power output relative to a given dispatch target. By considering the intertemporal constraints of the BESS storage capacity and power rating limits, their model results in studying a state-constrained finite-horizon stochastic optimal control problem. \cite{belloni2016stochastic} use approximate dynamic programming with discretizations of the wind process and models under linear dynamics of BESS and wind with linear performance criterion.   Reinforcement learning approaches to battery dispatch are reviewed in \cite{sage2025deep}.

The literature on multi-agent BESS models is primarily limited to  mean-field formulations with a continuum of BESSs (interpreted as small residential units unlike our large, discrete grid-scale storage assets). The seminal work by \cite{alasseur2020extended} considered an extended mean field game (MFG) model for a power market with small-scale distributed generation and storage. \cite{hubert2025mean} studies the optimal investment behavior of renewable electricity producers in a competitive market, where both prices and installation costs are influenced by aggregate industry activity. \cite{dumitrescu2024price} characterizes the optimal operating strategy for a stylized storage system, assuming an arbitrary exogenous price process, and determines the equilibrium price in a market driven by an exogenous demand process. \cite{feron2022price} models the price through the mean of the agents' strategies and, by incorporating a demand forecast, identifies the optimal strategies in both finite-player and mean field game settings.  \cite{gobet2022extended} considers McKean-Vlasov-type and \cite{gobet2026federated} federated control, which are different versions of cooperative mean-field for a continuum of infinitesimal residential BESSs.  \cite{fabini2026trading} considers a ``kinetic'' McKean-Vlasov formulation where an aggregator controls the second derivative (``acceleration'') of the SOC process, ensuring that the charging rates are smooth. \cite{BalakinRoger} studies the possibility of tacit collusion between two identical BESSs that  compete with several thermal Cournot-type generators.  They show that head-on competition is not always an equilibrium since it dramatically shrinks the arbitrage spreads and instead construct profit-maximizing asymmetric leader-follower equilibria.

A related strand of energy economics studies the impact of renewables and BESSs on the power market at large. On the one hand, the variable output of renewable generators imposes reliability costs; on the other hand, by smoothing supply, storage suppresses the peaks and valleys of prices. Once the aggregate BESS capacity is large enough, batteries stop being price-takers; the resulting market power leads to dis-economies of scale. Price impact affects both the optimal behavior of an individual BESS who ought to internalize the impact of her actions, as well as introduces competitive effects between BESSs. To this end, several works have sought to construct dynamic equilibrium models of the power market in the presence of a price-making BESS sector. \cite{karadumaneconomics} utilizes a structural merit-order model of power prices to account for equilibrium effects of storage, including implications for incumbent generator bidding. BESS profitability is shown to be significantly reduced due to price effects. The recent comprehensive study by \cite{butters2025soaking} builds both a dynamic equilibrium model for the strategy of a single myopic BESS, as well as determines the aggregate capacity of BESS that can be profitable in a given market and renewable subsidy setting. As the battery fleet expands in size, BESS operations lower the variation in mean equilibrium prices across hours of the day, putting an intrinsic cap on the deployment of new BESS. The above qualitative findings have been confirmed empirically.  \cite{lamp2022large} documents a strong reduction in average intra-day wholesale price spreads in CAISO during the late 2010s as additional BESS entered the market. \cite{kirkpatrick2025estimating} estimates the impact of BESS on reducing grid congestion; congestion is a key driver of local marginal price (LMP) fluctuations. \cite{karadumaneconomics} finds similar evidence in Australia. 

In this paper, we focus on a Linear-Quadratic (LQ) formulation of the problem, which allows for analytic tractability and clear interpretation of strategic behaviors. We derive semi-explicit solutions to the Hamilton–Jacobi–Bellman (HJB) equations governing optimal regional storage control and characterize the resulting equilibrium in both the general heterogeneous and the homogeneous case by utilizing the Markovian structure of the model. The derived Riccati ODE system can be numerically solved to consider markets with up to 50 operators. Furthermore, the framework allows us to study a Major-Minor version as well as consider the asymptotics as the number of agents increases. Our setup offers several modifications of cited works. For instance, we incorporate a common exogenous generation process $Q_t$ that impacts all agents, and we also explicitly capture the price impact of individual dispatch strategies. Moreover, unlike standard MFG formulations (e.g.~\cite{feron2022price}) our power price is a function of the \emph{sum} of controls, rather than the mean control.

From the power market perspective, our analysis provides novel insights into competitive effects due to price-making BESS fleets. We quantify the price effect channel whereby by seeking to profit from intra-day price spreads, energy arbitrageurs inevitably drive those spreads down, affecting not only themselves, but also other market stakeholders. Our model allows us to separate this impact into a competition effect and a supply effect. For the former, we moreover carry out new experiments that document the value of coalition formation and the potential benefit of aggregating multiple BESS units into a ``cartel''.

The rest of the article is organized as follows. Section \ref{sec:formulation} sets up the dynamics of the competitive BESS market and defines the optimization problems of the agents. Section \ref{sec:HJB_sys} provides the semi-explicit solution of the general heterogeneous and the special homogeneous cases in terms of a system of Riccati ODEs. In Section \ref{sec:comparative}, we investigate various sensitivities of our model, including BESS cost parameters, number of agents in the market, and their correlation. Section \ref{sec:major-minor} analyzes the impact of BESS sizing, including Major-Minor and oligopoly setups, which shed light on the competitive effects and the merit of coalition-forming. Section \ref{sec:asymptotics} provides asymptotic analysis as the number of operators goes to infinity, connecting to a mean-field formulation; Section \ref{sec:conclusion} concludes. The proofs of all the theorems are delegated to the Appendix.

\section{Power Market with Competitive BESS}\label{sec:formulation}
In this section we describe our power market model  which focuses on hybrid renewable electricity suppliers who competitively provide power. The suppliers interact through power prices $P_t$ which in turn depend on the aggregate supply of power. To focus on the behavior of the BESS and the respective state-of-charge dynamics, we aggregate all other producers into a total power production process $(Q_t)$ and concentrate on the interaction among $N$ BESS agents\footnote{Hereafter, we shall use \emph{agent} and \emph{operator} interchangeably.}. 

The basic economics of hybrid BESS operation is to maximize revenue given stochastic generation, battery state-of-charge (SOC) and power price processes. Specifically, the hybrid asset generates electricity at a rate $\mu_t$. That electricity can be stored in the BESS or sold to the grid (or additional electricity can be purchased and stored), with the grid-facing charge/discharge rate denoted as $\alpha_t$. In turn, the $\alpha_t$'s of the BESS operators and the non-renewable generation $Q_t$ determine the price $P_t$. Finally, to capture engineering features of the storage asset, we introduce frictions regarding the charge/discharge rates and the state-of-charge. 

To achieve maximal tractability, we work in a linear-quadratic framework that keeps all stochastic dynamics affine in the state variables. This allows us to convert equilibrium computation to a system of Riccati ODEs, which in turn can be efficiently scaled to (essentially arbitrarily) large markets. While restrictive in some sense, our setup is flexible in several other features that we incorporate in our formulation. Specifically, we (i) keep most coefficients explicitly time-dependent to capture the inherent intraday and weekly seasonality of power markets and prices; (ii) allow for common noise between the non-renewable generation $Q_t$ and the hybrid generation $\mu_t$; (iii) introduce regional asset-specific prices $P^i_t$ that proxy the locational marginal prices prevalent in electricity markets (e.g., arising due to transmission grid constraints); (iv) introduce regional price impact that can capture geographic separation among BESS operators. Although competition inherently cannibalizes profits, these effects are heterogeneous, with some operators being more exposed to competitive spillovers than others.  

\subsection{State Dynamics}
To formally define our mathematical setup, consider a probability space $(\Omega, \mathcal{F}, \PP)$. We shall consider $N$ competitive agents that face $N+1$ sources of risk. Our finite horizon $[0,T]$ should be interpreted as a single planning period, typically $T=24$ hours; the description below is motivated by the typical intra-day dynamics of power markets and renewable generation. 

The exogenous, non-renewable power production $(Q_t)$ evolves autonomously and is driven by the common noise (Wiener) process $(W^0_t)$. This noise reflects grid-wide uncertainties, such as weather conditions that affect all agents simultaneously. We assume that $Q_t$ follows a mean-reverting Ornstein-Uhlenbeck process that fluctuates around the time-dependent reference level $\{ \theta(t) \}$:
\begin{equation} \label{Q process}
    dQ_t = \kappa(\theta(t) - Q_t)\,dt + \sigma^0\,dW_t^0,
\end{equation}
where $\kappa > 0$ denotes the mean reversion rate and $\sigma^0$ the volatility of generation.  

The instantaneous mean generation rate of each hybrid asset $i$ is denoted by $\mu^i_t$. We assume the form
\begin{equation}\label{mu_Q}
    \mu^i_t = a^i(t) + b^i(t) Q_t, \qquad i=1,\ldots,N,
\end{equation}
so that individual power generation processes are correlated with the exogenous $Q_t$ via the deterministic factors $b^i(t)$ and also include a time-dependent but deterministic supply process $a^i(t)$.  Interpreting $\mu^i$ as solar generation, $a^i$ represents the idiosyncratic irradiance at the asset location, while $b^i$ captures the systematic grid-wide factor (regional weather conditions affecting all solar suppliers on the grid). 

BESS operators choose their control processes $(\alpha^i_t)$ that represent the charge/discharge rates from the battery: $\alpha_t^i > 0$ corresponds to charging the battery by purchasing energy from the grid, while $\alpha_t^i < 0$ corresponds to discharging or selling energy back to the grid.  The respective state process is constituted by the BESS State-of-Charge (aka the storage level) $S_t^i$ that follows
\begin{align}\label{eq:tilde-S}
 dS^i_t = \underbrace{\mu^i_t dt + \sigma^i dW^i_t}_{\text{Self-Generation}} + \underbrace{\alpha^i_t dt}_{\text{BESS Dispatch}} = (\mu^i_t + \alpha^i_t) dt + \sigma^i dW^i_t, \qquad i=1,\ldots, N,
\end{align}
where $\sigma^i$ is the $i$-th BESS volatility coefficient and $(W_t^i)$ is a Brownian motion representing local uncertainty specific to asset $i$. The interpretation of the $\sigma^i dW^i_t$ term is the stochasticity of the underlying generation beyond the drift $\mu^i_t$. We assume that $W^i$ has a constant correlation $\rho^i$ with the common noise driver $W^0$. Plugging this as well as \eqref{mu_Q}  into \eqref{eq:tilde-S} yields the ultimate dynamics 

\begin{equation}\label{S process}
    dS_t^i = (a^i(t) + b^i(t)Q_t +\alpha_t^i)\,dt + \rho^i\sigma^i\,dW_t^0 + \sigma^i\sqrt{1-(\rho^i)^2}\,dW_t^i
\end{equation}
where the $N+1$ Wiener processes $W^i,  i = 0, 1, \ldots, N$ are all independent.

We denote by $\bS_t = (S^1_t, \ldots S^N_t)$ the vector of all the BESS state-of-charge at time $t$, assumed to be publicly known by all agents.
The control for each agent $i$ is postulated to be in feedback form
\begin{equation}\label{alpha_t}
    \alpha_t^i = \alpha^i(t,Q_t,\bS_t),
\end{equation}
where $(Q_t,\bS_t)\in \bbR^{N+1}$ denotes the full state vector at time $t$, and $\alpha^i: [0,T] \times \bbR^{N+1} \to \bbR$ is a measurable function specifying the control policy for agent $i$. These control policies must belong to the admissible set $\mathcal A_0$, where for $t \in [0,T]$
\[
\mathcal{A}_t := \Big\{ (\alpha_s)_ {s \in [t, T]} \text{ progressively measurable w.r.t. } (\mathcal{F}_s)_{s\geq t}: \mathbb{E}  \Big[\int_t^T |\alpha_s|^2 ds  \Big] < \infty \Big\},
\]
and $(\mathcal{F}_t)_{t \in [0,T]}$ is the filtration generated by $(Q_s, \bS_s)$.
We use $\balpha \equiv( \alpha^1, \ldots, \alpha^N) \in \bbR^N$ to denote the vector, i.e., the control profile, of the BESS charging rates. 
\subsection{Economic Objective}\label{sec:objective}
The primary goal of each agent $i=1,\ldots, N$ is to maximize cumulative revenue over $[0,T]$ which is 
\[
-\int_0^T P^i_t \alpha^i_t dt
\]
where $P_t^i$ is the electricity locational marginal price for BESS $i$, and $\alpha_t^i$ is the above control (charging/discharging rate). The regional price $P^i_t$ is affected both by the actions of the $i$th agent and the decisions of other agents. We assume that the price $P^i$ is determined by a supply-demand equilibrium with a linear demand curve: 
\begin{align}\label{eq:price-impact}
    P_t^i = \bar{P}^i - c_1^i\big(Q_t - \bw_{i}^\top\balpha_t\big),
\end{align}
where $\bar{P}^i$ denotes the base price, $c_1^i > 0$ reflects the price sensitivity, and the local net supply is $Q_t - \bw_i^\top \balpha_t$. Above, $\bw_{i} \in \bbR^N_+$ is a weight vector: $w_{ij}$, which is the $j$-th element of $\bw_i$, captures the influence of the $j$-th BESS on agent $i$. Intuitively, $w_{ij}$ is small when $i$ and $j$ are far apart on the grid and (close to) 1 when $i$ and $j$ are near each other, with $w_{ii}=1$. Thus, discharging by other agents lowers the price for agent $i$ proportionally by $c^i_1 w_{ij}$, and our model can be viewed as a simplified proxy for the network congestion estimated in \cite{kirkpatrick2025estimating}. The base case is $w_{ij} = 1$ for all $i,j$, in which case $P^i_t = \bar{P}^i - c_1^i(Q_t -\sum_i \alpha^i_t)$ and the price is linear in the total \emph{net supply}. The revenue rate $-P^i_t \alpha^i_t$ is therefore linearly decreasing in exogenous supply $Q_t$, linearly decreasing in the discharge rate of other assets $\alpha^j$, and quadratically decreasing in the BESS own discharge rate $\alpha^i$. 

There is strong economic evidence that storage affects prices, i.e., that $P_t^i$ depends on $\alpha_t^i$. In particular, several studies \citep{butters2025soaking,lamp2022large,karadumaneconomics} econometrically estimate the resulting elasticities in real-life markets, which can be directly linked to our $c_1, w_{ij}$ parameters.

To reflect the operational characteristics of the battery, specifically finite charging capacity and finite storage capacity, running costs are imposed on  $\alpha_t^i$ and $S_t^i$. Specifically, we include a quadratic cost $c_2^i (\alpha_t^i)^2$ that penalizes aggressive charging/discharging (also important to reduce long-term battery degradation) and a quadratic cost $c_3^i (S_t^i - \zeta^i(t))^2$ that penalizes SOC deviations from the target level $\zeta^i(t)$.  The default $\zeta^i(t)=\bar{C}^i/2$ is interpreted as half the energy capacity of the BESS. The cost $c_3^i$ can be linked to round-trip efficiency of the battery technology and to battery degradation costs, both of which act to discourage deep cycling and hence penalize fluctuations in $S_t$.

\begin{remark}
More realistically, there are hard physical caps on the magnitude of $\alpha_t^i$ and of $S^i_t$ \citep{aung2025intraday}. In order to keep the LQ structure, we proxy these hard constraints by the soft quadratic penalties. We proxy the constraint $S^i_t \in [0, \bar{C}^i]$ with the penalty $c_3^i (S_t^i - \bar{C}^i/2)^2$. Another motivation is that $\zeta^i(t)$ represents a dispatch target process and the agent wishes to keep the battery SOC close to this level. See \cite{aung2025intraday} for an empirical analysis of the accuracy of replacing such nonlinear hard constraints with quadratic penalties. 
\end{remark}

Finally, we impose a quadratic terminal cost  at $T$ for SOC deviations from the final target SOC level $c_4^i (S_T^i - \zeta^i(T))^2$. The terminal constraint prevents economically irrelevant behavior due to the finite horizon (for example, trying to discharge all stored energy regardless of supply  and price signals). 
The cost parameters $c_2^i \ge 0, c_3^i \ge 0, c_4^i \ge 0$ are non-negative and can be zero. Note that typically we set $c_4^i \gg c_3^i$ in order to have a tight range for the resulting $S^i_T$ terminal SOC levels.

Combining all these cost components, each BESS operator $i$ aims to \emph{minimize} her total expected cost functional:
\begin{align} \label{cost function}
    J^{i}(\balpha) 
    &= \mathbb{E}\Big[\int_0^T \big\{P_t^i\,\alpha_t^i + c_2^{i} (\alpha_t^i)^2 + c_3^{i}(S_t^{i} - \zeta^i(t))^2 \big\}\,dt + c_4^{i}(S_T^{i} - \zeta^i(T))^2\Big]  \\
    &= \mathbb{E}\Big[\int_0^T \big[\underbrace{\big(\bar{P}^i - c_1^i (Q_t - \bw_{i}^\top\balpha_t)\big)\alpha_t^i}_{-\text{Running Revenue}} + \underbrace{c_2^i (\alpha_t^i)^2 + c_3^i\big(S_t^{i} - \zeta^i(t)\big)^2}_{\text{Running Cost}} \big]\,dt + c_4^i\underbrace{\big(S_T^{i} - \zeta^i(T)\big)^2}_{\text{Terminal Cost}}\Big],\nonumber
\end{align}
which depends on the entire control profile $\balpha_t$ due to the $c^i_1 (\bw_{i}^\top\balpha_t) \alpha^i_t = c^i_1 \sum_{j=1}^N w_{ij} \alpha^j_t \alpha^i_t$ term. Moreover, agent behavior is coupled due to dependence on the exogenous common factor $(Q_t)$.

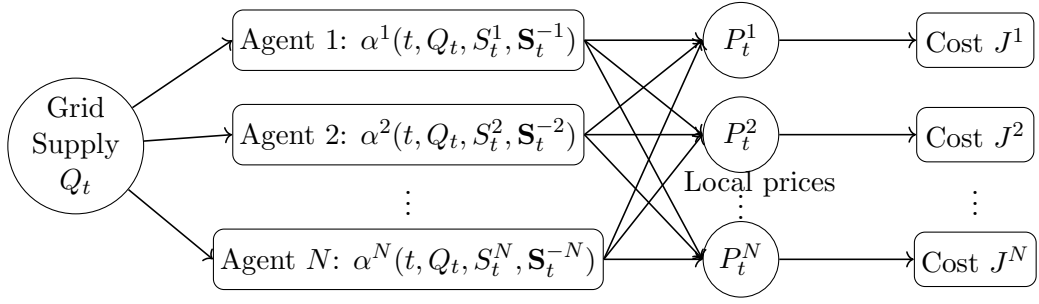
\begin{figure}[t]
\centering
\begin{tikzpicture}[
  box/.style={
    draw,
    rounded corners,
    align=center,
    inner sep=2.5pt,
    minimum width=3.5cm,
    minimum height=0.8cm
  },
  circ/.style={
    draw,
    circle,
    align=center,
    inner sep=0.5pt,
    minimum size=1 cm
  },
  pbox/.style={
    draw,
    rounded corners,
    align=center,
    inner sep=2.5pt,
    minimum width=1.55cm,
    minimum height=0.72cm
  },
  arr/.style={->, semithick}
]

\node[circ] (Q) at (-4.4,  0) {Grid \\ Supply \\ $Q_t$};

\node[box] (A1) at (0,  1.4) {Agent 1: $\alpha^1(t,Q_t,S_t^1, \bS_t^{-1})$};
\node[box] (A2) at (0,  0.15) {Agent 2: $\alpha^2(t,Q_t,S_t^2, \bS_t^{-2})$};
\node        at (0, -0.65) {$\vdots$};
\node[box] (AN) at (0, -1.5) {Agent $N$: $\alpha^N(t,Q_t,S_t^N, \bS_t^{-N})$};

\node[] at (4.65, -.5) {Local prices};

\node[circ] (P1) at (4.4,  1.4) {$P_t^1$};
\node[circ] (P2) at (4.4,  0.15) {$P_t^2$};
\node        at (4.4, -0.7) {$\vdots$};
\node[circ] (PN) at (4.4, -1.5) {$P_t^N$};

\node[pbox] (R1) at (7.5,  1.4) {Cost $J^1$};
\node[pbox] (R2) at (7.5,  0.15) {Cost $J^2$};
\node        at (7.5, -0.65) {$\vdots$};
\node[pbox] (RN) at (7.5, -1.5) {Cost $J^N$};


\draw[arr] (P1) -- (R1);
\draw[arr] (P2) -- (R2);
\draw[arr] (PN) -- (RN);

\draw[arr] (Q) -- (A1.west);
\draw[arr] (Q) -- (A2.west);
\draw[arr] (Q) -- (AN.west);

\foreach \A in {A1,A2,AN}{
  \foreach \P in {P1,P2,PN}{
    \draw[arr] (\A.east) -- (\P.west);
  }
}
\end{tikzpicture}
\caption{Finite-player coupling structure between individual controls $\alpha^i$, individual SOC $S^i$, shared driver $Q_t$, aggregate SOC state $\bS^{-i}_t$ and local prices $P^i_t$.}
\end{figure}

Our model extends the single-agent setting of \cite{aung2025intraday}  to an $N$-agent game, lifting the stochastic control problem to a finite player non-cooperative setup. Let $\balpha^{-i} = (\alpha^1,\alpha^{i-1},\alpha^{i+1},\ldots,\alpha^N)$ be the collection of controls except for agent $i$'s, then operator $i$'s objective can be written as
\[
    J^i(\balpha) = J^i(\alpha^i,\balpha^{-i}),
\]
and interpreted as if we freeze the controls of all other agents $\balpha^{-i}$ while agent $i$ faces a standard stochastic control problem of finding $\inf\limits_{\alpha^i \in \mathcal{A}_0} J^i(\alpha^i,\balpha^{-i})$.

If every agent simultaneously adopts such a best-response strategy, the resulting control profile $\hat\balpha = (\hat\alpha^1,\ldots\hat\alpha^N) \in \bbR^N$ constitutes a \textbf{Nash Equilibrium} (NE). In other words, $\hat\balpha$ is NE if no agent can reduce her cost by unilateral deviation
\begin{align*}
     J^i(\hat\alpha^i,\hat\balpha^{-i}) \leq J^i(\alpha,\hat\balpha^{-i}), \quad \forall i \in \{1,2,\ldots,N\} \text{ and } \alpha \in \mathcal{A}_0. 
\end{align*}
Assuming the infimum is attained, then $\hat\alpha^i$ satisfies
\begin{align*}
    J^i(\hat\alpha^i,\hat\balpha^{-i})=\inf_{\alpha^i} J^i(\alpha^i,\hat\balpha^{-i}), \quad \hat\alpha^i \in \arg\min_{\alpha^i} J^i(\alpha^i,\hat\balpha^{-i}).
\end{align*}
Since each cost functional depends on the full control profile $\balpha$ and the joint state dynamics $(Q_t, \bS_t)$, both of which are influenced by all agents, characterizing a Nash equilibrium requires solving $N$ coupled best-response problems simultaneously.

In the case where admissible controls take the Markov feedback form \eqref{alpha_t}, the conditional cost-to-go for agent $i$ is Markovian in the state variables $(Q_t,\bS_t)$. Thus, for a fixed profile $\bar{\balpha}^{-i}$ of others' strategies, we define the best-response value function of operator $i$ by,
\begin{align}\label{general_value_function}
    V^i(t,q,\bs; \bar{\balpha}^{-i}) &= \inf_{\alpha^i \in \mathcal{A}_t}\mathbb{E}\Big[\int_t^T \big(\bar{P}^i - c_1^i\big(Q_s - w_{ii}\alpha^i_s - \sum_{j\neq i}w_{ij}\bar{\alpha}^j(s,Q_s,\bS_s)\big)\big)\alpha_s^i \nonumber \\ + c_2^i (\alpha_s^i)^2 & + c_3^i(S_s^{i} - \zeta^i(s))^2\,ds + c_4^i(S_T^{i} - \zeta^i(T))^2 \, \big|\, Q_t = q,\, \bS_t =\bs\Big],\, i= 1,2,\ldots,N.
\end{align}
At a Nash equilibrium $\hat\balpha$, the game value of operator $i$ is then given by $V^i(t,q,\bs; \hat{\balpha}^{-i})$. The equilibrium $\hat \balpha$ and the corresponding game value  $V^i(t,q,\bs; \hat{\balpha}^{-i})$ are characterized jointly through a system of coupled Hamilton–Jacobi–Bellman equations. In the next section, we derive this system and provide a semi-explicit representation of $V^i(t,q,\bs; \hat{\balpha}^{-i})$ (cf. Theorem~\ref{General Theorem}). With a slight abuse of notation, we henceforth denote the equilibrium game value simply by $V^i(t, q, \bs)$.

\section{HJB System and Semi-explicit Solutions}\label{sec:HJB_sys}

In this section we present our main methodological results that provide a semi-explicit solution to the Nash equilibrium characterized by the system \eqref{general_value_function}. We start with the general case and then specialize to the homogeneous setting where all agents have identical objectives and state dynamics. All proofs are in Appendix \ref{General Theorem Proof}.

\subsection{General Heterogeneous agents}\label{hetero_agents}

Define the vector notations $\ubarP := [\bar{P}^1,\ldots,\bar{P}^N]^\top$, $\bc_1 := [c_1^1,\ldots,c_1^N]^\top$ and $\bc_2 := [c_2^1,\ldots,c_2^N]^\top$, and the matrix $\bW :=[\bw_1,\ldots,\bw_N]^{\top}$. Moreover, we denote by $\bI_N$ the identity matrix of size $N$, and by $\bJ_N$ a matrix of all ones. 

\begin{theorem}[General Case] \label{General Theorem}
In equilibrium, the Markovian value functions $V^i(t,q,\bs)$, $i = 1, \ldots, N$, satisfy the coupled HJB system: 
\begin{align}
& \partial_t V^i + \inf_{\alpha^i} \Big\{ 
      \alpha^i \big(\bar{P}^i- c_1^i(q -  \bw_{i}^{\top}[\hat\alpha^1, \ldots, \alpha^i, \ldots, \hat\alpha^N]) \big)
    + c_2^i (\alpha^i)^2 + (\alpha^i + a^i(t)+b^i(t) q) \partial_{s^i} V^i\Big\} \notag\\
 &   + c_3^i (s^i - \zeta^i(t))^2
    + \kappa(\theta(t) - q)\partial_q V^i + \sum_{j\neq i}  (\hat\alpha^j + a^j(t)+b^j(t) q) \partial_{s^j} V^i + \frac{1}{2}\mathrm{Tr}(\Sigma\Sigma^\top \partial_{q,\bs}^2V^i)=0,
\label{general HJBE-PDE}
\end{align}
with terminal condition $V^i(T, q, \bs) = c_4^i (s^i - \zeta^i(T))^2$. Here $\partial_{q, \bs}^2 V^i$ denotes the Hessian matrix of $V^i$, and $\Sigma \in \mathbb{R}^{(N+1) \times (N+1)}$ is the covariance matrix with non-zero entries $\Sigma_{00} = \sigma^0$, $\Sigma_{i0} = \sigma^i\rho^i$ and $\Sigma_{ii} = \sigma^i\sqrt{1-(\rho^i)^2}$, $i = 1, \ldots, N$.

Agent $i$'s equilibrium strategy $\hat\alpha^i(t,q,\bs)$ is given by
\begin{align}\label{opt_alpha}
    \hat \alpha^i(t, q, \bs) &= \bfm_{i} \diag \bd ^{-1}(\bc_1 q -[\partial_{s^1}V^1,\ldots,\partial_{s^N}V^N]^{\top}(t, q, \bs) -\ubarP),
\end{align}  
where  $\bfm_i$ is the $i$-th row vector of $\bM$, and 
\begin{equation} \label{d,U,u}
    \bd := [w_{11}, \ldots, w_{NN}]^\top \odot \bc_1+ 2\bc_2, \quad 
    \bM := (\bI_N + \diag \bd ^{-1} \diag {\bc_1} \bW)^{-1},
\end{equation}
provided that $\bI_N + \diag \bd ^{-1} \diag {\bc_1} \bW$ is invertible. 

Then, the game value  $V^i(t, q, \bs)$ of agent $i$ admits a quadratic form:
\begin{equation}\label{ansatz}
    V^i(t, q, \bs) = \begin{bmatrix}
q \\
\bs
\end{bmatrix}^\top \begin{bmatrix}
p^i_{0}(t) & \bp^i(t)^\top \\
\bp^i(t) &  \bP^i(t)
\end{bmatrix}
\begin{bmatrix}
q \\
\bs
\end{bmatrix} + \begin{bmatrix}
r_0^i(t) \\
\br^i(t)
\end{bmatrix}^\top \begin{bmatrix}
q \\
\bs
\end{bmatrix} + u^i(t),
\end{equation}
with $(\bP^i,\bp^i, \br^i,p^{i}_{0}, r_0^i,  u^i)(t)$ satisfying the $N$-coupled ODE system: 
\begin{equation} \label{general_ode_sys}
\left\{
\begin{aligned}
- \dot{\bP}^i(t) &= \frac{1}{2}  
    c_1^i (\bk_2^i(t)^{\top}\bk_5^i(t)+\bk_5^{i}(t)^{\top}\bk_2^i(t)) + c_2^i \bk_2^i(t)^{\top}\bk_2^i(t)  +  c_3^i \be_i \be_i^\top 
-   \bP^i(t) \bK_2(t) 
-  \bK_2^\top(t)  \bP^i(t), \\
- \dot{\bp}^i(t) 
&= -\frac{1}{2}c_1^i \big( k_1^i(t)\bk_5^i(t)^{\top} - \bk_2^{i}(t)^{\top}k_4^i(t) \big)-  c_2^i k_1^i(t)\bk_2^i(t)^{\top}
 - \kappa \bp^i(t) \\
 &\quad + \big(  \bP^i(t) (\bk_1(t)+\bb(t)) - \bK_2^\top(t) \bp^i(t) \big), \\
 - \dot{\br}^i(t) &= -\bar{P}^i \,\bk_2^i(t)+ c_1^i (\bk_2^i(t)^{\top}k_6^i(t)+k_3^i(t)\bk_5^i(t)^{\top}) + 2 c_2^i \bk_2^i(t)^{\top}k_3^i(t)
- 2 c_3^i \zeta^i(t) \be_i \\
&\quad +2 \kappa \theta(t) \bp^i(t) - \bK_2^\top(t) \br^i(t) - 2  \bP^i(t)(\bk_3(t)-\ba(t)), \\
    -\dot{p}^{i}_{0}(t) 
&= -c_1^i k_1^i(t)k_4^i(t) + c_2^i (k_1^i(t))^2 - 2\kappa p^{i}_{0}(t) + 2 \bp^i(t)^\top (\bk_1(t)+\bb(t)), \\
- \dot{r}_0^i(t) &= \bar{P}^i\, k_1^i(t)+ c_1^i \big( k_3^i(t)k_4^i(t)+k_1^i(t)k_6^i(t) \big) - 2 c_2^i k_1^i(t)k_3^i(t)  + \kappa \big( 2 \theta(t) p^{i}_{0}(t) - r_0^i(t) \big)  \\
&\quad +\big( \br^i(t)^\top (\bk_1(t)+\bb(t)) - 2 \bp^i(t)^\top (\bk_3(t)-\ba(t)) \big), \\
- \dot{u}^i(t) &= -\bar{P}^i\,k_3^i(t) + c_1^i\,k_3^i(t)k_6^i(t) + c_2^i \big( k_3^i(t)\big)^2 + c_3^i \big(\zeta^i(t)\big)^2
+ \kappa r_0^i(t) \theta(t) 
- \br^i(t)^\top (\bk_3(t)-\ba(t)) \\
&\quad+ \mathrm{Tr}\Big(\Sigma\Sigma^\top \big[\begin{smallmatrix}
p^{i}_{0}(t) & \bp^i(t)^\top \\
\bp^i(t) &  \bP^i(t)
\end{smallmatrix}\big]\Big),
\end{aligned}
\right.
\end{equation}
and terminal conditions
\begin{equation}\label{general_terminal}
\bP^i(T)= c_4^i \be_i \be_i^\top, \,
\bp^i(T) = \mathbf{0}, \,
\br^i(T) = -2 c_4^i \zeta^i(T) \be_i, \,
p^{i}_{0}(T) = 0,\,
r_0^i(T) = 0, \,
u^i(T) = c_4^i \zeta^i(T)^2.
\end{equation}
The intermediate quantities are
\begin{align} \label{intermediates}
    \bk_1(t) &:= \bM\begin{bmatrix}
        \frac{c_1^{1}-2\be_1^{\top}\bp^{1}(t)}{d_1}\\
        \vdots \\
        \frac{c_1^{N}-2\be_N^{\top}\bp^{N}\!(t)}{d_N} 
    \end{bmatrix},&\quad
    \bK_2(t) &:= \bM\begin{bmatrix}
        \frac{2\be_1^{\top}\bP^{1}(t)}{d_1}\\
        \vdots \\
        \frac{2\be_N^{\top}\bP^{N}\!(t)}{d_N} 
    \end{bmatrix},&  \quad 
    \bk_3(t) &:= \bM\begin{bmatrix}
        \frac{\be_1^{\top}\br^{1}(t)+\bar{P}^1}{d_1}\\
        \vdots \\
        \frac{\be_N^{\top}\br^{N}\!(t)+\bar{P}^N}{d_N} 
    \end{bmatrix},& \\
    \bk_4(t) &:= \mathbf{1}_N-\bW \bk_1(t),& \quad 
    \bK_5(t) &:= \bW \bK_2(t),& \quad
    \bk_6(t) &:= \bW \bk_3(t),& \notag
\end{align}
where  $k_1^i, k_3^i, k_4^i, k_6^i \in \bbR$, are the $i$-th element of $\bk_1, \bk_3, \bk_4, \bk_6 \in \bbR^N$, $\bk_2^i, \bk_5^i \in \bbR^{1\times N}$, are the $i$-th row vector of $\bK_2, \bK_5 \in \bbR^{N \times N}$, $\ba(t) = [a^1(t),\ldots,a^N(t)]^\top, \bb(t) = [b^1(t),\ldots,b^N(t)]^\top$ and $\mathbf{1}_{N}$ is the column vector of all ones of size $N$, and $\be_i$ is the $i$-th standard basis vector of $\bbR^N$. 
\end{theorem}

Theorem \ref{General Theorem} leverages the linear-quadratic structure to express $V^i(t,q,\bs)$ in the form of \eqref{ansatz}, where the coefficients solve the system of Ricatti ordinary differential equations (ODEs) \eqref{general_ode_sys}.
The detailed proof is deferred to Appendix \ref{General Theorem Proof}.
Observe that  \eqref{general_ode_sys}  has a triangular structure with the highest-order quadratic coefficients $\{\bP^i(t)\}_{i=1}^N$ forming a an autonomous subsystem, since the quantities $\bK_2(t)$ and $\bK_5(t)$ appearing in the first line (RHS of $-\dot{\bP}^i(t)$) depend only on $\{\bP^j(t)\}_{j=1}^N$. After that Riccati subsystem is solved, the remaining coefficients enter only through linear or affine-linear equations with time-dependent coefficients determined by $\{\bP^i(t)\}_{i=1}^N$.

\begin{proposition}[Well-Posedness]\label{general_wp}
The well-posedness of the $N$-coupled ODE system \eqref{general_ode_sys} is determined by the existence and uniqueness of the Riccati subsystem $\bP^i$, $i = 1,\ldots,N$. Define 
\begin{align}\label{QSK_0}
\calC_3 &:= \diag{c_3^1\be_1\be_1^\top,\ldots,c_3^N\be_N\be_N^\top}\in\bbR^{N^2\times N^2} \!,\quad
\calS := \diag{\frac{2}{d_1}\bM\be_1\be_1^\top,\ldots,\frac{2}{d_N}\bM\be_N\be_N^\top}\in\bbR^{N^2\times N^2},\nonumber\\
\calC_4 &:= \diag{c_4^1\be_1\be_1^\top,\ldots,c_4^N\be_N\be_N^\top}\in\bbR^{N^2\times N^2},
\end{align}
and let $\bB^i\in\bbR^{N\times N}$, $i=1,\ldots,N$ be the matrices with entries
\begin{equation}\label{B_entries}
b^i_{jl} :=
\frac{2c_1^i}{d_jd_l}
\Big(
m_{ij}\be_j^\top(\bW\bM)\be_l
+
\be_i^\top(\bW\bM)\be_j\,m_{il}
\Big)
+
\frac{4c_2^i}{d_jd_l}m_{ij}m_{il},
\end{equation}
where $m_{ij}$ is the $(i,j)$-th entry of $\bM$.
Then the condition $T < \frac{1}{2N\sqrt{\|\calC_3\|\beta}}\big[\pi - 2\tan^{-1}\big(\frac{N \sqrt{\beta}\|\calC_4\|}{\sqrt{\|\calC_3\|}}\big)\big]$, where  $\beta = \|\calS\|+\|\calS^\top\| + \max_{1\leq i\leq N}\| \bB^{i} \|$, $\|\cdot\|$ being any matrix-induced norm, is sufficient for \eqref{general_ode_sys} to be well-posed on $[0,T]$. 
\end{proposition}
The proof of Proposition \ref{general_wp} is given in Appendix \ref{General Theorem Proof}.
\begin{remark}
When $N=1$, the game reduces to a stochastic control problem. In this case, all quantities in \eqref{intermediates} are scalar and there is  system of 6 ODEs.

For $N$ agents, the total number of ODEs to be solved is $N^3 + 2N^2 + 3N$.
Indeed, for each agent $i$, the system \eqref{general_ode_sys} requires solving for the scalar 
$p^{i}_{0}(t)\in\mathbb{R}$, the vector
$\bp^i(t)\in\mathbb{R}^N$, each entry of the matrix 
$\bP^i(t)\in\mathbb{R}^{N\times N}$, the scalar
$r_0^i(t)\in\mathbb{R}$, the vector
$\br^i(t)\in\mathbb{R}^N$,
and the scalar $u^i(t)\in\mathbb{R}$,
leading to $N^2+2N+3$ ODEs per agent and hence $N^3+2N^2+3N$ coupled equations in total.

\end{remark}

In our experiments, we solve the coupled ODE system \eqref{general_ode_sys} using a standard explicit Runge-Kutta method (RK4) with a sufficiently small time step. The above approach can handle up to $\sim 50$ distinct agents (corresponding to $\sim 130,000$ equations). 

\subsection{The Homogeneous Case} \label{homo_sys}

In this section, we consider the homogeneous agent case which considerably reduces the number of ODEs to be solved. Specifically, we set 
\begin{equation}\label{eq:homo-condition}
    \bW = \bJ_N, \; \sigma^i = \sigma, \; \rho^i = \rho, \; a^i(t) = a(t), \; b^i(t) = b(t), \; \bar P^i = \bar P, \;  c_m^i = c_m \; (m = 1, 2, 3, 4), \; \zeta^i(t) = \zeta(t),
\end{equation}
for $i=1,2,3, \ldots,N$. Under these assumptions, all model parameters are identical across agents, and so are their game values. We therefore omit the superscript $i$ when no confusion arise.  We now state the main result for this case.

\begin{theorem}[Homogeneous case] \label{Homo_case}
Assuming~\eqref{eq:homo-condition}, the game value $\tilde V^i(t,q, \bs)$ for agent $i$ takes the form
\begin{align}\label{homoHJBE}
\tilde V^i(t,q, s^i, \bs^{-i}) =\, & 
\tilde p_1(t)\, q^2 + 2 \tilde p_2(t)\, s^i q + 2 \tilde p_3(t)\, \mathbf{1}_{N-1}^\top \bs^{-i} q + \tilde p_4(t)\, (s^i)^2 + 2 \tilde p_5(t)\, \mathbf{1}_{N-1}^\top s^i \bs^{-i} \nonumber\\[1ex]
& + (\bs^{-i})^\top ((\tilde p_7(t) - \tilde p_6(t)) \bI_{N-1} + \tilde p_6(t)\, \bJ_{N-1}) \bs^{-i} \nonumber \\
 &+ \tilde r_1(t)\, q + \tilde r_2(t)\, s^i + \tilde r_3(t)\, \mathbf{1}_{N-1}^\top \bs^{-i} + \tilde{u}(t),
\end{align}
with terminal condition $\tilde V^i(T,q, s^i, \bs^{-i}) = c_4(s^i-\zeta(T))^2$, where $\bs^{-i} = [\ldots, s^{i-1},s^{i+1}, \ldots]^{\top} \in \bbR^{N-1}$. 

The functions $\{\tilde p_1, \ldots, \tilde p_7, \tilde r_1, \ldots, \tilde r_3, \tilde u\}$ solve a coupled ODE system
\begin{align}\label{homo_ode_system}
\left\{
\begin{aligned}
    -\dot{\tilde p}_1(t) &= -c_1 g_1(t)\, g_5(t) + c_2 g_1^2(t) - 2\kappa \tilde p_1(t) + \bh_1(t)^\top \bh_5(t), \\
 -\dot{\tilde p}_2(t) &= -\frac{1}{2}c_1 \big( g_1(t)g_6(t) + g_2(t) g_5(t) \big) + c_2\, g_1(t) g_2(t) - \kappa\tilde p_2(t) + \frac{1}{2}\bh_1(t)^\top \bh_6(t) + \frac{1}{2}\bh_2(t)^\top \bh_5(t), \\
  -\dot{\tilde p}_3(t) &= -\half c_1 \big( g_1(t) g_7(t) + g_3(t) g_5(t) \big) + c_2 g_1(t)g_3(t) - \kappa\tilde p_3(t)  + \tilde p_2(t)g_3(t) + \tilde p_3(t) g_2(t) \\
  & \quad + (N-2)\tilde p_3(t) g_3(t) + g_1(t)\tilde p_5(t) + g_1(t) \tilde p_7(t) + (N-2)g_1(t) \tilde p_6(t) + \tilde p_5(t)b(t) + \tilde p_7(t) b(t) \\
  & \quad + (N-2)\tilde p_6(t) b(t), \\
 -\dot{\tilde p}_4(t) &= -c_1 g_2(t)g_6(t) + c_2 g_2^2(t) + c_3 + \bh_2(t)^\top \bh_6(t) , \\
 -\dot{\tilde p}_5(t) &= -\half c_1 \big( g_2(t) g_7(t) + g_3(t) g_6(t) \big) + c_2 g_2(t)g_3(t) +  \tilde p_4(t) g_3(t) + \tilde p_5(t)g_2(t) + (N-2)\tilde p_5(t) g_3(t)  \\
&\quad +  g_2(t) \tilde p_5(t) + g_3(t) \tilde p_7(t) + (N-2) g_3(t)\tilde p_6(t),  \\
-\dot{\tilde p}_6(t) &= -c_1 g_3(t) g_7(t) + c_2 g_3^2(t) + 2\tilde p_5(t) g_3(t) + 2\big( \tilde p_6(t) g_2(t) + \tilde p_7(t) g_3(t) + (N-3) \tilde p_6(t) g_3(t) \big),  \\
-\dot{\tilde p}_7(t) &= -c_1 g_3(t) g_7(t) + c_2 g_3^2(t)+ 2\tilde p_5(t) g_3(t) + 2\big( \tilde p_7(t) g_2(t) + (N-2) \tilde p_6(t) g_3(t) \big), 
\end{aligned}
\right. \\
\left\{
\begin{aligned}\label{homo_ode_system_r}
 -\dot{\tilde r}_1(t) &= \bar{P}  g_1(t) - c_1 \big( g_4(t)  g_5(t) + g_1(t) g_8(t) \big) 
+ 2c_2 g_1(t) g_4(t)  + \big( 2\kappa \tilde p_1(t) \theta(t) - \tilde r_1(t) \kappa \big) \\
& \quad + \big( \bh_1(t)^\top \bh_8(t) + \bh_4(t)^\top \bh_5(t) \big),  \\
 -\dot{\tilde r}_2(t) &= \bar{P}  g_2(t) - c_1 \big( g_4(t) g_6(t) + g_2(t) g_8(t) \big) 
+ 2c_2 g_2(t) g_4(t) - 2c_3 \zeta(t)  +2\kappa\tilde p_2(t) \theta(t) \\
& \quad + \big( \bh_2(t)^\top \bh_8(t) + \bh_4(t)^\top \bh_6(t) \big), \\
-\dot{\tilde r}_3(t) &= \bar{P} g_3(t) 
- c_1 \big( g_4(t)  g_7(t) + g_3(t)  g_8(t) \big) 
+ 2c_2  g_3(t) g_4(t) 
+ 2\kappa  p_3(t) \theta(t)  +  \tilde r_2(t) g_3(t) + \tilde r_3(t)  g_2(t) \\
& \quad + (N-2) \tilde r_3(t)  g_3(t)  +\big( g_4(t) + a(t) \big) \big( 2\tilde p_5(t) + 2\tilde p_7(t) + 2(N-2) \tilde p_6(t) \big), \\
-\dot{\tilde{u}}(t) &= 
\bar{P} g_4(t) 
- c_1 g_4(t) g_8(t) 
+ c_2 g_4(t)^2 
+ c_3 (\zeta(t))^2  + \tilde r_1(t) \kappa \theta(t) 
+ \bh_4(t)^\top \bh_8(t) 
+ (\sigma^0)^2  \tilde p_1(t) \\
& \quad + 2\sigma^0\sigma\rho\big(2 \tilde p_2(t) + 2(N-1)\tilde p_3(t)\big) + \sigma^2\rho^2\big(2(N-1)\tilde p_5(t)+(N-1)(N-2)\tilde p_6(t)\big) \\
&\quad + \sigma^2(1-\rho^2) \big( \tilde p_4(t) + (N-1)\, \tilde p_7(t) \big),
\end{aligned}
\right.
\end{align}
with terminal conditions:
\begin{align} \label{homo_terminal}
\begin{aligned}
&\tilde p_1(T) = 0, \quad \tilde p_2(T) = 0, \quad \tilde p_3(T) = 0, \quad \tilde p_4(T) = c_4, \quad \tilde p_5(T) = 0, \quad \tilde p_6(T) = 0, \\
& \tilde p_7(T) = 0, \quad \tilde r_1(T) = 0, \quad \tilde r_2(T) = -2c_4\, \zeta(T), \quad \tilde r_3(T) = 0, \quad \tilde u(T) = c_4\, \zeta(T)^2.
\end{aligned}
\end{align}
The constant in \eqref{d,U,u} reduces to $d = c_1 + 2c_2$ and denoting $\eta_{0}=\frac{N\bar{P}}{c_1+2c_2}$ and $\eta_{1}=1 + \frac{Nc_1}{c_1+2c_2}$,
all other functions $\{g_1, \ldots, g_8, \bh_1, \bh_2, \bh_4, \bh_5, \bh_6, \bh_8 \in \bbR^N\} $ are defined as
\begin{align*}
    g_1(t) &:= \frac{c_1}{d\eta_{1}} + \frac{2Nc_1\, \tilde p_2(t)}{d^2\eta_{1}} - \frac{2\tilde p_2(t)}{d}, \qquad
    g_2(t) := \frac{2c_1\big( \tilde p_4(t) + (N-1)\tilde p_5(t) \big)}{d^2\eta_{1}}  - \frac{2\tilde p_4(t)}{d}, \\[1ex]
    g_3(t) &:= \frac{2c_1\big( \tilde p_4(t) + (N-1)\tilde p_5(t) \big)}{d^2\eta_{1}}  - \frac{2\tilde p_5(t)}{d},\quad\;
    g_4(t) := \frac{Nc_1\,\tilde r_2(t)}{d^2\eta_{1}} + \frac{c_1 \eta_{0}}{d\eta_{1}} - \frac{\tilde r_2(t) + \bar{P}}{d}, \\[1ex]
    g_5(t) &:= \frac{1 + \frac{2N\,\tilde p_2(t)}{d}}{\eta_{1}},\quad
    g_6(t) = g_7(t) := \frac{2\tilde p_4(t) + 2(N-1)\tilde p_5(t)}{d\eta_{1}},\quad
    g_8(t) := \frac{N\,\tilde r_2(t)}{d\eta_{1}} + \frac{\eta_{0}}{\eta_{1}}, \\[1ex]
    \bh_1(t) &:= \begin{bmatrix}
2\tilde p_2(t) \\
2\tilde p_3(t)\, \mathbf{1}_{N-1}
\end{bmatrix}, \quad
    \bh_2(t) := \begin{bmatrix}
2\tilde p_4(t) \\
2\tilde p_5(t)\, \mathbf{1}_{N-1}
\end{bmatrix}, \quad  
    \bh_4(t) := \begin{bmatrix}
\tilde r_2(t) \\
\tilde r_3(t)\, \mathbf{1}_{N-1}
\end{bmatrix}, \\
\bh_5(t) &:= (g_1(t)+b(t)) \mathbf{1}_{N}, \,
 \quad
\bh_6(t) := \begin{bmatrix}
    g_2(t) \\
    g_3(t)\, \mathbf{1}_{N-1}
\end{bmatrix}, \quad 
\bh_8(t) := (g_4(t) + a(t))\, \mathbf{1}_{N}.
\end{align*}
Consequently, the equilibrium BESS charging strategy of the $i$-th operator is 
\begin{align}\label{homo_optimal control} 
\hat{\alpha}^i(t, q, \bs) &=
g_1(t)q + g_2(t) s^i + g_3(t) \, \mathbf{1}_{N-1}^\top \bs^{-i} 
+ g_4(t).
\end{align}
\end{theorem}
The proof is deferred to Appendix \ref{General Theorem Proof}. 
This theorem shows that individual BESS controls are only a function of her own SOC $S^i$ and of the aggregate SOC of the others, $\sum_{j\neq i} S^j = \mathbf{1}_{N-1}^\top \bs^{-i}$. This structure arises from the fact that when all the cross-impact weights are the same, \eqref{d,U,u} implies that 
\begin{equation}\label{simp_M}
\bM=\bI_N-\frac{\diag \bd^{-1}\,\bc_1\,\mathbf{1}^\top}{1+\mathbf{1}^\top \diag \bd^{-1}\,\bc_1} = \bI_N- \frac{c_1}{2c_2+(N+1)c_1}\bJ_N ,
\end{equation}
which helps reducing equation \eqref{opt_alpha} to equation \eqref{homo_optimal control}. The following  two propositions are the counterpart of Proposition \ref{general_wp} and provide intervals where the above system of equations is guaranteed to be well-posed. See Appendix \ref{General Theorem Proof} for the detailed proof.

\begin{proposition}\label{homo_existence}
    The existence of solution for the homogeneous setting is determined by \\
    $\{\tilde p_4(t),\tilde p_5(t), \tilde p_6(t)\}$ which form a closed  Ricatti sub-system. Let
    \begin{equation}\label{gamma}
        \beta_{Hom} := \sqrt{\|\calM_1\|_2^2 + \|\calM_2\|_2^2 + \|\calM_3\|_2^2},
    \end{equation}
    where the matrices $\calM_1$, $\calM_2$, and $\calM_3$ are defined in \eqref{calM_matrix}-\eqref{calM} in the Appendix. Then \eqref{homo_ode_system}-\eqref{homo_ode_system_r} is well-posed on the time interval $[0,T]$ for any $T < \frac{1}{2\sqrt{\beta_{Hom} c_3}}\big[\pi - 2\tan^{-1}\big(\frac{ \sqrt{\beta_{Hom}} c_4}{\sqrt{c_3}}\big)\big]$.
\end{proposition}
\begin{proposition}\label{parameters_prop}
    If $N \ge 5$ and $c_1 >0$, $c_2 \geq 0$, $c_3 \geq 0$, $c_4 > 0$, then the homogeneous system \eqref{homo_ode_system}-\eqref{homo_ode_system_r} is well-posed on $[0,T]$ for any $T > 0$. 
\end{proposition}

\begin{remark}
    The proof of Proposition \ref{parameters_prop} relies on constructing a bounded region $\calI_N^{c_1,c_2,c_3,c_4}$ that contains the ODE trajectory. When $N < 5$, this construction does not apply and we have the weaker Proposition \ref{homo_existence}. The coefficients in the homogeneous Riccati system \eqref{homo_ode_system}-\eqref{homo_ode_system_r} can be interpreted according to the type of quadratic interaction they represent. Intuitively, the coefficient $\tilde p_4(t)$ is associated with the own-state quadratic term of the representative player, $\tilde p_5(t)$ measures the coupling between the representative player and the other players, and $\tilde p_6(t)$ measures the coupling among the other players. For small $N$, these coupling effects are relatively strong and are not sufficiently diluted across the population, which can produce a stronger feedback on the own-state coefficient $\tilde p_4(t)$.
\end{remark}

Under the homogeneous agents setting, the total number of differential equations to solve is $11$, independent of the number of agents $N$. In contrast, in the general case, the number of ODE equations is $\cO(N^3)$. Consequently, the homogeneity substantially lowers both computational and memory costs and allows to consider markets with arbitrary number of identical operators.

\section{Comparative Statics and Sensitivity Analysis}\label{sec:comparative}

\subsection{Model Illustration}\label{sec:illustration}

To illustrate our setup, we consider a toy case study with $N=8$ homogeneous BESS operators that are all hybrid-type. To construct the supply process $(Q_t)$ which is the primary determinant of market conditions, we roughly match to typical intra-day prices as observed in the California CAISO market circa 2026. CAISO has a very large share of renewable generators, primarily solar farms, and consistently experiences negative mid-day prices as well as  dominant evening price ramps. The corresponding $P_t$ exhibits a small peak in the early morning, a large trough in the afternoon and a second large peak in the evening. Taking supply elasticity $c_1=1$ and reverse-engineering leads to a sinusoidal average supply curve
\begin{align}\label{eq:sample-theta}
 \theta(t) = 30-14\sin \bigl(\frac{\pi}{12}t + \frac{16\pi}{24} \bigr)-7\sin \bigl(\frac{\pi}{6}t-\frac{7\pi}{24}\bigr). 
\end{align}
In \eqref{eq:sample-theta}, the implied daily TB (top-to-bottom) price range is $\$33.6$ with supply in the range of $[17.2,50.8]$, lowest average prices at 10:30am and peak at 6:30pm. For the solar self-generation we set
\begin{align}\label{eq:sample-mu}
\mu^i_t & = a^i(t) + b^i(t) Q_t, \qquad a^i(t) = 0.2 \sin(0.075\pi t - 2\pi/5) \vee 0 \\ 
b^i(t) &= 0.003(t-5 \frac{1}{3}) \cdot 1_{\{t \in[ 5 \frac{1}{3}, 8]\}} + 0.008 \cdot 1_{ \{ t \in [8,16] \}} - 0.003 (t-18 \frac{2}{3}) \cdot 1_{ \{t \in [16, 18 \frac{2}{3}] \}},
\end{align}
with correlation $\rho=0.6$ between idiosyncratic and grid-wide generation. This implies a maximum average self-generation of 0.6 GW at 11am and total daily self-generation of $\simeq 5$ GWh. There is no self-generation in the morning until 5:20am or at night after 6:40pm. For the BESS, we assume a total capacity of 10 GWh, with $\zeta(t) = 5, \;\forall t$.
The other parameters are listed in Table \ref{tab:parameters}. By Proposition \ref{parameters_prop}, the toy case setup admits a well-defined solution on the time horizon considered.

\begin{table}[!ht]
\caption{Model parameters used in numerical experiments of Section \ref{sec:comparative}. }
\label{tab:parameters}
\renewcommand{\arraystretch}{1.15}
\setlength{\tabcolsep}{3pt}
\centering
\begin{tabular}{p{4.3cm}p{3.6cm}|p{3.8cm}p{2.7cm}}
\hline
\textbf{Parameter} & \textbf{Value}
& \textbf{Parameter} &  \textbf{Value} \\ \hline
 Base price (\$) &  $\bar{P}= 50$
& Price sensitivity & $c_1 = 1$ \\[0.4ex]
 Control penalty& $c_2 = 0.1$
&Running SOC penalty  & $c_3 = 0.25$ \\[0.4ex]
 Terminal SOC penalty& $c_4 = 100$
& Horizon (hrs) & $T =24$ \\[0.4ex]
Dispatch target (GWh) & $\zeta(t) = 5, \;\forall  t \in [0,T]$
&Mean reversion rate  & $\kappa = 5$ \\[0.4ex]
 Correlation w/$Q$  & $\rho = 0.6$
& Supply volatility  & $\sigma^0 = 5$ \\[0.4ex]
Self-generation volatility &  $\sigma = 0.5$
& Average supply (GW)  & $\theta(t)$, see \eqref{eq:sample-theta} \\[0.4ex]
 Initial supply (GW) & $Q_0 = \theta (0)=23.43$
&Initial SOC (GWh) & $S_0=  5$ \\[0.4ex]
\hline
\end{tabular}

\end{table}

The left panel of Figure \ref{fig:baseline-state-intervals} visualizes one simulation of the realized market. The top row shows the individual controls $\alpha^i_t$ and SOC $S^i_t$ which evolve slightly differently due to the idiosyncratic shocks $\sigma^i dW^i_t$. The bottom-left panel shows the exogenous supply $Q_t$ driven by $\sigma^0 dW^0_t$ and the corresponding hybrid generation $\sum_i \mu^i_t$. Finally, the bottom-right panel shows the resulting power price $\hat P_t = \bar{P} - c_1( Q_t - \sum_i \hat\alpha^i_t)$. 

The right panel of Figure \ref{fig:baseline-state-intervals} shows the distribution of the same quantities (evaluated via $M=1000$ simulations). Since the agents are homogeneous, the distributions of $\hat\alpha^i_t, S^i_t$ are identical for $i=1,\ldots,N$. In our toy setup, net supply is highest during the midday and lowest around 6pm.
The BESS control countervails this pattern in order to buy low and sell high. As a result, the BESS charges in the first half of the day and then discharges in the afternoon and (less so) in the evening. The resulting prices are lowest in the late morning. Notably, the BESSs flatten the evening price peak, making $\E[\hat P_t]$ almost constant for $t\in[18,24]$. In particular, TB falls from \$33 without BESS to less than \$18 with the 8 operators, a reduction of 46.6\%.

Two other effects are notable in the behavior of $\hat\alpha^i_t$. First, due to the local generation, the BESS operation combines two objectives: selling its generated power $\mu^i_t$ (with the BESS acting as a ``pass-through'' for that additional supply), as well as engaging in energy arbitrage to charge when $\hat P_t$ is low and discharge when $\hat P_t$ is high. In Figure \ref{fig:baseline-state-intervals} this manifests itself in $\hat\alpha^i_t$ being mostly negative, as the BESS needs to sell all the (daytime-generated) local $\mu^i_t$. In order to clear out SOC capacity, the BESS discharges heavily in the morning (even though prices are not very high), so that it has headroom to absorb ``excess'' local supply mid-day. This effect is driven by the SOC constraints (parameter $c_3$). Second, the terminal penalty that drives $S^i_T \simeq 5$ leads to high variance in $\alpha^i_t$ in the last two hours of the day, as the BESS ``rushes'' to match the postulated terminal SOC and must adjust its strategy depending on how the realized trajectory of $\mu^i_t$ unfolded. Otherwise, the variance of $\hat\alpha^i_t$ is roughly constant throughout the day, consistent with the constant volatility of $\mu^i_t$ and $\{Q_t\}$.

\begin{figure}[!htb]
    \centering
    a)\includegraphics[width=0.47\textwidth]{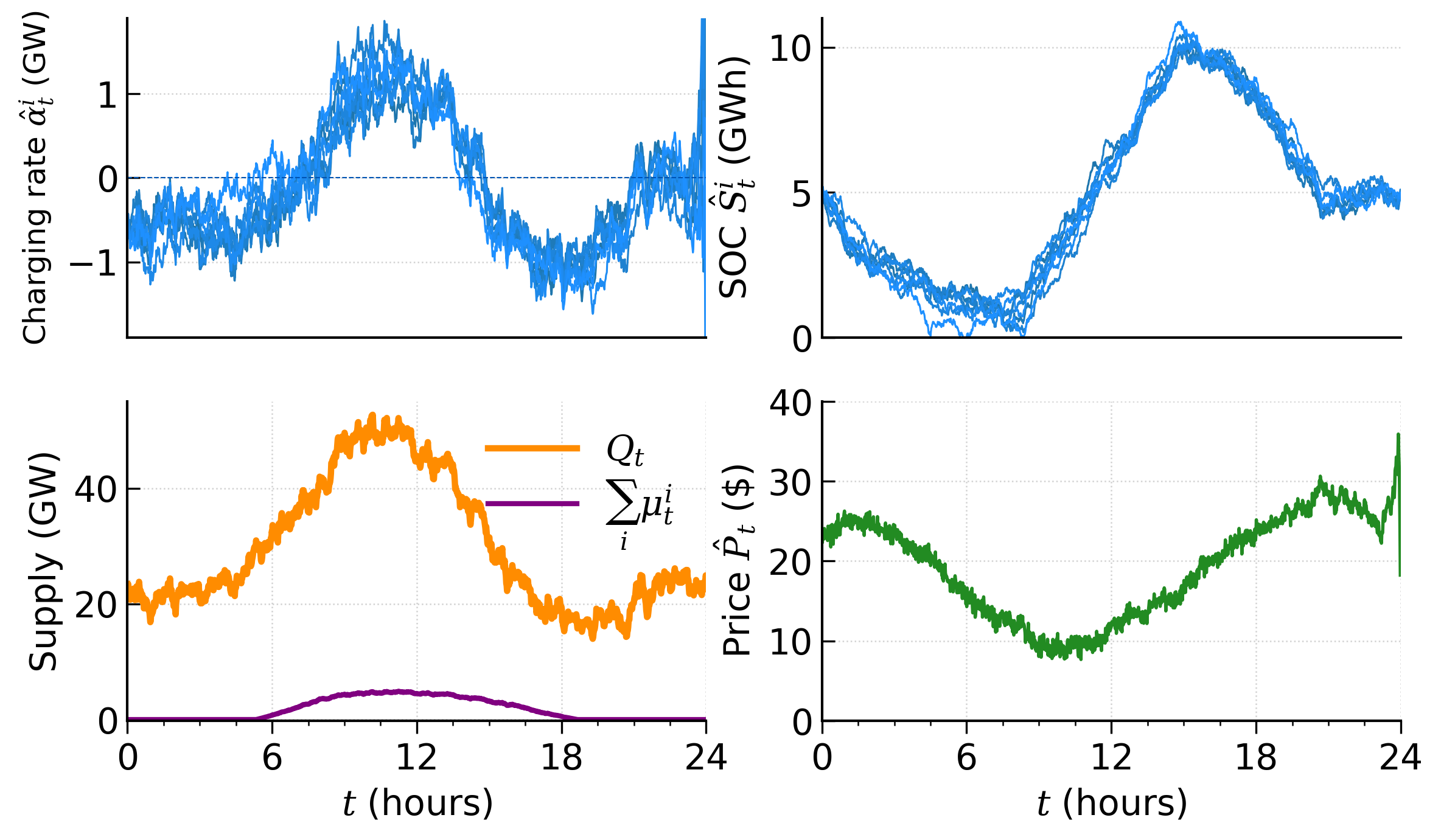}
    b)\includegraphics[width=0.47\textwidth]{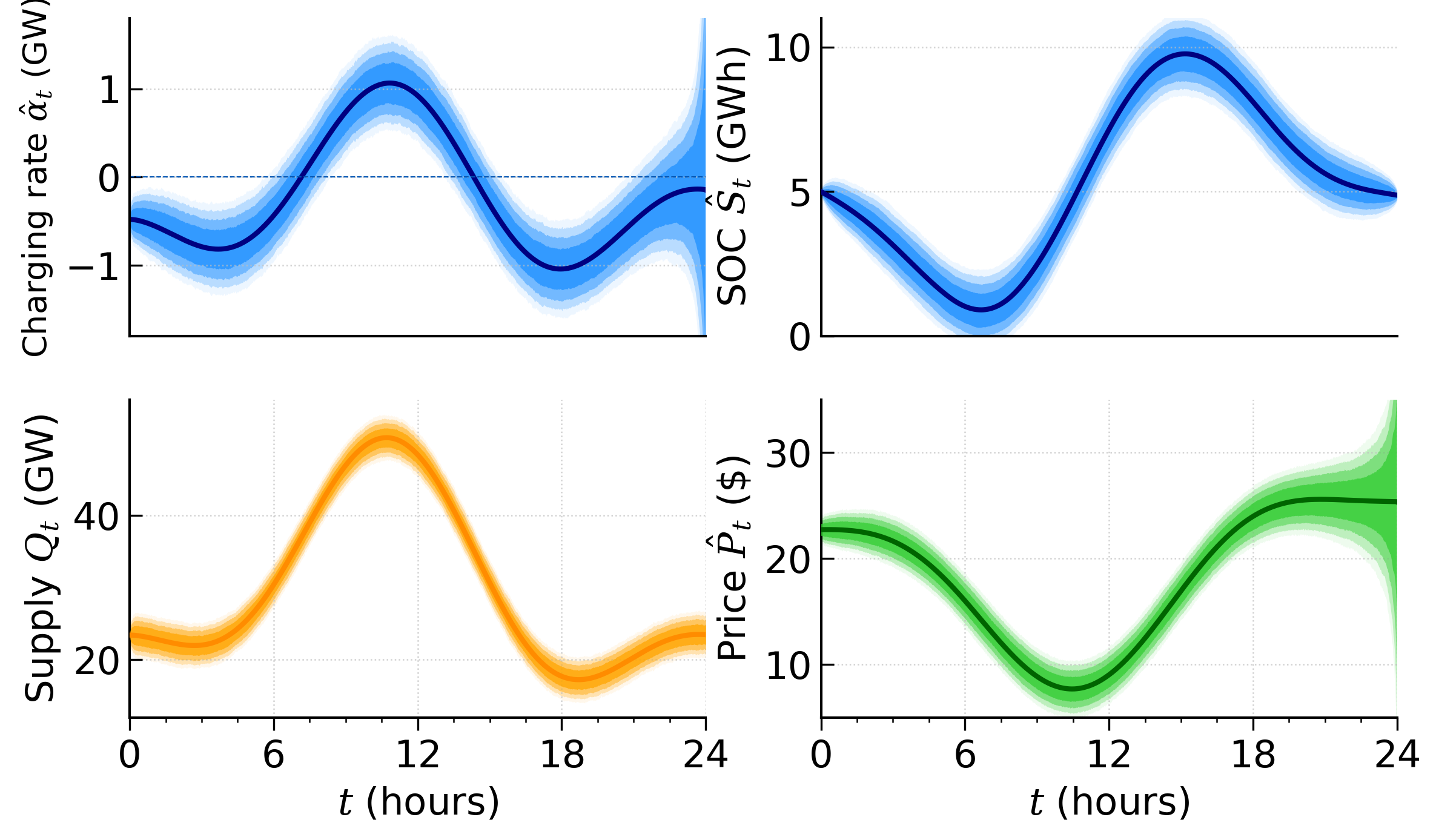} 
    \caption{ A market with $N=8$ homogeneous BESS operators, with parameters in Table \ref{tab:parameters}. \emph{Left:} illustrative scenario showing equilibrium $(\hat\alpha^i_t,\hat S^i_t)$ for each $i=1,\ldots,N$, as well as $Q_t, \sum_i (a^i(t)+b^i(t)Q_t)$ and $\hat{P}_t \equiv \hat{P}^i_t$, $\forall i$. \emph{Right:} distribution of same quantities across 1000 paths. The shaded bands denote the 60\%, 80\%, 90\% and 95\%-quantile regions.}
    \label{fig:baseline-state-intervals}
\end{figure}

\subsection{Comparative Statics}
To tease out the role and the impact of different parameters in our model, in this section, we present a series of comparative statics, whereby we vary one or two variables, keeping the others fixed. To summarize the behavior of the agents (assumed homogeneous, hence with identical functional forms of their controls $\alpha^i(t,q,\bs)$), we record the following summary statistics. First, we track the average total charge/discharge over the day,
\begin{align}\label{eq:TC}
TC = \E \Bigl[ \int_0^{T^\circ} |\hat\alpha^i_t| dt \Bigr].
\end{align}
Above $T^\circ = 21$ is a cut-off hour, in order for the metric not to get skewed by the (occasionally large) charging/discharging driven by the terminal condition, note the wide ``fan'' in the right panel of Figure \ref{fig:baseline-state-intervals}.

Second, we track the total storage utilization:
\begin{align}\label{eq:TS}
TS = \E \Bigl[ \max_{t \le T^\circ} \hat S_t - \min_{t \le T^\circ} \hat S_t \Bigr].
\end{align}
Third, we track the daily average price spread, known in the industry as top-to-bottom:
\begin{align}\label{eq:tb}
TB = \E \Bigl[ \max_{t \le T^\circ} \hat P_t - \min_{t \le T^\circ} \hat P_t \Bigr].
\end{align}

\noindent\textbf{Capacity constraint $c_2$ and energy constraint $c_3$.}
As we have introduced in  Section \ref{sec:objective},  the two running cost terms are governed by the capacity constraint parameter $c_2$, which discourages aggressive charging/discharging, and the energy constraint parameter $c_3$, which prevents the battery storage to be over empty or over full. These two parameters jointly affect the best response of the agent by limiting the charge and discharge rate as well as monitoring the battery state to not be far from the dispatch level. And therefore we want to examine the effect of the parameters separately by considering two scenarios: when $c_2 \gg c_3$, by setting $c_3$ extremely small and $c_2$ extremely big, we disregard the energy constraint and focus only on the capacity constraint; and when $c_3 \gg c_2$, we impose little limit on the charge/discharge rate while only governing the battery storage. Figure \ref{fig:c2_c3} illustrates the effect of $c_2$ and $c_3$ alone on the choice of the agents' action as well as the process of the battery storage.

\begin{figure}[!htb]
    \centering
        \includegraphics[width=0.32\linewidth]{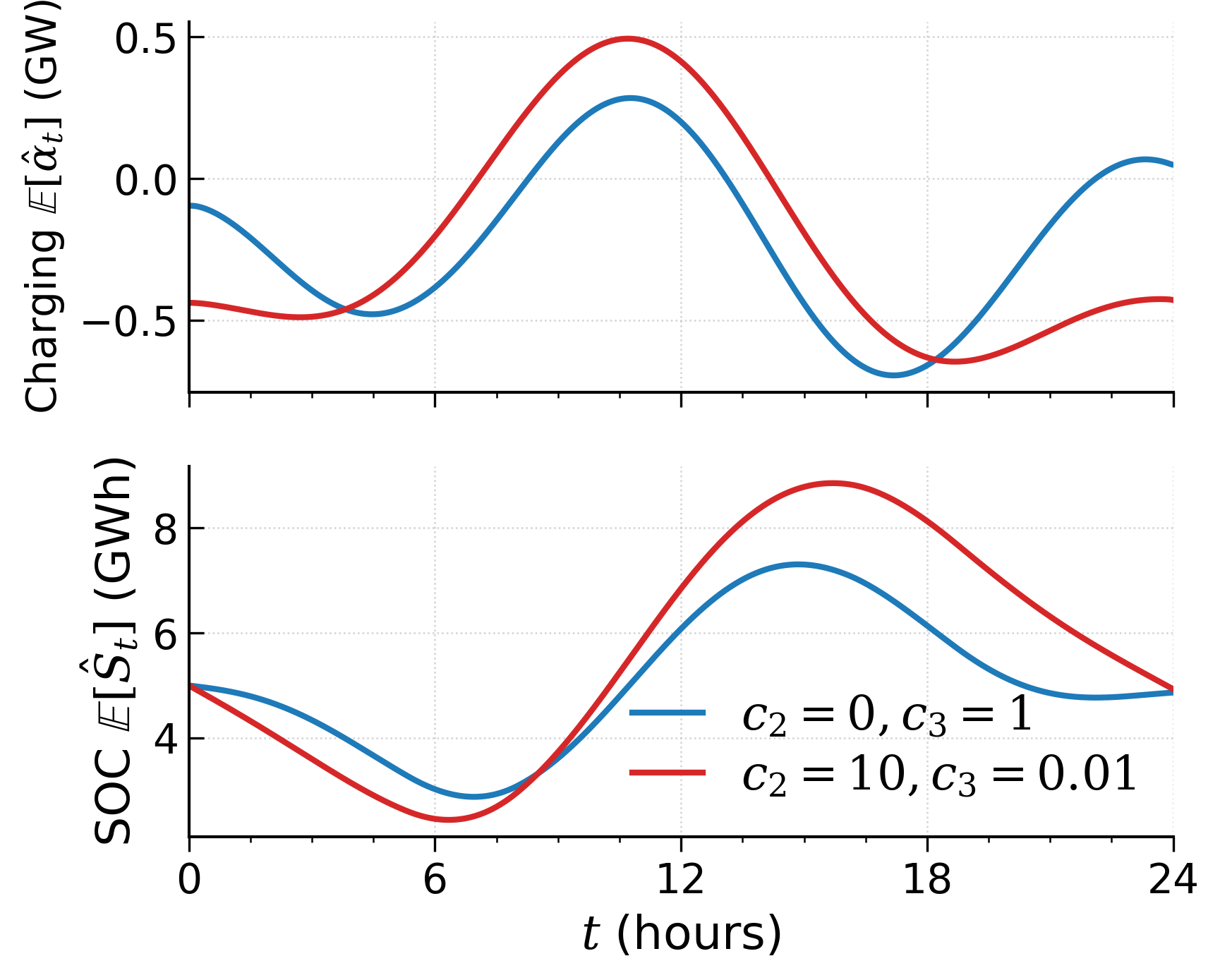}
        \includegraphics[width=0.32\linewidth]{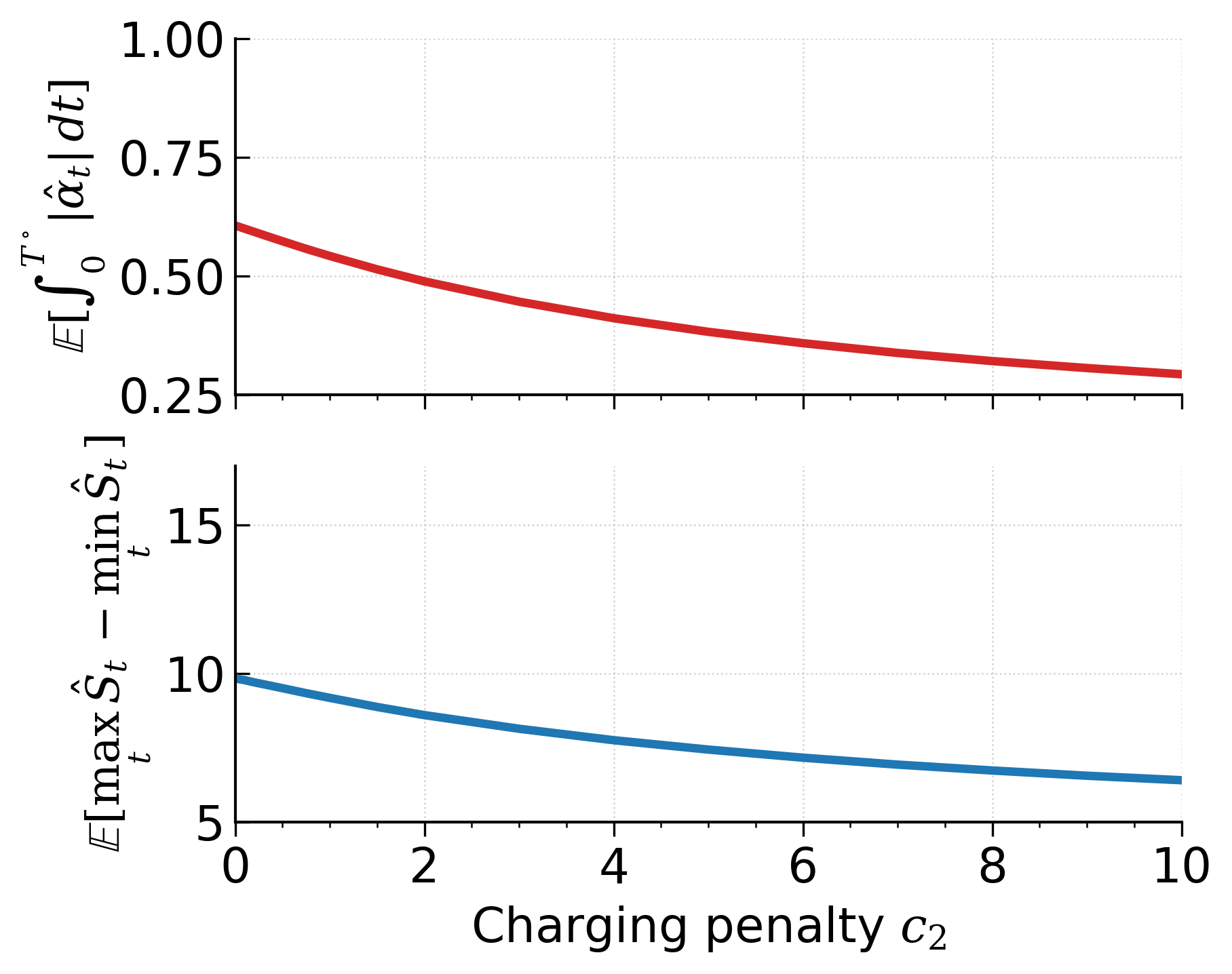}
        \includegraphics[width=0.32\linewidth]{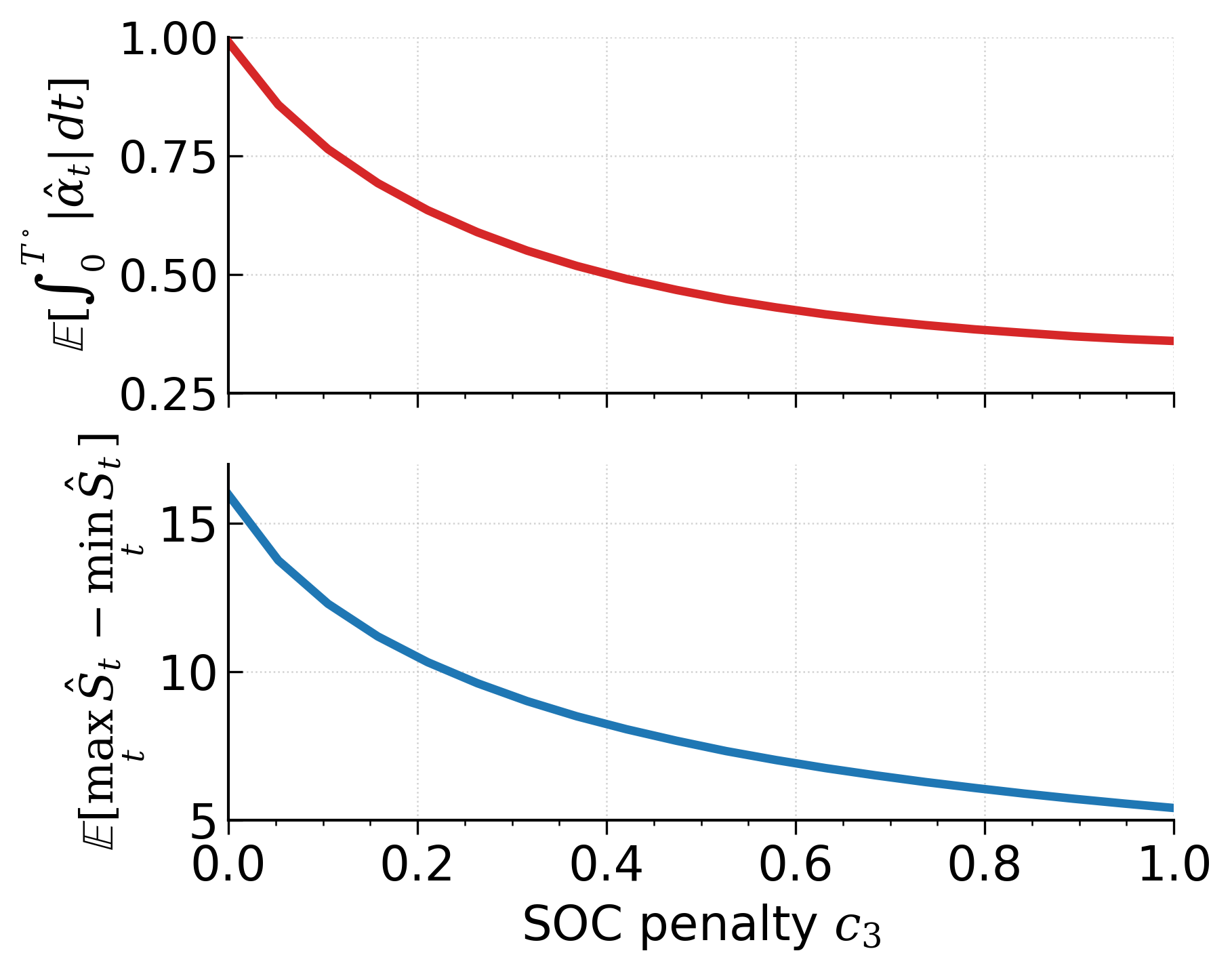}
  \caption{\emph{Left panel:}  $\E [\hat{\alpha}_t]$ and  $\E[\hat S_t]$. \emph{Middle and Right:} Effect of $c_2$ and $c_3$ on  average total charging and discharging TC \eqref{eq:TC} and average range of SOC TS \eqref{eq:TS}, holding $c_3 = 0$ and $c_2 = 0$ respectively.}
  \label{fig:c2_c3}
\end{figure}

\noindent\textbf{Price sensitivity $c_1$.}
The price impact $c_1$ affects the sensitivity of the price to the total supply. Higher $c_1$ strengthens both the feedback effect of the individual $\alpha^i_t$ (which contributes a quadratic cost $c_1 (\alpha^i_t)^2$ just like $c_2 (\alpha^i_t)^2$) as well as the impact of others' charging $c_1 \sum_{j \neq i} \alpha_j \alpha_i$. These two effects lead to ambiguous impact on the total charging.

\begin{figure}[!htb]
    \centering
     \includegraphics[width=0.49\textwidth]{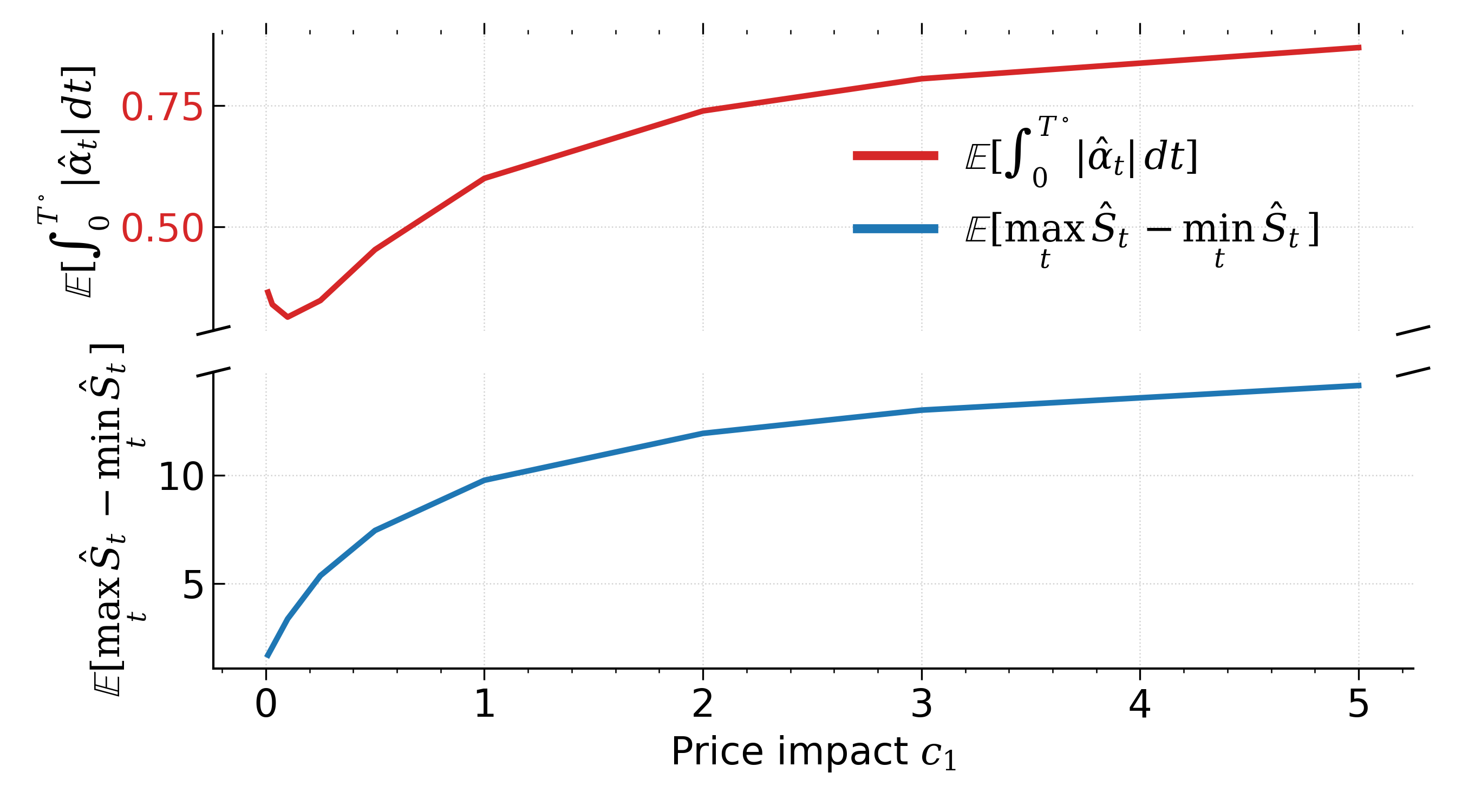}
    \includegraphics[width=0.49\textwidth]{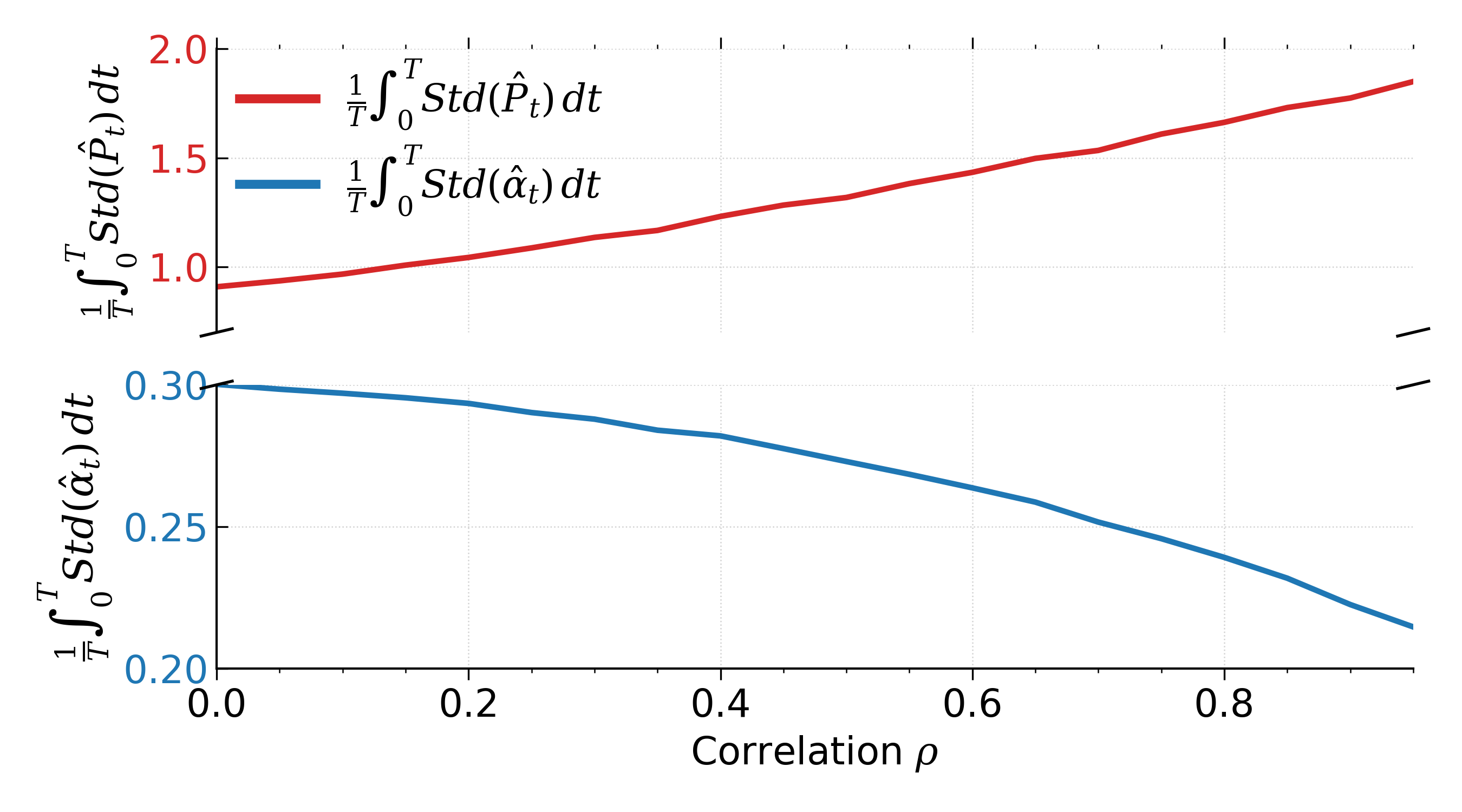}
    \caption{ \emph{Left panel:} effect of price impact $c_1$ on total charge/discharge TC \eqref{eq:TC} and on Range of SOC. \emph{Right:} Effect of $\rho$ on average $\Std(\hat\alpha_t^i)$ and $\Std(\hat P_t^i)$. We use the base setting of Section \ref{sec:illustration} with $N=8$ homogeneous agents. }
    \label{fig:c1_rho}
\end{figure}

\smallskip
\noindent\textbf{Correlation $\rho$ and SOC volatility $\sigma$.}\label{rho_effect}
To analyze the effect of the correlation parameter $\rho$ and the BESS volatility $\sigma$, we evaluate the mean and variance of the resulting equilibrium quantities.

\begin{proposition} \label{prop:homo_mean_control}
 Let $\{g_1(t),g_2(t),g_3(t),g_4(t)\}$ be the functions defined in Theorem~\ref{Homo_case} and set
 $$\tilde{g}(t) := g_2(t) + (N-1)g_3(t).$$
 Under the homogeneous-agent setting, the exogenous production $Q_t$, the equilibrium SOC $\hat S_t^i$, the equilibrium control $\hat \alpha_t^i$, and the equilibrium price $\hat P_t^i$ for agent $i$ have the following expectations:
 \begin{align}
 \E[Q_t] & = e^{-\kappa t}Q_0 + \kappa\int_0^t e^{-\kappa (t-s)}\theta(s)\,ds; \label{mean_Q} \\
  \E[\hat S_t^i] &  :=m_S(t) = e^{\int_0^t \tilde{g}(u)\,du}
\Big(
S_0  + \int_0^t
e^{-\int_0^s \tilde{g}(u)\,du}
\big((g_1(s)+b(s))\E[Q_s]+g_4(s)+a(s)\big)\,ds
\Big);\label{mean_S} \\
     \E[\hat\alpha_t^i] & = g_1(t)\E[Q_t] + \tilde{g}(t) m_S(t)+g_4(t);\label{mean_control} \\
     \E[\hat P_t^i] &= \bar{P} + c_1\big((Ng_1(t)-1)\E[Q_t] + N \tilde{g}(t)m_S(t) + Ng_4(t)\big).\label{mean_price}
 \end{align}
Due to homogeneity, the above expressions do not depend on the index $i$.

Denote by $\bar S_t = \frac{1}{N}\sum_{i=1}^N \hat S_t^i$ the average equilibrium SOC process, and 
define $C_{Q\bar S}(t) := \Cov(Q_t,\bar S_t)$,  and $V_{\bar S}(t) := \Var(\bar S_t)$. Then $\Var(Q_t) = V_Q(t) = (\sigma^0)^2\big(1 - e^{-2\kappa t}\big)/(2\kappa)$, and $C_{Q\bar S}(t)$ and $V_{\bar S}(t)$ satisfy 
 \begin{align}
     C_{Q\bar S}(t) &= \int_0^t e^{\int_s^t (\tilde{g}(u)-\kappa)\,du}\big((b(s)+g_1(s))\frac{(\sigma^0)^2}{2\kappa}(1-e^{-2\kappa s}) + \rho\sigma\sigma^0\big)\,ds,  \label{cov_Q_bar_S}\\ 
     V_{\bar S}(t) &= \int_0^t e^{2\int_s^t  \tilde{g}(u)\,du}\big(2(b(s) + g_1(s))C_{Q\bar S}(s) + \rho^2\sigma^2 + \frac{1-\rho^2}{N}\sigma^2\big)\,ds. \label{Var_bar_S}
 \end{align}
 Furthermore, the variances of the equilibrium control and prices, denoted by $\Var(\hat \alpha_t^i)$ and $\Var(\hat P_t^i)$, satisfy
 \begin{align}
     \Var(\hat \alpha_t^i) &= g_1^2(t)V_Q(t) + \tilde{g}(t)^2 V_{\bar S}(t) + 2g_1(t)\tilde{g}(t)C_{Q\bar S}(t) \nonumber \\
     &\quad + (g_2(t)-g_3(t))^2\sigma^2(1-\rho^2)\Big(1-\frac{1}{N}\Big)\int_0^t e^{2\int_s^t (g_2(u)-g_3(u))\,du}\,ds; \label{var_control}\\
     \Var(\hat P_t^i) &= c_1^2(1-Ng_1(t))^2V_Q(t) + N^2c_1^2 \tilde{g}(t)^2 V_{\bar S}(t) - 2Nc_1^2(1-Ng_1(t))\tilde{g}(t)C_{Q\bar S}(t).\label{var_price}
 \end{align}
\end{proposition}
The proof is deferred to Appendix \ref{appen:comparative}.
Proposition \ref{prop:homo_mean_control} shows that the mean equilibrium control $\E[ \hat \alpha^i_t] $ and mean equilibrium SOC $\E[ \hat S^i_t]$ do not depend on either diffusion coefficient $\sigma$ or the correlation parameter $\rho$ between $W^i$ and $W^0$. 
This is because the homogeneous ODE system \eqref{homo_ode_system} (except for $\tilde u$), and hence the coefficients $g_i$, are free of $\sigma$ or $\rho$. 

By contrast, the nonnegative parameters $\rho$ and $\sigma$ do affect the respective (co)variances. As $\rho$ increases, the stochasticity in SOC dynamics become more aligned with the one in the exogenous supply $Q_t$ through their shared exposure to the common shock $W^0$, this raises the covariance $C_{Q\bar S}(t)$. The variance $V_{\bar S}(t)$ depends on $\rho$ both indirectly through $C_{Q\bar S}(t)$ and directly through a $\rho^2$-term, so its dependence on $\rho$ is generally nonlinear and not necessarily monotonic, unless $b(t) + g_1(t) \geq 0$ then $V_{\bar S}(t)$ is increasing in $\rho$. Increasing $\rho$ strengthens the common-noise contribution in $\Var(\hat \alpha_t^i)$ through $V_{\bar S}(t)$, but simultaneously reduces the idiosyncratic term proportional to $1-\rho^2$. Since the sign of the coefficient in from of $C_{Q\bar S}(t)$ is not determined,  the overall effect is ambiguous. For the price variance, the dependence on $\rho$ comes entirely through $C_{Q\bar S}(t)$ and $V_{\bar S}(t)$, but the overall effect is again ambiguous. 
Figure~\ref{fig:c1_rho} illustrates the effect of $\rho$ on the electricity price $\hat P_t^i$ and the control $\hat \alpha^i_t$ under the current parameter choice; see Table~\ref{tab:parameters}. Increasing $\rho$ shifts randomness from the idiosyncratic component toward the common component, and in the present parameter regime, this reduces $\Std(\hat \alpha^i_t)$ but increases $\Std(\hat P_t^i)$. 

\subsection{Supply and Competition}\label{sec:hybrid-arb}

The number of agents $N$ is a key variable in our model. Within the context of hybrid-BESS operators, $N$ carries a dual impact. First, larger $N$ creates more competition, amplifying the ``cannibalization'' of available price spreads that generate energy arbitrage profits. Second, larger $N$ increases the total supply through the regional generation $\mu^i_t$. To disentangle these two effects, in this section we set up distinct experiments that quantify the impact of additional supply and the impact of additional competition. First, we fix the number of agents $N$, and vary the fraction of them that have  self generation ($\mu^i_t >0)$. Second, we fix the total supply and vary the number of ``arbitrageurs'', i.e., agents that engage in pure energy arbitrage and have $\mu^i_t \equiv 0$.

\medskip
\noindent\textbf{Supply Effect.}
For our first experiment, we fix $N$ and vary the total regional supply. To do so, we let $n$ agents have self-generation $\mu^i_t$ and $N-n$ agents be pure arbitrageurs as shown in the top diagram of Figure \ref{fig:supply_demo}. We then study the controls and price behavior as a function of $n$. Note that in this setting we have 2 classes of  agents---Arbitrageurs and Hybrid---that are homogeneous within the class. While we use the general-case solution to compute $(\hat \alpha^i_t, \hat S^i_t)$, it is also possible to derive a system of equations for the case where  $N$ agents fall into $k$ homogeneous subgroups. In the latter case, there are  $\big(3k+5+\frac{(k+2)(k+1)}{2}\big)k$ to solve, which leads to $34$ variables to solve for with 2 classes.

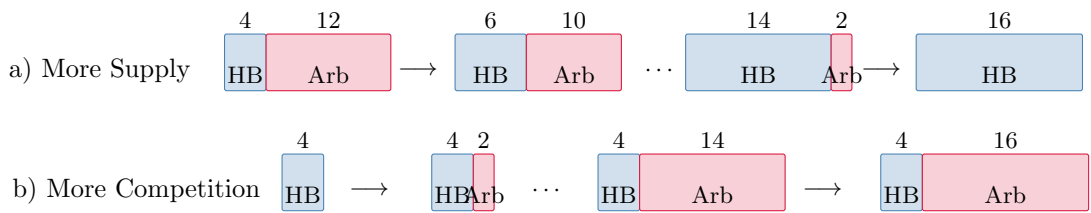
\begin{figure}[!htp]
\centering
\begin{tikzpicture}[
  x=0.55cm,
  y=0.75cm,
  every node/.style={font=\footnotesize},
  box/.style={draw, rounded corners=0.8pt, line width=0.4pt}
]

\node[font=\small] at (-3,0.36){a) More Supply};

\def\H{1}        
\def\U{0.25}        
\def\G{0.05}        
\def\Ylab{1.25}     
\pgfmathsetmacro{\PanelW}{12*\U+\G}
\pgfmathsetmacro{\W}{12*\U}  

\filldraw[fill=hybridcolor!25, draw=hybridcolor, box]
  (0,0) rectangle ({(4/16)*16*\U},\H);
\filldraw[fill=arbicolor!20, draw=arbicolor, box]
  ({(4/16)*16*\U},0) rectangle ({16*\U},\H);

\node at ({((4/16)*16*\U)/2},0.28) {HB};
\node at ({(4/16)*16*\U + ((12/16)*16*\U)/2},0.28) {Arb};

\node at ({((4/16)*16*\U)/2},\Ylab) {4};
\node at ({(4/16)*16*\U + ((12/16)*16*\U)/2},\Ylab) {12};

\node at ({\W + 1.65},0.36) {$\longrightarrow$};

\begin{scope}[xshift={\PanelW cm}]
\filldraw[fill=hybridcolor!25, draw=hybridcolor, box]
  (0,0) rectangle ({(6/14)*16*\U},\H);
\filldraw[fill=arbicolor!20, draw=arbicolor, box]
  ({(6/14)*16*\U},0) rectangle ({16*\U},\H);

\node at ({((6/14)*16*\U)/2},0.28) {HB};
\node at ({(6/14)*16*\U + ((8/14)*16*\U)/2},0.28) {Arb};

\node at ({((6/14)*16*\U)/2},\Ylab) {6};
\node at ({(6/14)*16*\U + ((8/14)*16*\U)/2},\Ylab) {10};

\end{scope}
\node at ({3*\W + 1.5},{0.36}) {$\cdots$};

\begin{scope}[xshift={2*\PanelW cm}]
\filldraw[fill=hybridcolor!25, draw=hybridcolor, box]
  (0,0) rectangle ({(14/16)*16*\U},\H);
\filldraw[fill=arbicolor!20, draw=arbicolor, box]
  ({(14/16)*16*\U},0) rectangle ({16*\U},\H);

\node at ({((14/16)*16*\U)/2},0.28) {HB};
\node at ({(14/16)*16*\U + ((2/16)*16*\U)/2},0.28) {Arb};

\node at ({((14/16)*16*\U)/2},\Ylab) {14};
\node at ({(14/16)*16*\U + ((2/16)*16*\U)/2},\Ylab) {2};
\end{scope}
\node at ({5*\W + 0.8},{0.36}) {$\longrightarrow$};

\begin{scope}[xshift={3*\PanelW cm}]
\filldraw[fill=hybridcolor!25, draw=hybridcolor, box]
  (0,0) rectangle ({16*\U},\H);

\node at ({8*\U},0.28) {HB};
\node at ({8*\U},\Ylab) {16};
\end{scope}

\end{tikzpicture}
\vspace*{10pt}

\begin{tikzpicture}[
  x=0.55cm,
  y=0.75cm,
  every node/.style={font=\footnotesize},
  box/.style={draw, rounded corners=0.8pt, line width=0.4pt}
]

\node[font=\small] at (-3.5,0.36){b) More Competition};

\def\H{1}
\def\U{0.25}
\def\Ylab{1.25}

\filldraw[fill=hybridcolor!25, draw=hybridcolor, box]
(0,0) rectangle ({4*\U},\H);

\node at ({2*\U},0.28) {HB};
\node at ({2*\U},\Ylab) {4};

\node at (2.1,0.36) {$\longrightarrow$};

\begin{scope}[shift={(3.6,0)}]

\filldraw[fill=hybridcolor!25, draw=hybridcolor, box]
(0,0) rectangle ({4*\U},\H);

\filldraw[fill=arbicolor!20, draw=arbicolor, box]
({4*\U},0) rectangle ({6*\U},\H);

\node at ({2*\U},0.28) {HB};
\node at ({5*\U},0.28) {Arb};

\node at ({2*\U},\Ylab) {4};
\node at ({5*\U},\Ylab) {2};

\end{scope}

\node at (6.4,0.36) {$\cdots$};

\begin{scope}[shift={(7.6,0)}]

\filldraw[fill=hybridcolor!25, draw=hybridcolor, box]
(0,0) rectangle ({4*\U},\H);

\filldraw[fill=arbicolor!20, draw=arbicolor, box]
({4*\U},0) rectangle ({18*\U},\H);

\node at ({2*\U},0.28) {HB};
\node at ({11*\U},0.28) {Arb};

\node at ({2*\U},\Ylab) {4};
\node at ({11*\U},\Ylab) {14};

\end{scope}

\node at (13,0.36) {$\longrightarrow$};

\begin{scope}[shift={(14.4,0)}]

\filldraw[fill=hybridcolor!25, draw=hybridcolor, box]
(0,0) rectangle ({4*\U},\H);

\filldraw[fill=arbicolor!20, draw=arbicolor, box]
({4*\U},0) rectangle ({20*\U},\H);

\node at ({2*\U},0.28) {HB};
\node at ({12*\U},0.28) {Arb};

\node at ({2*\U},\Ylab) {4};
\node at ({12*\U},\Ylab) {16};

\end{scope}

\end{tikzpicture}

\caption{Market composition experiments. Top row:  supply effect conveyed through increasing the number of hybrid BESS (HB) participants, while keeping total number $N=16$ of BESS fixed. Bottom row: competition effect accessed through fixing the number of hybrid BESS HB = 4 and varying the number $n$ of Arbitrageur (Arb) operators.}
\label{fig:supply_demo}
\end{figure}

The left panel of Figure \ref{fig:supply-effect} shows the average controls $\E[\hat{\alpha}^{(\mathrm{Type})}_t]$ of the Hybrid and Arbitrageur operators as a function of $t$, as $n$ varies. Because the Hybrid operators have their own regional generation, the amplitude  $|\E[\hat{\alpha}^{\mathrm{Hyb}}_t]|$ is lower than the discharge  $|\E[\hat{\alpha}^{\mathrm{Arb}}_t]|$ of the Arbitrageurs. Moreover, the peak charge/discharge hour is later for the Arbitrageurs. The right panel of Figure \ref{fig:supply-effect} shows that as supply rises (i.e., $n\uparrow$), prices fall (almost) linearly, and the timing of the afternoon price trough shifts a bit later.
\begin{figure}[!htbp]
        \centering
\begin{tabular}{cc}
    \includegraphics[width=0.49\linewidth]{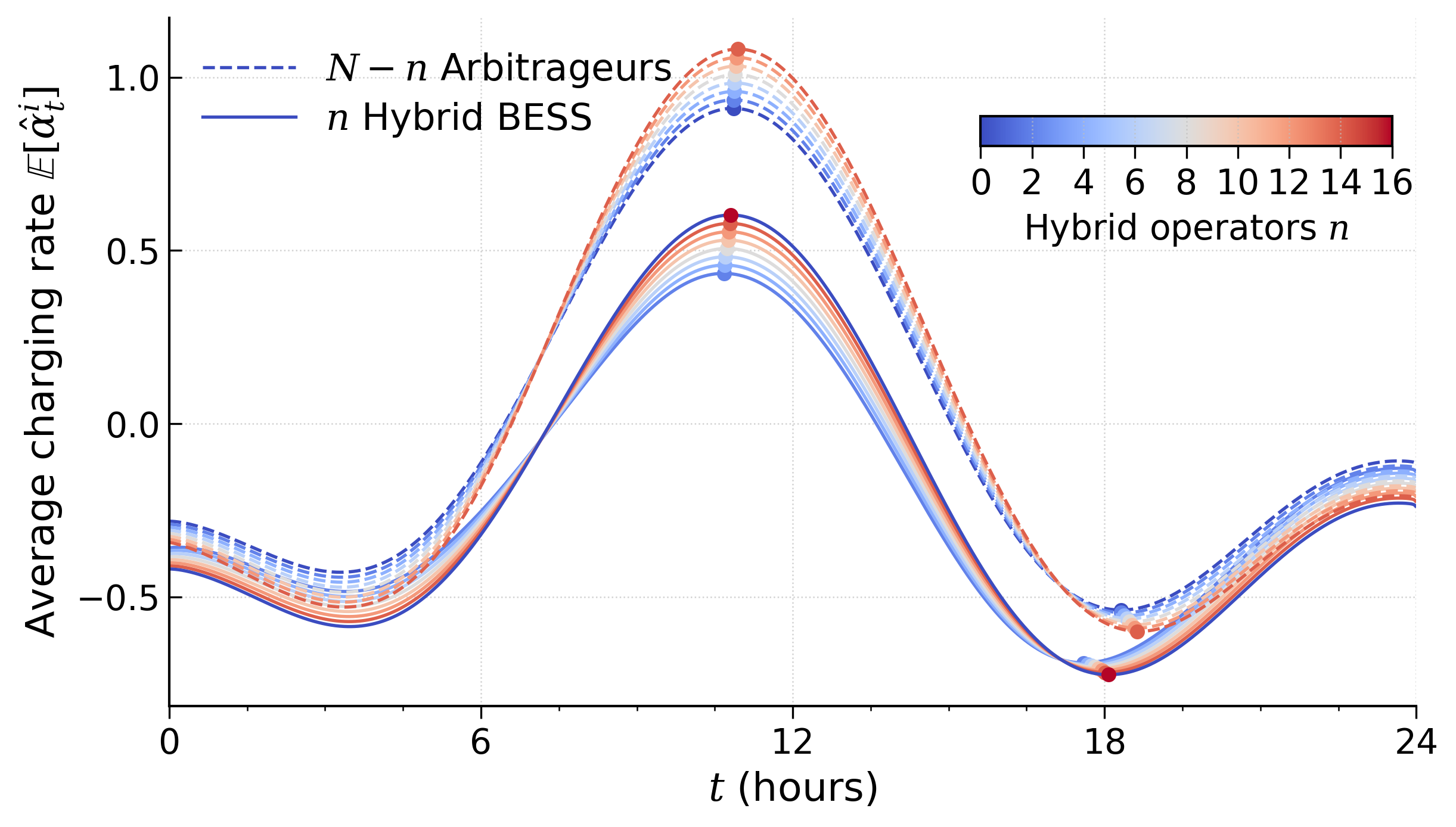}  &
        \includegraphics[width=0.47\linewidth]{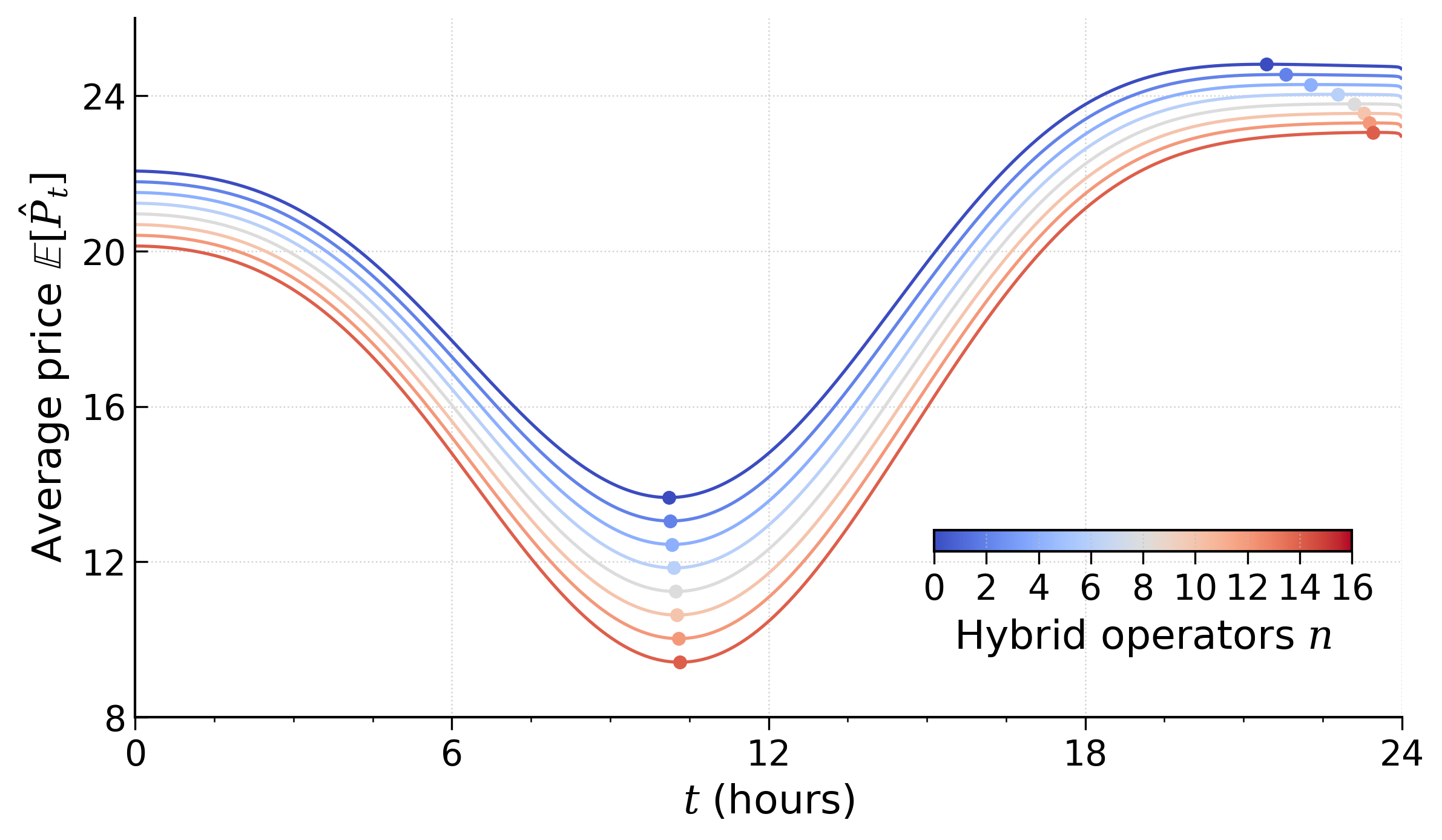}
    \end{tabular}
    \caption{Effect of added hybrid supply. \emph{Left panel}: average control of operators  $\E\big[\hat{\alpha}^{(\mathrm{Type})}_t\big], \mathrm{Type} \in \{\mathrm{Arb}, \mathrm{Hyb}\}$, with fixed $N=16$ and Hybrid assets color coded according to $n=0,2,\ldots,16$. \emph{Right panel:} average equilibrium price $\E[\hat P_t]$ as a function of $t$, varying $n$. The dots indicate the time of the daily peak/trough of each curve.}\label{fig:supply-effect}
\end{figure}

\noindent\textbf{Competition Effect.} We consider $n=4$ Hybrid operators and 
$N \ge 4$ total agents, i.e.~$N-n$ Arbitrageurs, varying $N$. Figure \ref{fig:competition-effect} shows three different competitive effects. First, more agents flatten price fluctuations, driving down the peaks and raising the valleys. In particular, with sufficiently many agents, the afternoon peak is entirely erased. Moreover, the peak and the trough of the typical price curve get further apart.  This confirms the qualitative conclusions in \cite{butters2025soaking} who point out that ``the marginal effects diminish as the battery fleet increases in size''.  Second, as more and more arbitrageurs appear in the market, the Hybrid BESS operators absorb their regional production more and more. Indeed, the gap between the SOC of Hybrid and Arbitrageur operators increases in $N$, so Hybrid operators become more focused on their self-generation and less similar to Arbitrageurs.  Third, more fierce competition creates an overall negative effect on revenue: prices not just become flatter but also their average level drops. In this case study, average $\frac{1}{T} \int_0^T \hat P_t \,dt$ declines by about 1.2\% as we move from 4 to 24 agents. In Section \ref{sec:asymptotics} we formalize this effect by studying the asymptotics (in the fully homogeneous case) as $N \to \infty$.

\begin{figure}[!htbp]
        \centering
   \includegraphics[width=0.49\linewidth]{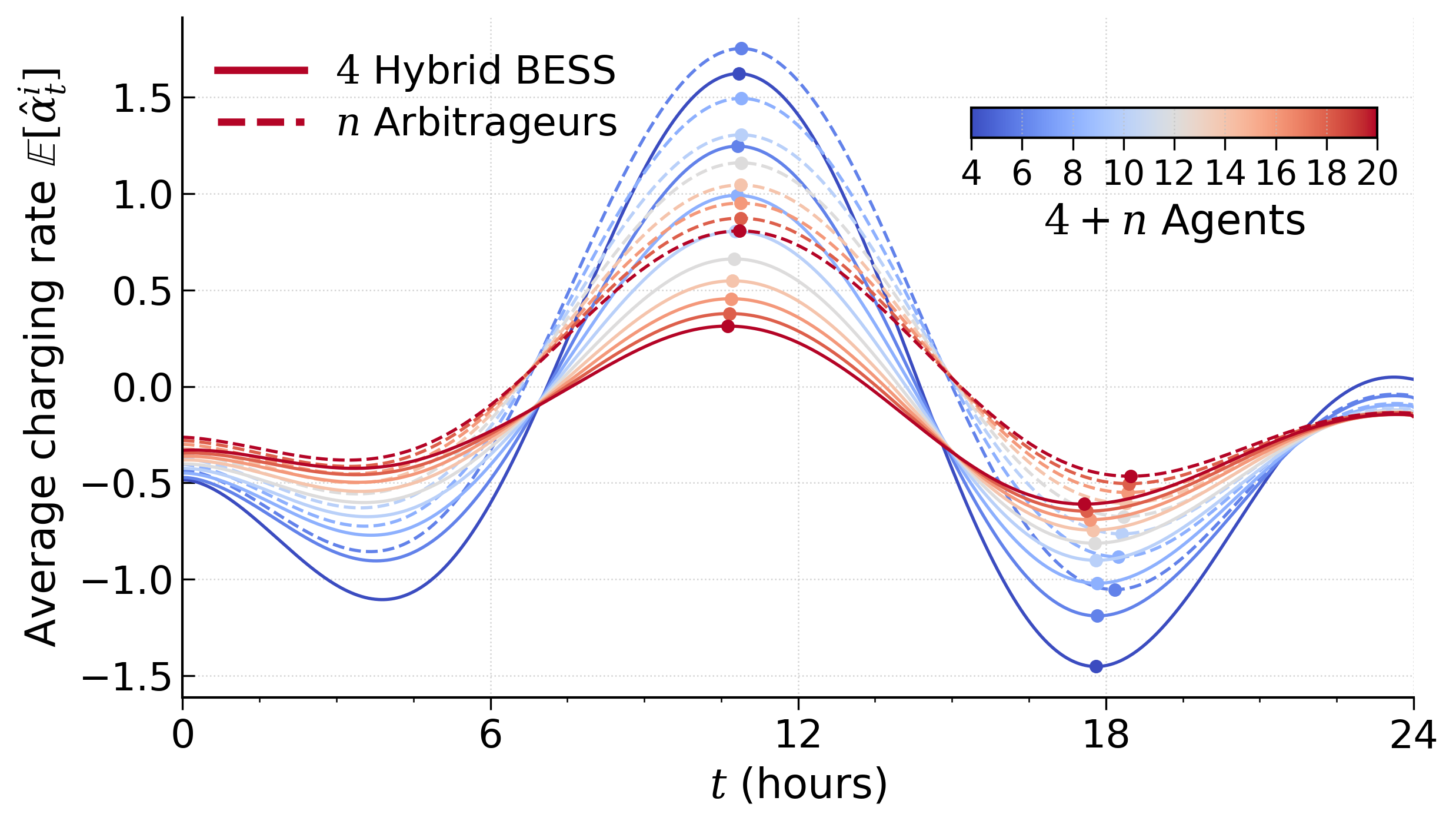}
   \includegraphics[width=0.49\linewidth]{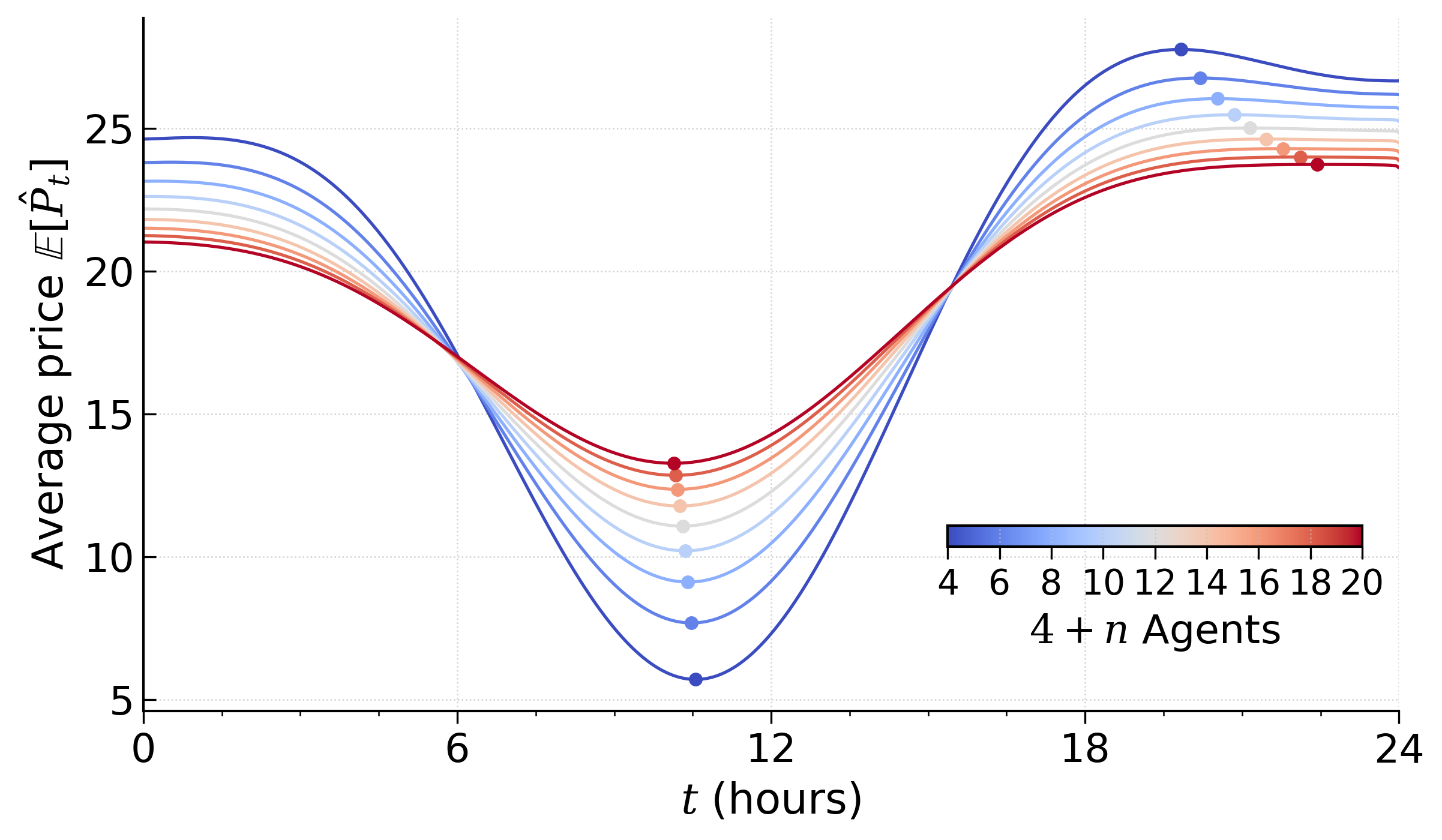}
              \caption{ \emph{Left panel:} average control of operators  $\E\big[\hat{\alpha}^{(\mathrm{Type})}_t\big], \mathrm{Type} \in \{\mathrm{Arb}, \mathrm{Hyb}\}$ as a function of $t$. \emph{Right panel:} daily average equilibrium price $\E[ \hat P_t] $ as the number of BESS varies. The dots indicate the time of the daily peak/trough of each curve.}
        \label{fig:competition-effect}
\end{figure}

\subsection{Behavior of Heterogeneous Markets}\label{sec:hetero}

To have a broader qualitative insight into how competition impacts markets, we consider a large-scale simulation experiment. To this end, we simulate 500 different markets with $N=10$ heterogeneous BESS, according to the following rules.
To simulate a collection of $N$ distinct agents, we start with the base setup of Section \ref{sec:illustration} and vary their self-generation capacity $G^i$, taking $\sigma^i = \sqrt{G^i}$ and $a^i(t) = G^i \cdot a(t), b^i(t) = G^i \cdot b(t)$ in \eqref{eq:sample-mu} with $\log G^i \sim \Unif(-1, 1)$ i.i.d. We also vary the correlations with the common noise driver, $\rho^i \sim \Unif(0, 0.9)$. Secondly, we vary the SOC capacities by sampling $c^i_3 \sim \Unif(0.1, 1)$ which controls the penalty on SOC deviations; smaller $c^i_3$ is interpreted as larger BESS energy capacity. In sum, the simulations represent a collection of heterogeneous hybrid BESS that differ in their hybrid build-outs, their correlation with the net load and in their battery capacities. 

Figure \ref{fig:hetero-hist-price} summarizes how the $N=10$ BESS operators collectively flatten the prices. We visualize the percent reduction of daily price TB with the BESS present compared to the case where there are no BESS at all:
\begin{align*}
    \frac{ TB(\bar{P} - c_1 (Q_t - \sum_i \hat\alpha^i_t) )}{TB( \bar{P} - c_1 Q_t  )}.
\end{align*}
Figure \ref{fig:hetero-hist-revenue} shows the distribution of BESS daily revenue, normalized by its generation capacity $\frac{1}{G_i} \bigl(\bar{P}^i - c_1^i (Q_t - \sum_j \hat\alpha^j_t) \hat\alpha^i_t \bigr)$. We observe that there is a wide range of revenues, primarily driven by the varying BESS capacities.

\smallskip
\noindent\textbf{Varying net supply curves and price impact weights.} For a different randomization and heterogeneity aspect, we maintain identical BESS characteristics but instead sample a collection of different $\theta(\cdot)$ curves, which represent the net load across days of the year. Specifically, we replace  \eqref{eq:sample-theta} with
\begin{align*}
\theta(t) & = \bar{\Theta} - A_1 \sin \bigl( \frac{\pi }{24}(2t +F_1) \bigr) - A_2 \sin \bigl(\frac{\pi}{24}(4t + F_2) \bigr), \\
\bar{\Theta} & \sim \Unif(25, 35), \, A_1  \sim \Unif(18, 26), \, A_2 \sim \Unif(6,10), \, F_1 \sim \Unif(12,18), \, F_2 \sim \Unif(-9, -5).
\end{align*}
This randomizes the magnitude and timing of the evening price peak and the presence of the morning mini-peak. 
We furthermore randomize the weights $w_{ij}$ in \eqref{eq:price-impact}, which represents the different impact of neighboring BESS depending on grid congestion regimes. For the latter, we set $w_{ii} = 1$ and  sample independently the off-diagonal elements $w_{ji}=w_{ij}\sim\Unif (0.5,1)\, \forall i >j$. Figure \ref{fig:theta-shape} demonstrates that daily TB is reduced by 30-40\% (median of 36.2\%) with  10 BESS and by nearly half (35-48\%, median of 40.8\%) with 15 BESS, giving a sense of incremental price flattening as $N$ grows.

\begin{figure}[!htbp]
     \centering
     \begin{subfigure}{0.33\linewidth}
         \includegraphics[width=\linewidth]{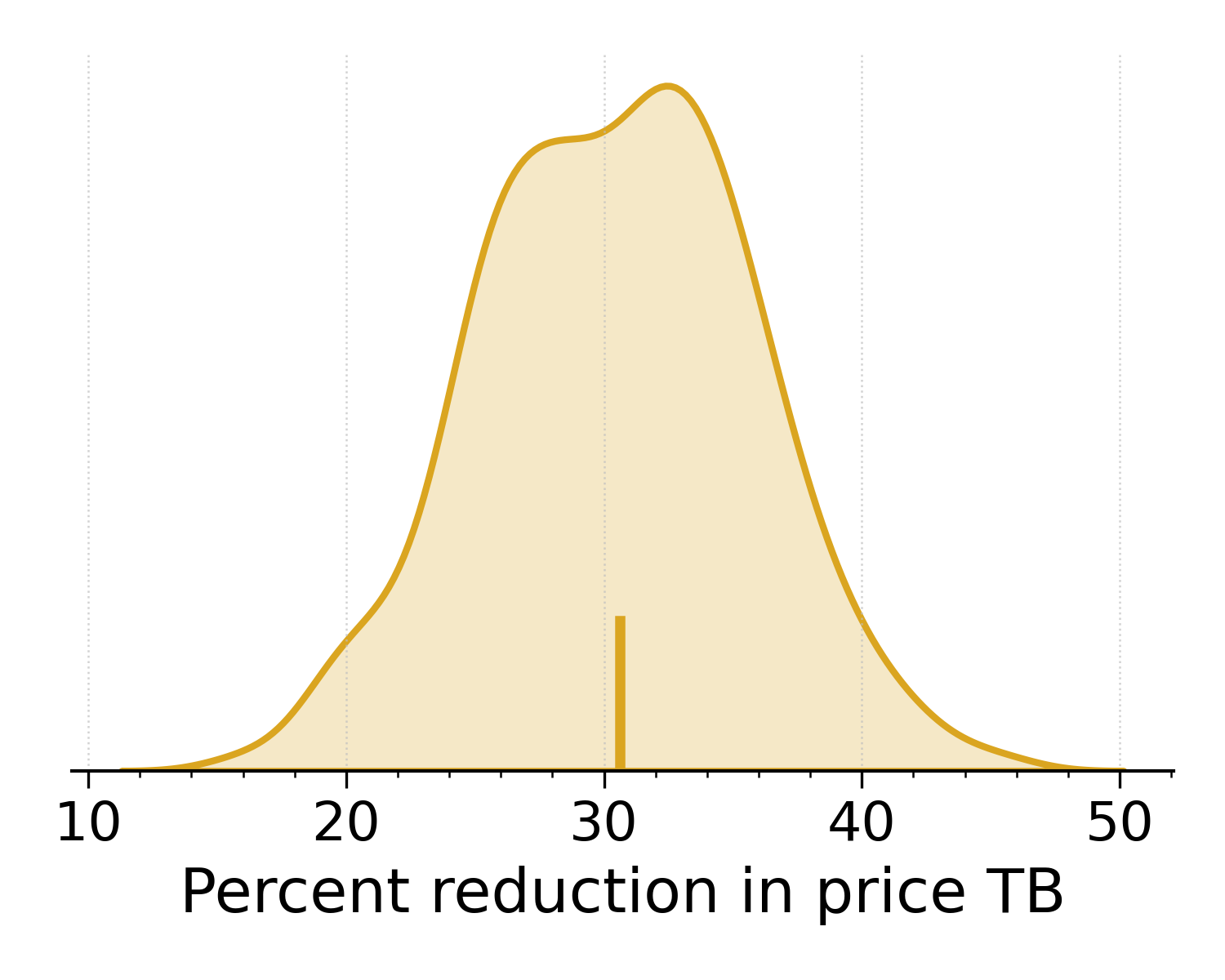}
         \caption{Percent reduction in daily TB price}
        \label{fig:hetero-hist-price}
     \end{subfigure}
     \begin{subfigure}{0.325\linewidth}
         \includegraphics[width=\linewidth]{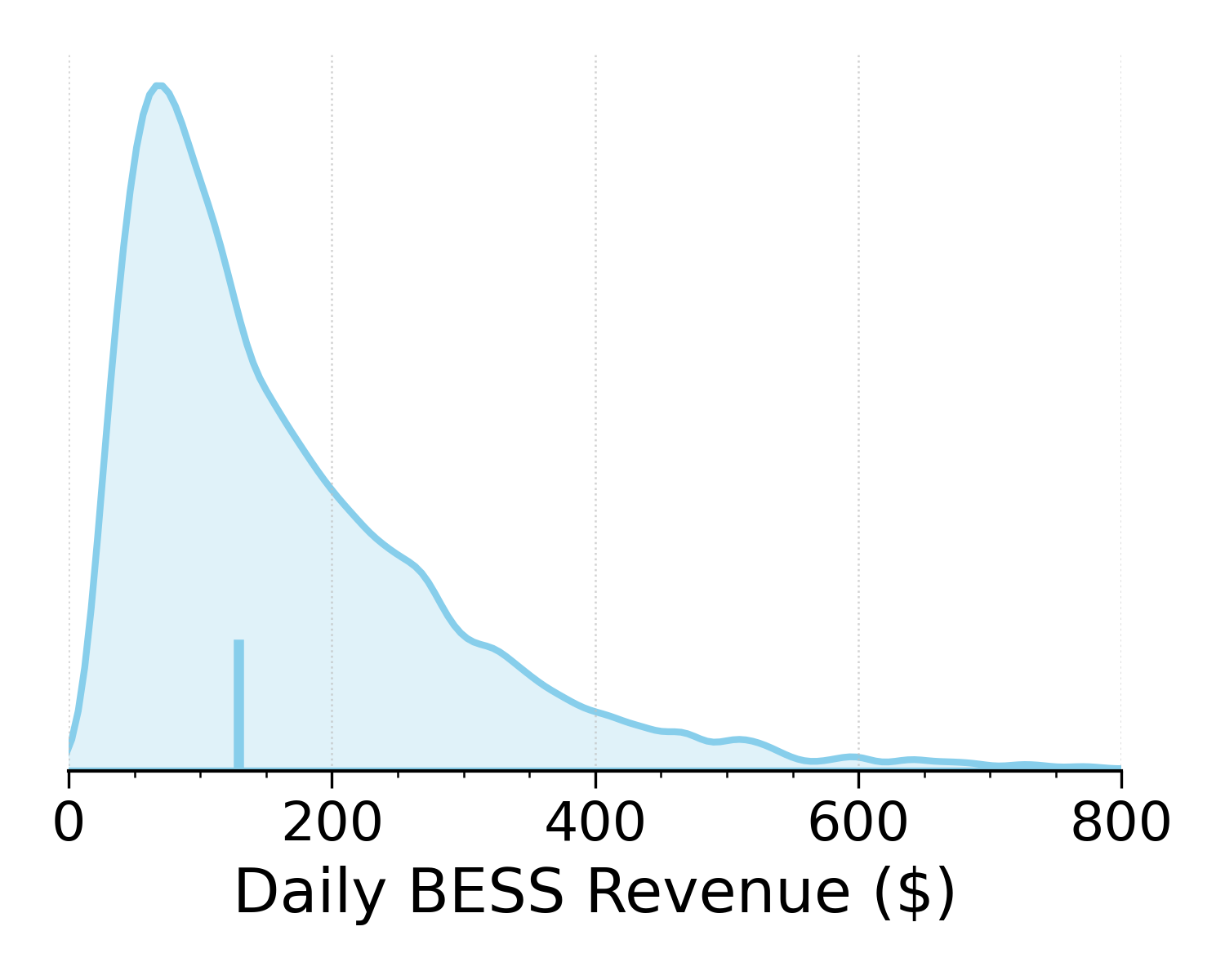}
         \caption{Daily BESS revenue}
         \label{fig:hetero-hist-revenue}
     \end{subfigure}
     \begin{subfigure}{0.325\linewidth}
         \includegraphics[width=\linewidth]{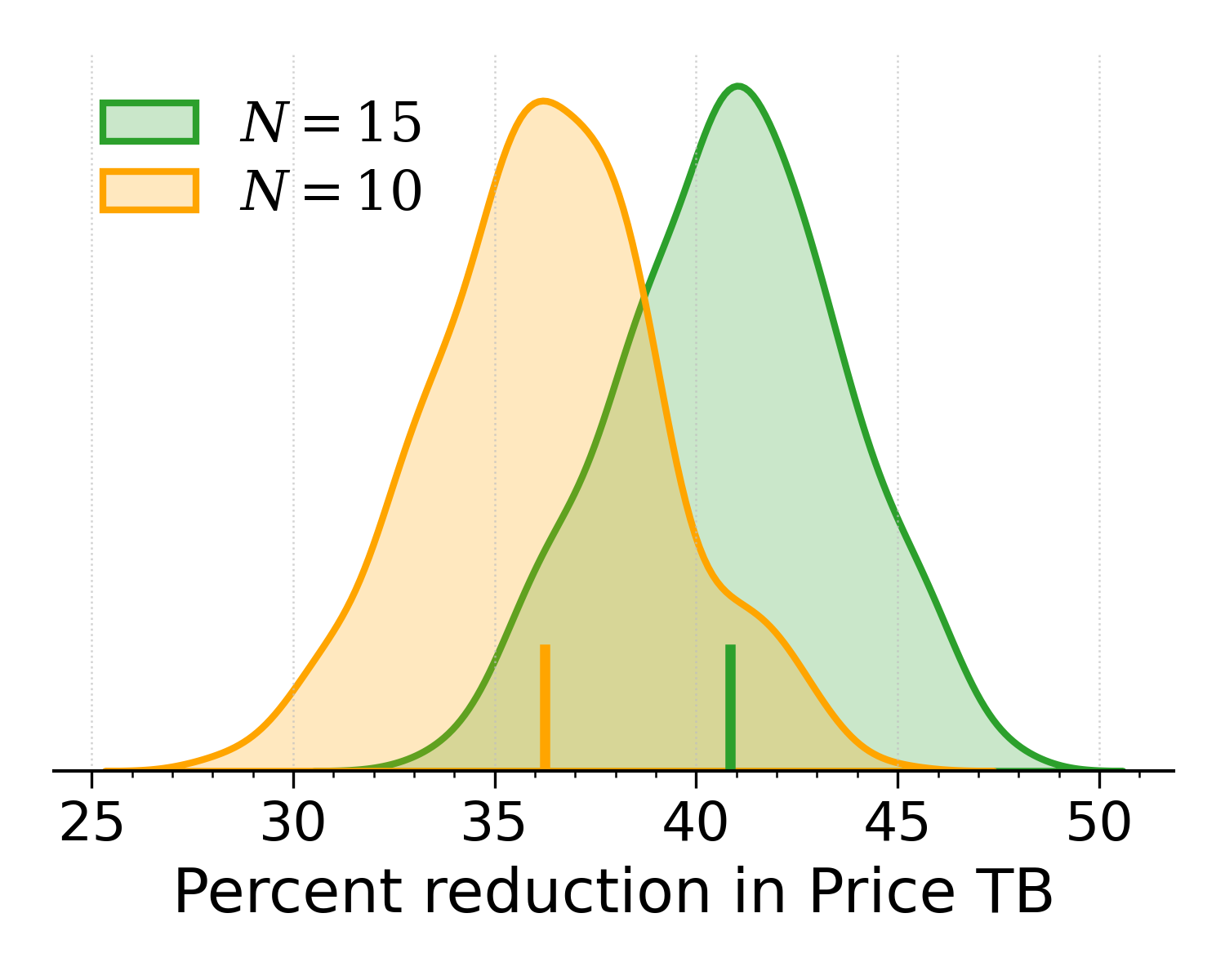}
         \caption{Percent reduction in daily TB price}
         \label{fig:theta-shape}
     \end{subfigure}
     \caption{Distributions across 500 heterogeneous markets with $N=10$ BESS operators. Vertical lines denote the corresponding medians.}
\end{figure}

\section{Impact of BESS Sizing}\label{sec:major-minor}

In real-life power markets, the relative size of hybrid assets is highly heterogeneous, and includes both relatively small and relatively large assets. The issue of sizing is pertinent to the question of market power, whereby some large operators might be able to influence prices through their charging strategies, in effect transitioning from price takers to price makers. A related issue concerns potential mergers and coalition formation (e.g.~a coordinated dispatch of multiple BESS) and their impact on the market dynamics. Finally, analysis of BESS size sheds light on the gap between monopoly and competitive equilibria. 

To analyze these market features, we use the  homogeneous setup introduced in \eqref{eq:homo-condition}, further restricting to the case of all $N$ BESS being  arbitrageurs, and $\rho_i =0 \, \forall i$. We then consider $\cR_m$ ``Minor'' operators each of whom control $m \ge 1$ assets and $\cR_M$  ``Major'' operators each of whom control $M \ge m$ assets. Since in total there are $N$ BESS units, the corresponding $(m,M,N)$ market is defined by
\begin{align}\label{merge_eq}
    m \times \calR_m + M \times \calR_M = N.
\end{align}
From the control perspective, the total number of operators is $\calR_m+\calR_M$, and thus $\balpha \in \bbR^{\calR_m+\calR_M}$. Below we concentrate on a single Major operator $\calR_M=1$. There market then consists of two types of agents, 1 Major and $\mathcal{R}_m$ Minors, where the control of $i^{th}$ operator of each type
is encoded as $\alpha^{i,m, (m,M,N)}$ and $\alpha^{i,M, (m,M,N)}$, in the $(m, M, N)$ market.

Managing multiple units corresponds to summing up the respective SOCs, revenue and cost of several otherwise identical assets. Because the assets are homogeneous and permutation-invariant (Section \ref{homo_sys}), the control $\alpha^{i,(\mathrm{Size})}_t$ for a block of  
$\mathrm{Size}$ units should match the sum of the individual controls, for $\mathrm{Size} = \{M, m\}$.
 Summing the constituent SOC's, the block SOCs denoted as $S^{i,(\mathrm{Size})}$ has the dynamics
\begin{align}
    S^{j,(\mathrm{Size})}_t &= S^{j,(\mathrm{Size})}_0 + \int_{0}^t\alpha^{j,(\mathrm{Size})}_s\,ds + \int_{0}^t \sqrt{\mathrm{Size}}\;\sigma\, d W^{j,(\mathrm{Size})}_s , \quad S^{j,(\mathrm{Size})}_0 = \mathrm{Size}\cdot S_0. \label{large_agent_dynamics}
\end{align}
The diffusion coefficient scales as $\sqrt{\mathrm{Size}}$ by independence of the underlying region-specific Brownian motions. Specializing to Major and Minor operators, denoted by the superscripts $S^{i,(M)}, S^{i,(m)}$, and generically as $S^{i,(\mathrm{Size})}$, $\mathrm{Size} \in \{m,M\}$, the objective is thus 
\begin{align}
\label{eq:J-major}
J^{\ell,(\mathrm{Size})}(\balpha) & = \mathbb{E} \Big[\underbrace{ \int_0^T 
\Big( \bar{P} - c_1 \Bigl(Q_t -  
\sum_{i=1}^{\calR_m} \alpha^{i,(m)}_t-\sum_{i=1}^{\calR_M}\alpha^{i,(M)}_t\Bigr)\Big) \alpha^{\ell,(\mathrm{Size})}_t }_{-Revenue} \Big] \\ \nonumber
& \qquad +  \mathbb{E} \Big[  \underbrace{\frac{c_2}{\mathrm{Size}}(\alpha^{\ell,(\mathrm{Size})}_t)^2 
   + \frac{c_3}{\mathrm{Size}}\big(S^{\ell,(\mathrm{Size})}_t - \mathrm{Size} \cdot \zeta (t)\big)^2 
+ \frac{c_4}{\mathrm{Size}}\big(S^{\ell,(\mathrm{Size})}_T - \mathrm{Size} \cdot \zeta (T)\big)^2 }_{Cost}\Big].
\end{align}

In \eqref{eq:J-major}, the cost coefficients are scaled appropriately to preserve incentives. The aggregate revenue is linear in $\alpha$, which is already consistent with summing $\alpha^{i,(\mathrm{Size})} = \sum_{\ell=1}^{\cR_{\mathrm{Size}}} \alpha^\ell$. To ensure that the scale of dispatch costs in \eqref{eq:J-major} matches the sum of individual asset costs,   the effective cost coefficients are divided by $\mathrm{Size}$. Similarly, the SOC target is scaled linearly as $\mathrm{Size}\cdot\zeta(t)$. 

\subsection{Cost of Competition and Major/Minor Market}\label{sec:major}

\begin{figure}[!htp]
\centering
\begin{tikzpicture}[
  x=0.52cm,
  y=0.78cm,
  every node/.style={font=\footnotesize},
  box/.style={draw, rounded corners=0.8pt, line width=0.4pt}
]


\def\H{0.9}
\def\W{5.2}
\def\B{0.62}
\def\G{1.1}
\def\Y{0.45}

\node[font=\small] at (-4, \Y) {a) Major operator size};

\begin{scope}[shift={(0,0)}]
\filldraw[fill=arbicolor!25, draw=arbicolor, box] (0,0) rectangle (\W,\H);

\filldraw[fill=arbicolor!25, draw=arbicolor, box] (0,0) rectangle (\B,\H);
\filldraw[fill=arbicolor!25, draw=arbicolor, box] (\B,0) rectangle (2*\B,\H);
\filldraw[fill=arbicolor!25, draw=arbicolor, box] (2*\B,0) rectangle (3*\B,\H);

\filldraw[fill=arbicolor!25, draw=arbicolor, box] (\W-3*\B,0) rectangle (\W-2*\B,\H);
\filldraw[fill=arbicolor!25, draw=arbicolor, box] (\W-2*\B,0) rectangle (\W-\B,\H);
\filldraw[fill=arbicolor!25, draw=arbicolor, box] (\W-\B,0) rectangle (\W,\H);

\filldraw[fill=hybridcolor!25, draw=hybridcolor, box]
  (0,0) rectangle (\B,\H);

\node at (0.5*\B,\Y) {$M$};
\node at (0.5*\B,\Y*2.5) {1};
\node at (1.5*\B,\Y) {$m$};
\node at (1.5*\B,\Y*2.5) {1};
\node at (2.5*\B,\Y) {$m$};
\node at (2.5*\B,\Y*2.5) {1};
\node at (0.5*\W,\Y) {$\cdots$};
\node at (\W-2.5*\B,\Y) {$m$};
\node at (\W-2.5*\B,\Y*2.5) {1};
\node at (\W-1.5*\B,\Y) {$m$};
\node at (\W-1.5*\B,\Y*2.5) {1};
\node at (\W-0.5*\B,\Y) {$m$};
\node at (\W-0.5*\B,\Y*2.5) {1};
\end{scope}

\node at ({\W+0.55},0.45) {$\longrightarrow$};

\begin{scope}[shift={({\W+\G},0)}]
\filldraw[fill=arbicolor!25, draw=arbicolor, box] (0,0) rectangle (\W,\H);

\filldraw[fill=arbicolor!25, draw=arbicolor, box] (0,0) rectangle (\B,\H);
\filldraw[fill=arbicolor!25, draw=arbicolor, box] (\B,0) rectangle (2*\B,\H);
\filldraw[fill=arbicolor!25, draw=arbicolor, box] (2*\B,0) rectangle (3*\B,\H);

\filldraw[fill=arbicolor!25, draw=arbicolor, box] (\W-3*\B,0) rectangle (\W-2*\B,\H);
\filldraw[fill=arbicolor!25, draw=arbicolor, box] (\W-2*\B,0) rectangle (\W-\B,\H);
\filldraw[fill=arbicolor!25, draw=arbicolor, box] (\W-\B,0) rectangle (\W,\H);

\filldraw[fill=hybridcolor!25, draw=hybridcolor, box]
  (0,0) rectangle (2*\B,\H);

\node at (\B,\Y) {$M$};
\node at (\B,2.5*\Y) {2};
\node at (2.5*\B,\Y) {$m$};
\node at (2.5*\B,2.5*\Y) {1};
\node at (0.5*\W,\Y) {$\cdots$};
\node at (\W-2.5*\B,\Y) {$m$};
\node at (\W-2.5*\B,2.5*\Y) {1};
\node at (\W-1.5*\B,\Y) {$m$};
\node at (\W-1.5*\B,2.5*\Y) {1};
\node at (\W-0.5*\B,\Y) {$m$};
\node at (\W-0.5*\B,2.5*\Y) {1};
\end{scope}

\node at ({2*\W+\G+0.55},0.45) {$\cdots$};

\begin{scope}[shift={({2*(\W+\G)},0)}]
\filldraw[fill=arbicolor!25, draw=arbicolor, box] (0,0) rectangle (\W,\H);

\filldraw[fill=arbicolor!25, draw=arbicolor, box] (0,0) rectangle (\B,\H);
\filldraw[fill=arbicolor!25, draw=arbicolor, box] (\B,0) rectangle (2*\B,\H);
\filldraw[fill=arbicolor!25, draw=arbicolor, box] (2*\B,0) rectangle (3*\B,\H);

\filldraw[fill=arbicolor!25, draw=arbicolor, box] (\W-3*\B,0) rectangle (\W-2*\B,\H);
\filldraw[fill=arbicolor!25, draw=arbicolor, box] (\W-2*\B,0) rectangle (\W-\B,\H);
\filldraw[fill=arbicolor!25, draw=arbicolor, box] (\W-\B,0) rectangle (\W,\H);

\filldraw[fill=hybridcolor!25, draw=hybridcolor, box]
  (0,0) rectangle (\W-\B,\H);

\node at ({0.5*(\W-\B)},\Y) {$M$};
\node at ({0.5*(\W-\B)},2.5*\Y) {31};
\node at ({\W-0.5*\B},\Y) {$m$};
\node at ({\W-0.5*\B},2.5*\Y) {1};

\end{scope}

\node at ({3*\W+2*\G+0.55},0.45) {$\longrightarrow$};

\begin{scope}[shift={({3*(\W+\G)},0)}]
\filldraw[fill=arbicolor!25, draw=arbicolor, box] (0,0) rectangle (\W,\H);

\filldraw[fill=arbicolor!25, draw=arbicolor, box] (0,0) rectangle (\B,\H);
\filldraw[fill=arbicolor!25, draw=arbicolor, box] (\B,0) rectangle (2*\B,\H);
\filldraw[fill=arbicolor!25, draw=arbicolor, box] (2*\B,0) rectangle (3*\B,\H);

\filldraw[fill=arbicolor!25, draw=arbicolor, box] (\W-3*\B,0) rectangle (\W-2*\B,\H);
\filldraw[fill=arbicolor!25, draw=arbicolor, box] (\W-2*\B,0) rectangle (\W-\B,\H);
\filldraw[fill=arbicolor!25, draw=arbicolor, box] (\W-\B,0) rectangle (\W,\H);

\filldraw[fill=hybridcolor!25, draw=hybridcolor, box]
  (0,0) rectangle (\W,\H);

\node at ({0.5*(\W)},\Y) {$M$};
\node at ({0.5*(\W)},2.5*\Y) {32};
\end{scope}

\end{tikzpicture}
\vspace*{6pt}

\begin{tikzpicture}[
  x=0.52cm,
  y=0.78cm,
  every node/.style={font=\footnotesize},
  box/.style={draw, rounded corners=0.8pt, line width=0.4pt}
]

\def\H{0.9}
\def\W{5.2}
\def\B{0.62}
\def\G{1.1}
\def\Y{0.38}
\def\YT{1.12}
\def\MW{0.625*\W} 

\node[font=\small] at (-4,\Y) {b) Minor operator size};

\begin{scope}[shift={(0,0)}]

\filldraw[fill=arbicolor!25, draw=arbicolor, box] (0,0) rectangle (\W,\H);

\filldraw[fill=hybridcolor!25, draw=hybridcolor, box]
(0,0) rectangle (\MW,\H);

\draw[box] (\MW,0) rectangle ({\MW+0.5*\B},\H);
\draw[box] (\MW+0.5*\B,0) rectangle ({\MW+1*\B},\H);
\draw[box] (\W-\B,0) rectangle ({\W-0.5*\B},\H);
\draw[box] ({\W-0.5*\B},0) rectangle (\W,\H);

\node at ({0.5*\MW},\Y) {$M$};
\node at ({0.5*\MW},\YT) {20};

\node at ({\MW+0.25*\B},\YT) {1};

\node at ({\MW+0.75*\B},\YT) {1};

\node at ({0.5*(\MW+\W)},\Y) {$\cdots$};

\node at ({\W-0.75*\B},\YT) {1};
\node at ({\W-0.25*\B},\YT) {1};

\end{scope}

\node at ({\W+0.55},0.45) {$\longrightarrow$};

\begin{scope}[shift={({\W+\G},0)}]

\filldraw[fill=arbicolor!25, draw=arbicolor, box] (0,0) rectangle (\W,\H);

\filldraw[fill=hybridcolor!25, draw=hybridcolor, box]
(0,0) rectangle (\MW,\H);

\draw[box] (\MW,0) rectangle ({\MW+\B},\H);
\draw[box] ({\W-\B},0) rectangle (\W,\H);

\node at ({0.5*\MW},\Y) {$M$};
\node at ({0.5*\MW},\YT) {20};

\node at ({\MW+0.5*\B},\Y) {$m$};
\node at ({\MW+0.5*\B},\YT) {2};

\node at ({0.5*(\MW+\W)},\Y) {$\cdots$};

\node at ({\W-0.5*\B},\Y) {$m$};
\node at ({\W-0.5*\B},\YT) {2};

\end{scope}

\node at ({2*\W+\G+0.55},0.45) {$\cdots$};

\begin{scope}[shift={({2*(\W+\G)},0)}]

\filldraw[fill=arbicolor!25, draw=arbicolor, box] (0,0) rectangle (\W,\H);

\filldraw[fill=hybridcolor!25, draw=hybridcolor, box]
(0,0) rectangle (\MW,\H);

\draw[box] (\MW,0) rectangle ({0.5*(\MW+\W)},\H);
\draw[box] ({0.5*(\MW+\W)},0) rectangle (\W,\H);

\node at ({0.5*\MW},\Y) {$M$};
\node at ({0.5*\MW},\YT) {20};

\node at ({0.5*(\MW+0.5*(\MW+\W))},\Y) {$m$};
\node at ({0.5*(\MW+0.5*(\MW+\W))},\YT) {6};

\node at ({0.5*(0.5*(\MW+\W)+\W)},\Y) {$m$};
\node at ({0.5*(0.5*(\MW+\W)+\W)},\YT) {6};

\end{scope}

\node at ({3*\W+2*\G+0.55},0.45) {$\longrightarrow$};

\begin{scope}[shift={({3*(\W+\G)},0)}]

\filldraw[fill=arbicolor!25, draw=arbicolor, box] (0,0) rectangle (\W,\H);

\filldraw[fill=hybridcolor!25, draw=hybridcolor, box]
(0,0) rectangle (\MW,\H);

\draw[box] (\MW,0) rectangle (\W,\H);

\node at ({0.5*\MW},\Y) {$M$};
\node at ({0.5*\MW},\YT) {20};

\node at ({0.5*(\MW+\W)},\Y) {$m$};
\node at ({0.5*(\MW+\W)},\YT) {12};

\end{scope}
\end{tikzpicture}

\caption{Illustration of Major-Minor market compositions. \emph{Top row:} increasingly dominant Major operator of size $M$ in a market with a fixed total of $N=32$ BESS units. \emph{Bottom row:} single Major operator that controls $M=20$ BESS units and a crowd of identical Minor operators each of whom control $m=1,2,3,4,6,12$ units. The latter case corresponds to a non-symmetric duopoly.}
\label{fig:market-size-32-demo1}
\end{figure}
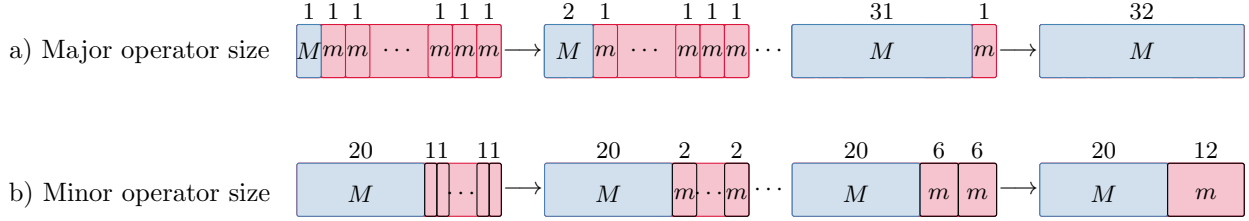

We first consider the case where there is a single Major operator $\cR_M=1$ of size $M$ and $\cR_n = N-M$ Minor operators of size $m=1$, see the top diagram in Figure \ref{fig:market-size-32-demo1}. The base case $M=1$ corresponds to homogeneous, symmetric competition of equally-sized operators, while $M=N$ corresponds to a monopolist that single-handedly operates all the BESSs. Note that for this case study in order to isolate the competitive effects, we assume that all operators are arbitrageurs and there is zero self-generation $\mu^i_t \equiv 0$.

Figure \ref{fig:major_minor_profit} shows the relative magnitude of BESS charging rates as $M$ varies. The primary effect is that competition increases dispatch as the operators do not fully internalize the full impact of their actions on others. Consequently, as $M$ grows and there are fewer agents competing, the total BESS dispatch drops in magnitude. The monopolist ($M=32$) utilizes the battery nearly half as much as the crowd of $N=32$ individual assets. Hence, highly competitive markets generate extra ``churn'' as each operator aggressively charges and discharges, cannibalizing everyone's profits. Two related effects are documented in the left panel of Figure \ref{fig:major_minor_profit}. First, with less competition each operator increases their dispatch. Specifically, for a Minor operator who always operates an asset of unit size, we see a substantial increase (that is convex in $M$) as the Major agent ``absorbs'' other competitors. When only two operators with size $M=31$ and $m=1$ remain in the market, the Minor operator exerts roughly 3.10 times its baseline level of control reflecting stronger incentives due to reduced competition. Second, while a larger operator naturally accounts for a larger share of aggregate dispatch, that pattern is sub-linear and $S$-shaped. An operator controlling 5 units dispatches at 3.12x relative to the homogeneous baseline (internalizing competition causes this sub-linear growth), an operator of 16 units dispatches at 6.59x and finally a monopolist dispatches 16.16 times more. In proportional terms, a Major operator who controls half of all capacity uses only $6.59/(6.59+1.42 \cdot 16)=22.5\%$ of total dispatch and a Major operator who controls $M/N=31/32 \approx 96.9\%$ of total capacity only accounts for 83.9\% of aggregate dispatch. 

\begin{figure}[!htbp]
    \centering
       \includegraphics[width=0.325\textwidth]{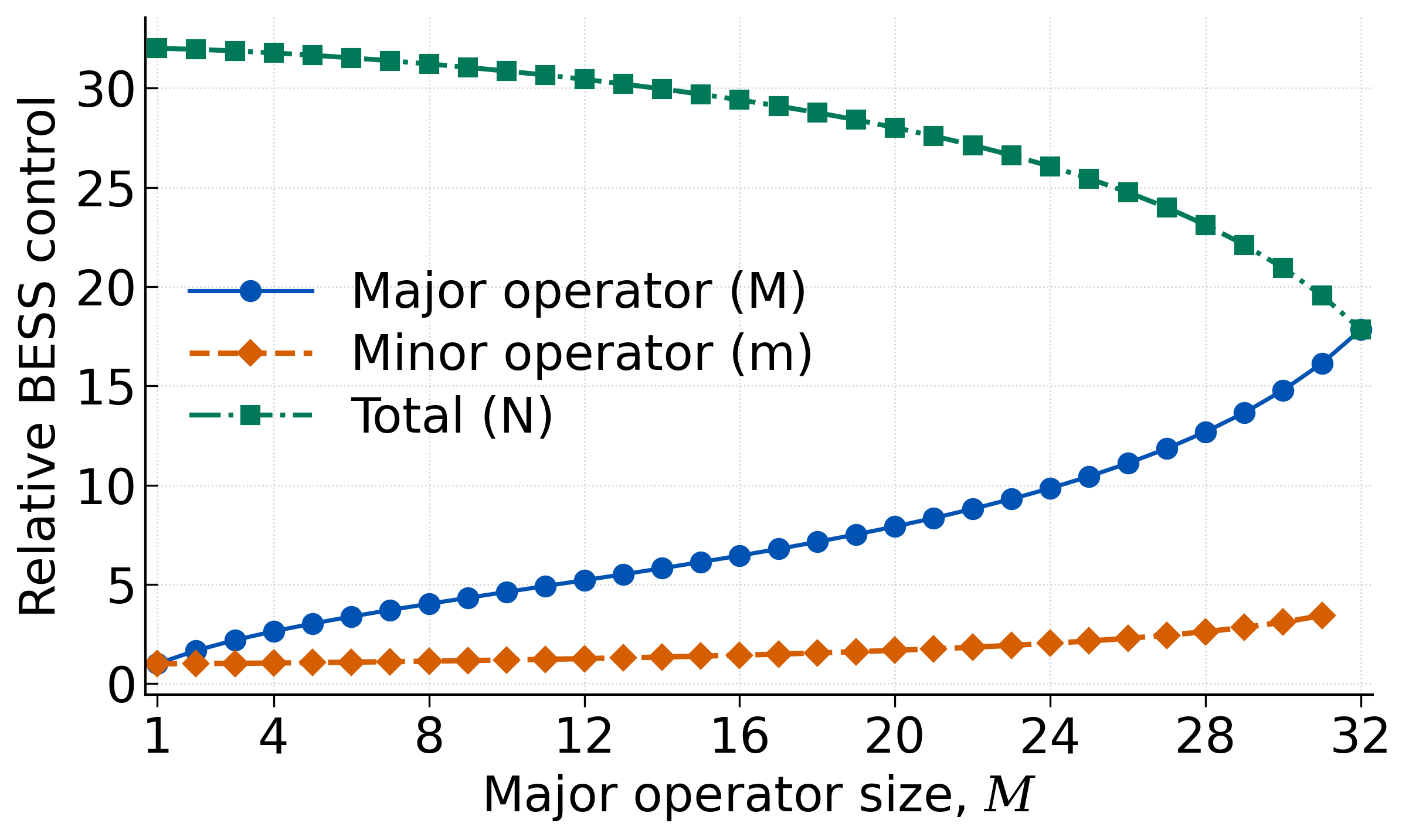}
        \includegraphics[width=0.325\linewidth]{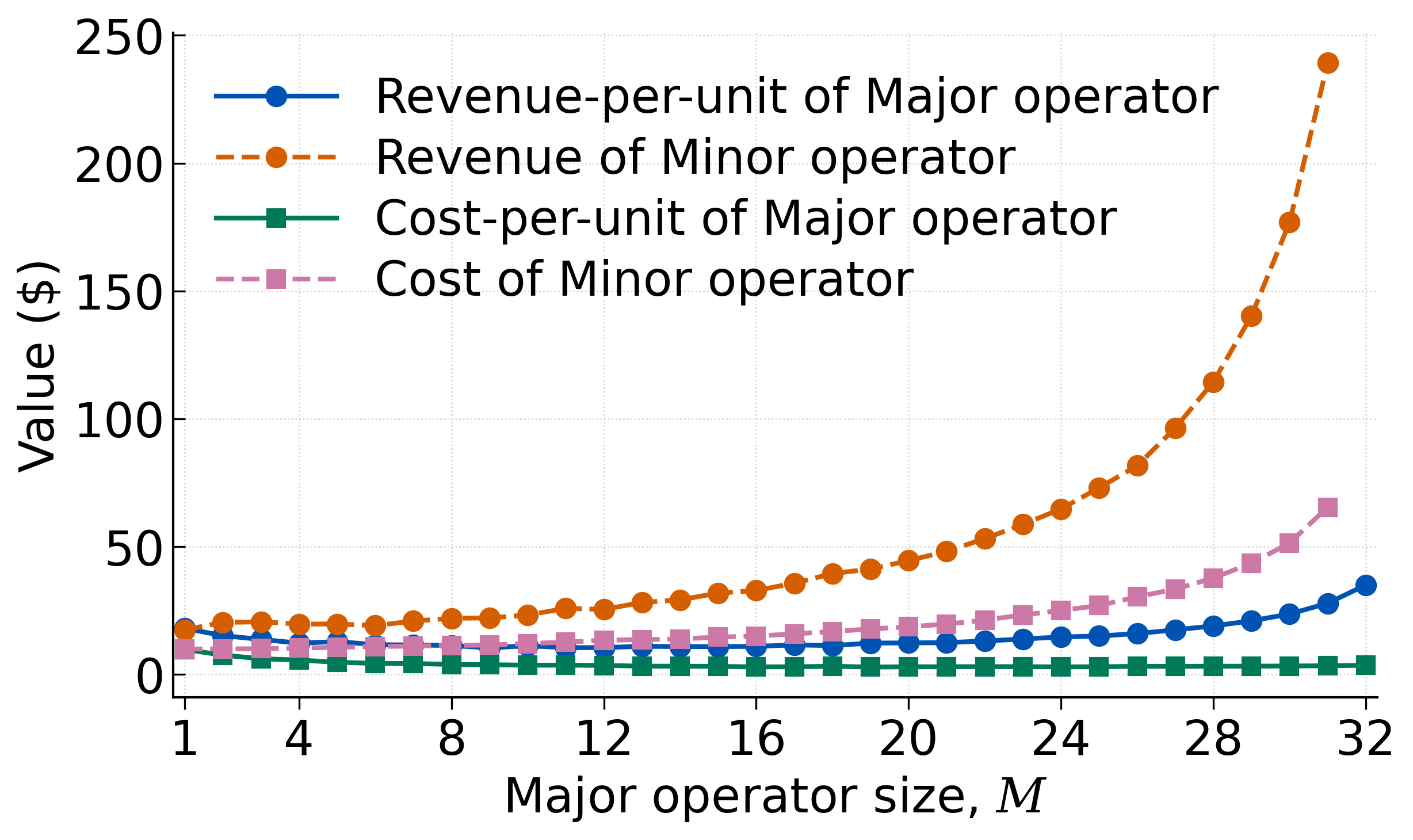}
        \includegraphics[width=0.325\linewidth]{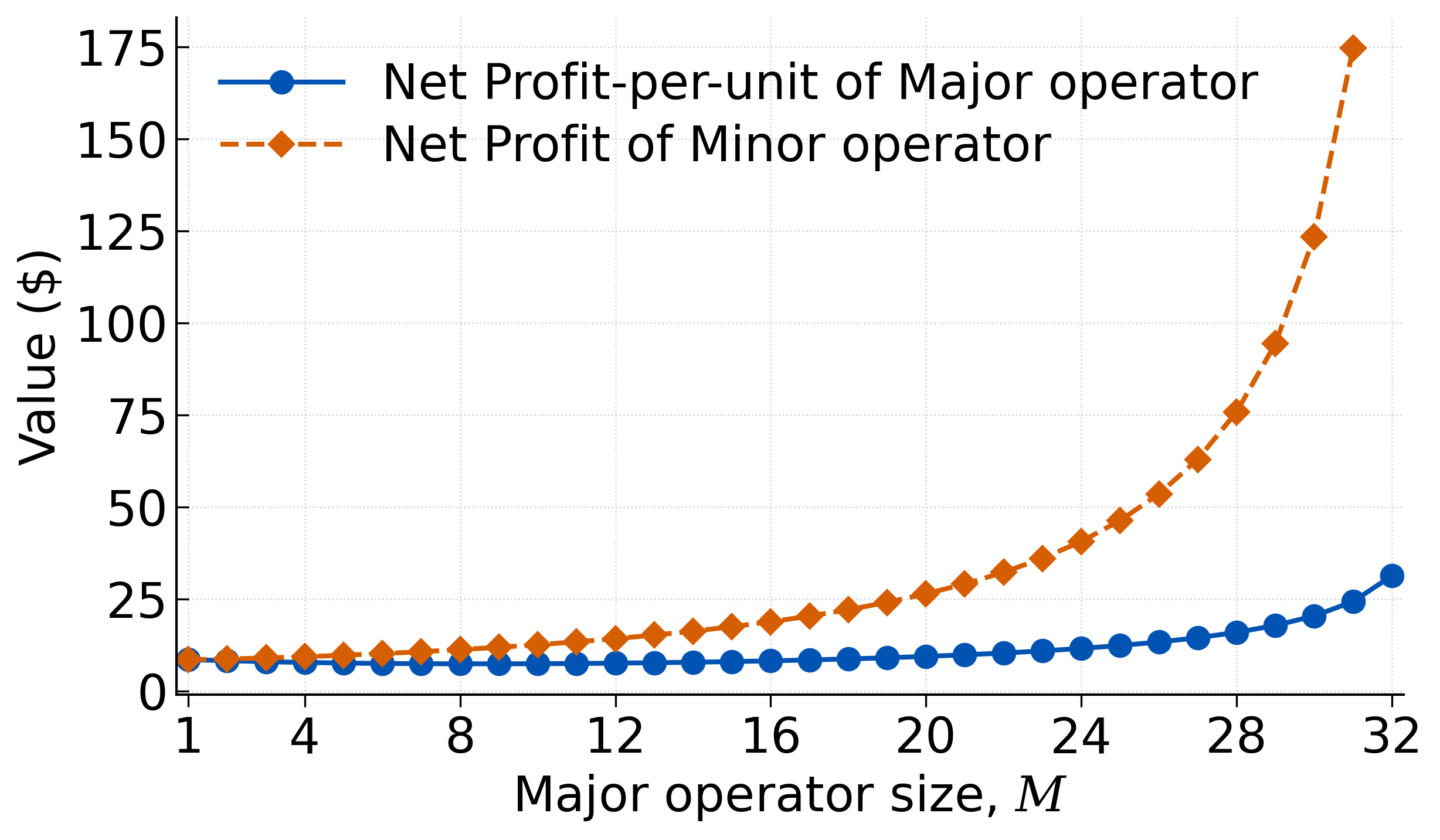}
       \caption{Average dispatch rate (left panel), per-unit revenue \& cost (middle panel) and net profit (right panel) of Major and Minor operators. We consider a market consisting of $N=32$ BESS units, with 1 Major operator controlling $M=1,\ldots,32$ units and $N-M$ Minor operators controlling 1 unit each. The left panel is normalized such that '1' represents the average maximum dispatch $\E[\max_{t \le T^\circ} |\hat{\alpha}_t|]$ of a single-unit operator in a homogeneous market of 32 assets, so that the green and orange curves start at 1.
       }
        \label{fig:major_minor_profit}
\end{figure}

To further understand the oligopoly effect of a monopolist better coordinating respective revenues and costs, the middle and right panels of Figure \ref{fig:major_minor_profit} visualize the revenues, costs and net profits of the Major and Minor operators, expressed in per-unit (i.e.~normalizing the Major operator's quantities by $M$) quantities. We observe that from a marginal per-unit perspective, the bigger beneficiary is actually the Minor agent: they get rid of competition and gain an edge. Their costs rise, but their revenues rise much faster. On the other hand, on a per-unit basis, the costs of the Major operator decline (and are essentially constant for $M>10$) due to diversification, but their revenue has a $U$-shape pattern and does not grow much for $M$ large. As a result for $M>20$, a Minor agent is actually more profitable on a per-unit basis. In other words, if a Minor operator is proposed to join an existing coalition of centrally run $M$ BESS and share the profits, they are better off remaining independent and can generate more profit on their own, if costs of both types of operators are scaled in the manner of \eqref{eq:J-major}.

\subsection{Oligopoly with a Major Operator vs Duopoly}\label{sec:duopoly}

As a complement to the above analysis, we consider the impact of competition on a Major operator of a fixed size. To this end we fix $N$ and take a single Major operator $\cR_M=1$ who controls a fixed fraction of all the assets, with the remaining BESS controlled by $\cR_m$ Minor operators of equal size $m$, $\calR_m = \frac{N-M}{m}$, see bottom row  diagram in Figure \ref{fig:market-size-32-demo1}.

In Figure \ref{fig:major_minor_profit_m}, we  take $N=32, M=20$ and then consider $m \in \{1,2,3,4,6,12\}$. For instance the first case corresponds to a single Major operator and 12 Minor operators of size 1, the last case corresponds to a duopoly $\cR_M+\cR_m = 2$ with an operator of size 20 and an operator of size 12. As $m$ increases, we again have the oligopoly effects: there are fewer controllers in the market, so the cannibalization is mitigated and agents internalize the cost of over-dispatch. When the market ultimately consists of two operators with assets of the same size ($M=m=16$), aggregate control falls to approximately $68.6\%$ of the baseline. This mirrors the result from Section \ref{sec:major}: consolidation reduces competition, which lowers aggregate control.

The left panel of Figure \ref{fig:major_minor_profit_m} shows the relative battery dispatch of the Major and Minor operators as $m$ varies. With fewer competitors, the market becomes more concentrated and the Major operator has stronger incentives to exert more control. As a result,  as $m$ increases, the magnitude of Major operator's charge/discharge rises. Compared to having to deal with 12 Minor competitors, in a duopoly the Major operator will dispatch 51\% more.  For the Minor operators, as their number shrinks, their dispatch grows almost linearly. For instance, if 12 Minor operators of size $m=1$ are replaced with 6 Minor operators of size $m=2$, their respective dispatch grows by 91\%, i.e., scales almost linearly. We refer to \cite{butters2025soaking} for a more structural comparison between monopoly, duopoly and full competition among strategic BESS.

\begin{figure}[!htbp]
    \centering
        \includegraphics[width=0.325\linewidth]{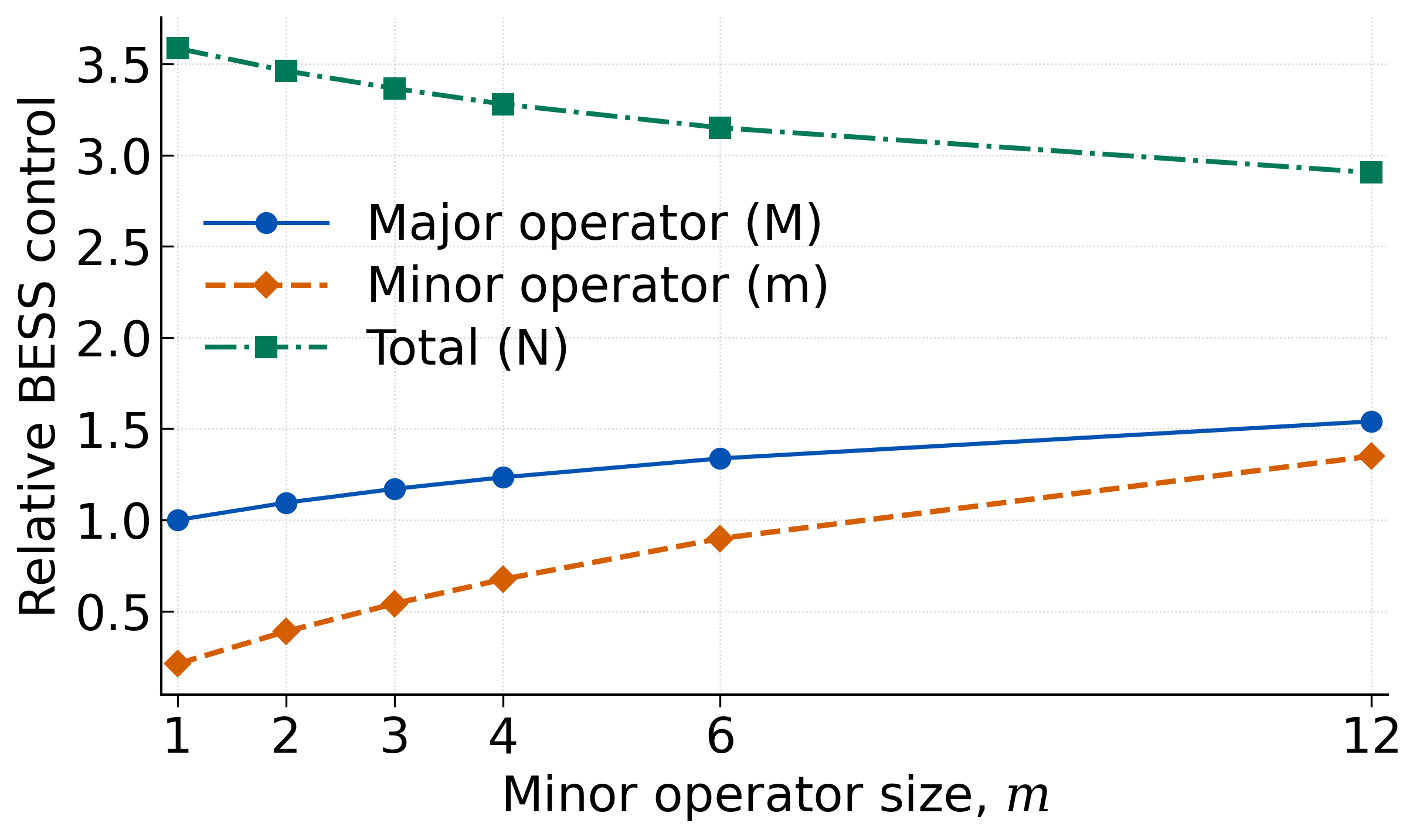}
        \includegraphics[width=0.325\linewidth]{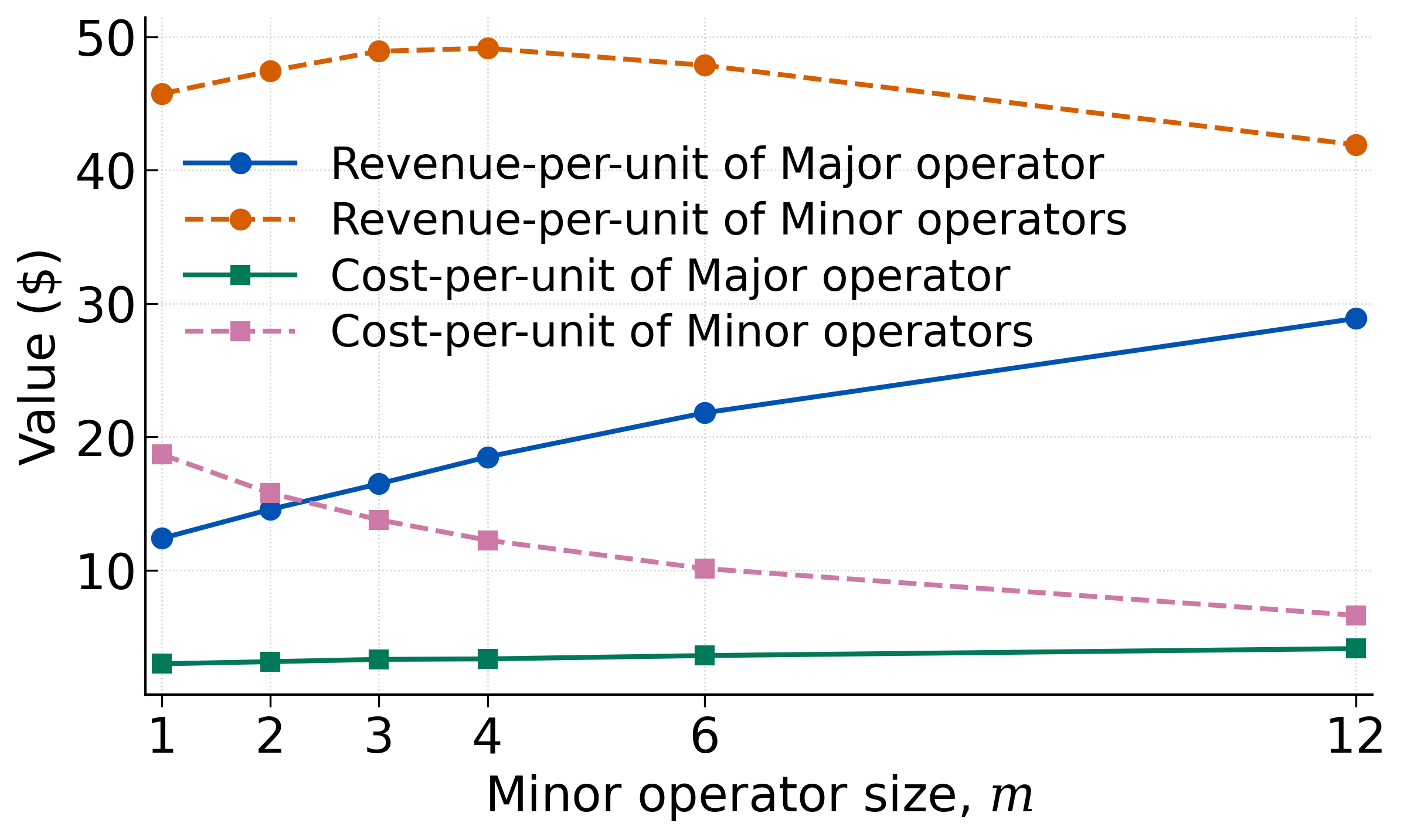}
        \includegraphics[width=0.325\linewidth]{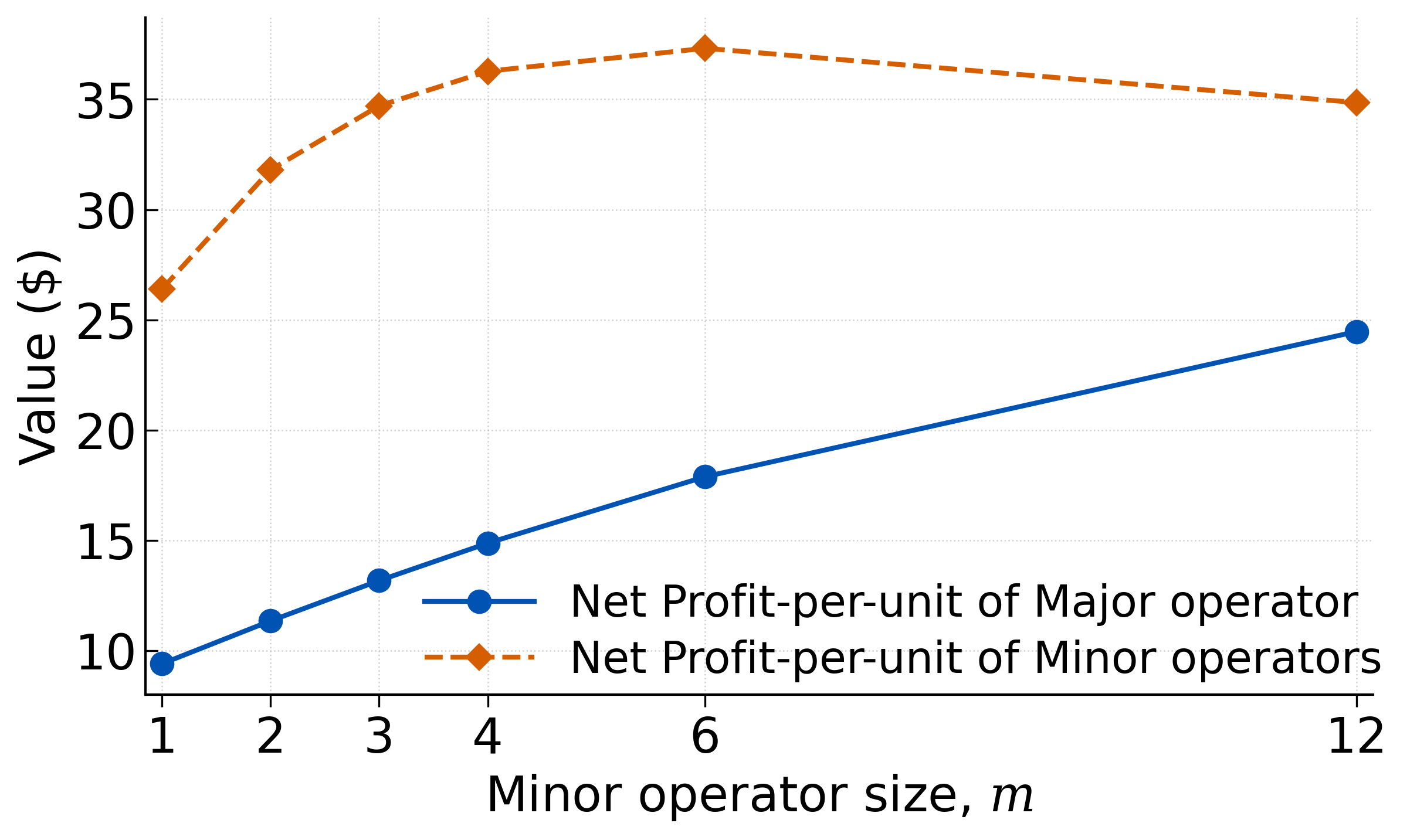}
        \caption{Average dispatch rate (\emph{left panel}), per-unit revenue \& cost (\emph{middle panel}) and net profit (\emph{right panel}) of Major and Minor operators. We consider a market consisting of $N=32$ BESS units, with ${\cal R}_M=1$ Major operator controlling $M=20$ units and ${\cal R}_m= (N-M)/m$ Minor operators controlling $m$ units each. The left panel is normalized such that `1' corresponds to the average maximum dispatch $\E[\max_{t \le T^\circ} |\hat{\alpha}^{(M)}_t|]$ of the Major operator with unit-size minor operators, $M=20,m=1$, so the blue curve starts out at 1.}
        \label{fig:major_minor_profit_m}
\end{figure}

As the number of Minor operators is reduced, their costs shrink dramatically (middle panel of Figure \ref{fig:major_minor_profit_m}) while their revenue is very stable, creating a major gain on a per-unit basis. In tandem, the profitability of the Major operator rises essentially linearly in $m$, though at a much smaller rate. Hence the primary benefactors of a market consolidation are the Minor operators.  Notably,  The right panel of Figure \ref{fig:major_minor_profit_m} shows that there is a loss in efficiency comparing Tri-opoly $m=6$ with a Duopoly $m=12$. Thus, there is no incentive for two smaller operators that are of size 6 to merge into a single duopolist of size 12. A coalition of size 6 is stable in the sense that further consolidation is uneconomic.

\section{Asymptotic Behavior of a Large $N$-player Game}\label{sec:asymptotics}

In this section, we study the asymptotic behavior of the Nash equilibrium in the homogeneous game as the number of operators $N \to \infty$. This provides insights regarding the behavior of very large markets, and moreover allows to quantify how the number of operators $N$ affects prices. Intuitively, the arbitrage driven behavior of each BESS leads to a completely flat price in the limit. Moreover, the competitive effects of over-dispatching depress prices and make everyone unprofitable. 

To do the asymptotic analysis,  we use a regular perturbation approach, expanding relevant quantities characterizing $(V, \hat\alpha)$ in \eqref{homo_ode_system} in powers of $\frac{1}{N}$ and computing the solutions up to a prescribed order. More precisely, letting $\epsilon = \frac{1}{N}$, we expand  $\tilde p_i(t,\epsilon)$\footnote{In what follows, whenever a quantity depends on $\epsilon = 1/N$, we will include $\epsilon$ explicitly as an argument to emphasize this dependence.} and $\tilde r_i(t,\epsilon)$, the solution to the system \eqref{homo_ode_system}, as 
\begin{align}
    \tilde p_i(t,\epsilon) = \tilde p_i^{(0)}(t) + \tilde p_i^{(1)}(t)\epsilon + \cO(\epsilon^2),\label{p_expansion}\\
    \tilde r_i(t,\epsilon) = \tilde r_i^{(0)}(t) + \tilde r_i^{(1)}(t)\epsilon + \cO(\epsilon^2).\label{r_expansion}
\end{align}
We then characterize $\tilde p_i^{(0)}, \tilde p_i^{(1)}, \tilde r_i^{(0)}, \tilde r_i^{(1)}$, which in turn show the asymptotic behavior of the equilbrium control $\hat \alpha^i_t$ and the state dynamics $\hat{\bS}_t$.   

By \eqref{homo_optimal control}, the equilibrium control depends on $\tilde p_2(t, \epsilon), \tilde p_4(t, \epsilon), \tilde p_5(t, \epsilon), \tilde r_2(t, \epsilon)$. Moreover, \eqref{homo_ode_system} shows that the ODEs for $\tilde p_4(t, \epsilon)$ and  $\tilde p_5(t, \epsilon)$ are coupled with those for $\tilde p_6(t, \epsilon)$ and $\tilde p_7(t, \epsilon)$. In addition, the ODE for  $\tilde p_2(t, \epsilon)$ (and $\tilde r_2(t, \epsilon)$) involves  $\tilde p_3(t, \epsilon)$ (and $\tilde r_3(t, \epsilon)$). Accordingly, we focus on the collection $$\tilde p_2(t, \epsilon),\; \tilde p_3(t, \epsilon), \; \tilde p_4(t, \epsilon), \; \tilde p_5(t, \epsilon), \; \tilde p_6(t, \epsilon), \;\tilde p_7(t, \epsilon), \;\tilde r_2(t, \epsilon), \; \tilde r_3(t, \epsilon).$$ The results are summarized as follows. All proofs are presented in Appendix~\ref{appen:asymptotics}. 

\begin{proposition} \label{p0_sol}
(i)    The leading order coefficients $\tilde p_i^{(0)}(t)$ in the expansion of $\tilde p_i(t, \epsilon)$ satisfy \\ $\tilde p_3^{(0)}  = \tilde p_5^{(0)} = \tilde p_6^{(0)} = \tilde p_7^{(0)} \equiv 0$, as well as 
\begin{equation}\label{p2^0 sol}
    \tilde p_2^{(0)}(t) = \int_t^T e^{-\kappa(s-t)}(c_3(T-s)+c_4)b(s)\,ds,
\end{equation}
\begin{equation} \label{p4^0 sol}
        \tilde p_4^{(0)}(t) = \sqrt{\frac{d^2c_3}{4d-4c_2}}\frac{\big(c_4+\sqrt{\frac{d^2c_3}{4d-4c_2}}\big)e^{2\sqrt{\frac{(4d-4c_2)c_3}{d^2}(T-t)}}+c_4-\sqrt{\frac{d^2c_3}{4d-4c_2}}}{\big(c_4+\sqrt{\frac{d^2c_3}{4d-4c_2}}\big)e^{2\sqrt{\frac{(4d-4c_2)c_3}{d^2}(T-t)}}-c_4+\sqrt{\frac{d^2c_3}{4d-4c_2}}}.
\end{equation}

\noindent (ii) The leading order coefficients  $\tilde r_i^{(0)}(t)$ in the expansion of $\tilde r_i(t, \epsilon)$, $i = 2, 3$ satisfy $\tilde r_3^{(0)}(t) \equiv 0 $ and
    \begin{align}\label{r_2^0 sol}
       \tilde r_2^{(0)}(t) &=  2\int_{t}^{T} (c_3(T-s)+c_4)a(s)-c_3 \zeta(s)\,ds + 2\int_t^T\int_s^T \kappa e^{-\kappa(u-s)}\theta(s)(c_3(T-u)+c_4)b(u)\,du\,ds \nonumber\\
       &\quad - 2c_4 \zeta(T).
    \end{align}

\noindent(iii) The first order corrections $\tilde p_i^{(1)}(t)$ to $\tilde p_i(t, \epsilon)$, $i = 4, 5$ solve the coupled ODE system
\begin{equation}\label{p1_sys_p4}
    \left\{
    \begin{aligned}
        - \dot{\tilde{p}}_4^{(1)}(t) &= \Big( \frac{8}{d} - \frac{8c_{2}}{d^{2}} \Big) 
\big( \tilde p_4^{(0)}(t) \big)^{2} 
+ \Big( \frac{12}{d} - \frac{8c_{2}}{d^{2}} \Big) 
\tilde p_4^{(0)}(t)\, \tilde p_5^{(1)}(t) 
+ \Big( \frac{8}{d} - \frac{8c_{2}}{d^{2}} \Big) 
\tilde p_4^{(0)}(t)\, \tilde p_4^{(1)}(t), \\
 -\dot{\tilde{p}}_5^{(1)}(t) &= 
 \Big(\frac{4}{d} - \frac{4c_{2}}{d^{2}}\Big) (\tilde p_4^{(0)}(t))^{2},
    \end{aligned}
    \right.
\end{equation}
with terminal conditions $\tilde p_4^{(1)}(T) = \tilde p_5^{(1)}(T) = 0$. Moreover, $\tilde p_6^{(1)}(t) = \tilde p_7^{(1)}(t) \equiv 0$.

\end{proposition}
 
With the characterizations of  $\tilde p_i^{(0)}(t)$, $\tilde p_i^{(1)}(t)$, and $\tilde r_i^{(0)}(t)$ in hand, we now analyze the expansions for the individual $\hat\alpha_t^{i,\epsilon} = \hat\alpha^{i,\epsilon}(t,Q_t,\hat \bS_t)$ and aggregate control processes. 

\begin{theorem}\label{regional_control}
    In terms of $\epsilon$, the equilibrium control process $\hat\alpha_t^{i,\epsilon}$ of agent $i$ admits the expansion 
\begin{align}\label{regional_optimal}
\hat\alpha_t^{i,\epsilon}
&= \frac{2}{d}\tilde p_4^{(0)}(t)\big(\bar S_t-\hat S_t^i\big)  + \epsilon\Big[
\Big(1-\frac{2\tilde p_2^{(0)}(t)}{c_1}\Big)Q_t
+ \frac{2}{d}\big(\tilde p_5^{(1)}(t)-\tilde p_4^{(1)}(t)\big)\hat S_t^i
- \Big(\frac{1}{c_1}\tilde r_2^{(0)}(t)+\frac{\bar{P}}{c_1}\Big)
\Big] \nonumber\\
&\quad + \epsilon^2\Big(
\frac{2}{d}\tilde p_4^{(1)}(t)
-\frac{2(c_1+d)}{dc_1}\tilde p_5^{(1)}(t)
-\frac{2}{c_1}\tilde p_4^{(0)}(t)
\Big)\mathbf 1_{N-1}^T\hat \bS_t^{-i}
+ \cO(\epsilon^2).
\end{align}
where $\bar{S}_t = \frac{1}{N}\sum_{j=1}^{N} \hat S_t^j$ and $\hat \bS_t^{-i} = [\ldots,\hat S_t^{i-1},\hat S_t^{i+1},\ldots]^\top$. 

At equilibrium, the aggregate control process $\sum_{i=1}^N\hat{\alpha}_t^{i,\epsilon}$ is 
    \begin{equation}\label{agg_optimal}
\sum_{i=1}^N \hat{\alpha}_t^{i,\epsilon}
=
\Big(1-\frac{2\tilde p_2^{(0)}(t)}{c_1}\Big) Q_t
- \frac{2}{c_1}\big(c_3(T-t)+c_4\big)\bar S_t
- \Big(\frac{1}{c_1}\tilde r_2^{(0)}(t)+\frac{\bar{P}}{c_1}\Big)
+ \cO(\epsilon).
\end{equation}
\end{theorem}

Notice that in the expansion of $\hat\alpha_t^{i,\epsilon}$ we include the term $\epsilon^2\big(
\frac{2}{d}\tilde p_4^{(1)}(t)
-\frac{2(c_1+d)}{dc_1}\tilde p_5^{(1)}(t)
-\frac{2}{c_1}\tilde p_4^{(0)}(t)
\big)\mathbf 1_{N-1}^T\hat \bS_t^{-i}$. Although this term is of order $\epsilon^2$ for each individual agent, it is multiplied by $1_{N-1}^T\hat\bS_t^{-i} = \sum_{j \neq i} \hat S_t^j$, and therefore contributes $\mathcal{O}(1)$ after aggregation. 

Substituting the aggregate equilibrium control \eqref{agg_optimal} into the expression for the price
$\hat P_t^{\epsilon}=\bar{P}
- c_1\big(Q_t-\sum_{j=1}^N \hat \alpha_t^{j,\epsilon}\big)$
yields the following

\begin{corollary}\label{cor:price}
    The equilibrium price process $\hat{P}^{\epsilon}_t$ admits the asymptotic expansion
    \begin{equation}\label{equi_price}
\hat P_t^{\epsilon}
=
-2\tilde p_2^{(0)}(t) Q_t
-2\big(c_3(T-t)+c_4\big)\bar S_t
- \tilde r_2^{(0)}(t)
+ \cO(\epsilon).
\end{equation}
\end{corollary}

To analyze the asymptotic behavior of the equilibrium price as $N \to \infty$, it remains to study the dynamics of the averaged SOC $(\bar S_t)$.
\begin{proposition}
\label{bar_S_t_limit}
The average SOC process $\bar S_t$ converges in $L^2$:
    \begin{equation}\label{eq:barS-limit}
         \bar{S}_t \xrightarrow{L^2} S_0 + \int_0^t (a(s) + b(s)Q_s)\,ds + \sigma\rho W_t^0.
    \end{equation}    
\end{proposition} 
\noindent Observe that the sole stochastic driver in \eqref{eq:barS-limit} is the common noise $(W^0_t)$; the individual shocks $W^i$ vanish in the limit due to averaging across agents.

Economically, the limit $N \to \infty$ only makes sense when the operators are arbitrageurs with $\mu^i_t \equiv 0$; otherwise there would be infinite hybrid supply (even if the other quantities have well-defined limits). In the case $a(t) = b(t) = 0$, $\forall t$, we have that $\tilde{p}^{(0)}_2 = 0$ in \eqref{p2^0 sol} and $\tilde{r}^{(0)}_2$ in \eqref{r_2^0 sol} simplifies to $-2 c_3 \int_t^T \zeta(s) ds - 2 c_4 \zeta(T)$, making the asymptotic expansion in \eqref{equi_price} reduce to
$$
\hat P_t^{\epsilon} = -2\big(c_3(T-t)+c_4\big)(S_0 + \sigma \rho W^0_t)+ 2 c_3 \int_t^T \zeta(s) ds + 2c_4\zeta(T) + \cO(\epsilon).
$$

\section{Conclusion}\label{sec:conclusion}

In this article we construct a game-theoretic model for strategic interaction among battery energy storage systems during intraday dispatch. By explicitly modeling price formation we shift from the conventional view of BESS as a marginal price taker to BESS having significant market power, reflecting their growing aggregate dominance and the ensuing power price spread cannibalization. Moreover, we directly consider both pure-arbitrageur BESS and hybrid BESS that have their own generation source, building in a coupling between the net supply $Q_t$ and the self-generation sources $\mu^i_t$. Embedding the above features into the linear-quadratic paradigm guarantees a high degree of tractability for a generic heterogeneous multi-agent setup, for the first time offering a quantitative testbed to explore competitive BESS dispatch. Our flexible framework permits to directly study the role of asset size (Section \ref{sec:major-minor}) or of the mix between Arbitrageur and Hybrid operator types (Section \ref{sec:hybrid-arb}). We are also able to quantify the impact of additional players entering the market and the asymptotic behavior as $N \to \infty$.

While the linear-quadratic structure provides analytic tractability through Riccati equations and semi-explicit solutions, it abstracts away from important nonlinearities such as storage inefficiencies, degradation effects, and nonlinear cost penalties. In these non-LQ problems, the associated HJB equations become analytically intractable and numerically challenging. Recent advances in deep learning via neural network offer alternatives, such as fictitious play algorithms for finite-agent stochastic games  \citep{han2020deep,hu2021deep}. A related key extension is to allow for nonlinear inverse demand curves. In particular, the merit-order mechanism for power prices is intrinsically convex, creating an asymmetry between prices soaring when net load is large and price elasticity being much less when net load is low. Consequently, BESS will asymmetrically flatten price peaks and only marginally fill price valleys, leading to well-documented reduction in average price \citep{lamp2022large,karadumaneconomics}.

A different open problem is to address the Major-Minor non-LQ game, focusing on the interaction between a Major operator and a a continuum of Minor agents. The Major operator is a price maker while the Minor ones are price takers. Recent contributions \citep{feron2020price,fujii2022equilibrium, bichuch2025stackelberg} have studied related Major-Minor formulations in energy markets, as well as the Stackelberg mean field game version. Finally, the present setup postulated standard ``symmetric-type'' sub-game perfect Nash equilibria. Many more equilibria can be in principle envisioned, for example breaking the symmetry even when all operators are homogeneous.

\bibliographystyle{plainnat}

\appendix
\section{Appendix}
\subsection{Proofs for Section \ref{sec:HJB_sys}}\label{General Theorem Proof}
\begin{proof}[Proof of Theorem \ref{General Theorem}]
Fixing a Markovian control profile $\bar{\balpha} ^{-i} = [ \bar\alpha^1,\ldots,\bar \alpha^{i-1},\bar\alpha^{i+1},\ldots,\bar\alpha^N]$ of all other agents, agent $i$ faces a standard stochastic control problem, and her optimal cost is given by the best-response value function $ V^i(t,q,\bs; \bar{\balpha}^{-i})$ in \eqref{general_value_function}.  By dynamic programming principle (see, e.g. \citep{pham2009continuous}), her specific value function  $ V^i(t,q,\bs; \bar{\balpha}^{-i})$  then satisfies the HJB equation \eqref{general HJBE-PDE} with $\hat\alpha^j$ replaced by $\bar \alpha^j$, $j \neq i$. The respective best-response feedback control for agent $i$ is then obtained by minimizing the Hamiltonian which depends on $ V^i(t,q,\bs; \bar{\balpha}^{-i})$ and $\bar\balpha^{-i}$. 

At a Markovian Nash equilibrium $\hat{\balpha}$, each component $\hat\alpha^i$ is a best response to $\hat{\balpha}^{-i}$ for $i=1,2,\ldots,N$. Aggregating, the collection of best-response value functions $\{V^i(t,q,\bs; \bar{\balpha}^{-i})\}_{i=1}^N$ solves the coupled HJB system \eqref{general HJBE-PDE}. Define the Hamiltonian of agent $i$, $H(\alpha^i;\balpha^{-i})$ 
\begin{equation*}
    H(\alpha^i; \balpha^{-i}) := \alpha^i \big( \bar{P}^i -  c_1^i (q - \bw_{i}^{\top} \balpha ) \big) 
    + c_2^i (\alpha^i)^2 + \partial_{s^i} V^i \cdot \alpha^i,
\end{equation*}
which summarizes the dependence of agent $i$'s Hamiltonian on its control $\alpha^i$. Differentiating with respect to $\alpha^i$ and  setting $\frac{\partial H}{\partial \alpha^i} =0$ gives the first-order condition 
\begin{equation}\label{eq:alpha-foc}
   \hat \alpha^i  + \frac{ c_1^i \bw_i^\top \balpha }{c_1^i w_{ii} + 2 c_2^i}  = \frac{ c_1^i q}{c_1^i w_{ii} + 2 c_2^i} 
- \frac{\partial_{s^i} V^i}{c_1^i w_{ii} + 2 c_2^i} 
- \frac{\bar{P}^i}{c_1^i w_{ii} + 2 c_2^i}.
\end{equation}
Note that the denominators $(c_1^i w_{ii} + 2 c_2^i)$ are precisely the $i$-th entry of $\bd$ in \eqref{d,U,u}. At a Nash equilibrium $\hat\balpha$, the above condition must hold simultaneously for each $i = 1, 2, \ldots, N$. Assuming that $ \bI_N +  \diag \bd ^{-1}\diag {\bc_1}\bW$ is invertible and vectorizing \eqref{eq:alpha-foc}, the equilibrium control profile $\hat\balpha$ is therefore characterized by the linear system
\begin{equation*}
   (\bI_N +  \diag \bd ^{-1}\diag {\bc_1}\bW)\hat\balpha = \diag \bd ^{-1} (\bc_1 q -[\partial_{s^1}V^1,\ldots,\partial_{s^N}V^N]^{\top} -\ubarP).
\end{equation*}
This yields the representation \eqref{opt_alpha} for agent $i$'s equilibrium control  $\hat \alpha^i$ (recall the definition of $\bM$ in \eqref{d,U,u}).

We next derive the ODE system~\eqref{general_ode_sys}--\eqref{general_terminal} that characterizes the game value $V^i(t, q, \bs)$ of agent $i$. Substituting \eqref{opt_alpha} into equation \eqref{general HJBE-PDE}, the HJB equation for agent $i$ becomes
\begin{align} \label{general_HJBE}
- \partial_t V^i =\
&\underbrace{\bar{P}^i\, \hat\alpha^i}_{\footnotesize\textcircled{0}}
-\underbrace{ c_1^i \hat\alpha^i
\big(q- \bw_{i}^\top\hat\balpha\big) }_{\footnotesize\textcircled{1}} +\underbrace{c_2^i (\hat\alpha^i)^2}_{\footnotesize\textcircled{2}} + \underbrace{c_3^i (s^i - \zeta^i(t))^2}_{\footnotesize\textcircled{3}} 
+ \underbrace{\partial_q V^i \cdot \kappa(\theta(t) - q)}_{\footnotesize\textcircled{4}} \nonumber\\
&+ \underbrace{\sum_{j=1}^N \partial_{s^j} V^i 
\big(\hat\alpha^j + a^j(t) + b^j(t) q\big)}_{\footnotesize\textcircled{5}} + \underbrace{\frac{1}{2}\mathrm{Tr}(\Sigma\Sigma^\top \partial_{q,\bs}^2V^i)}_{\footnotesize\textcircled{6}}.
\end{align}
We adopt the quadratic ansatz \eqref{ansatz}
and compute the partial derivatives (with $\dot{p}(t)$ generically denoting the time-derivative of the function $p(t)$)
\begin{align*}
    \partial_t V^i(t,q,\bs) &=  \dot{p}^{i}_{0}(t) q^2 + 2(\dot{\bp}^i(t))^\top \bs q + \bs^\top \dot{\bP}^i(t) \bs + \dot{r}_0^i(t) q + (\dot{\br}^i(t))^\top \bs + \dot{u}^i(t), \\
\partial_q V^i(t,q,\bs) &= 2 p^{i}_{0}(t) q + 2 \bp^i(t)^\top  \bs + r_0^i(t),\quad
\partial_{\bs} V^i(t,q,\bs) = 2 \bp^i(t)q + 2  \bP^i(t) \bs + \br^i(t),\\ 
\partial_{qq} V^i(t,q,\bs) &= 2 p^{i}_{0}(t),\quad
\partial_{q\bs} V^i(t,q,\bs) = 2 \bp^i(t), \quad
\partial_{\bs \bs} V^i(t,q,\bs) = 2  \bP^i(t).
\end{align*}
To alleviate notation, we introduce the intermediate quantities in \eqref{intermediates}. Then under the quadratic ansatz, agent $i$'s optimal control $\hat \alpha^i$ takes the affine form
\begin{equation*}
   \hat \alpha^i = k_1^i(t)\,q - \bk_2^i\,\bs - k_3^i(t),
\end{equation*}
and the interaction term in \textcircled{1}, $q-\bw_{i}^\top\balpha$, becomes 
\begin{equation*}
    q-\bw_{i}^\top\balpha = k_4^{i}(t)\,q + \bk_5^{i}(t)\,\bs + k_6^{i}(t).
\end{equation*}
Substituting the ansatz and these expressions into \eqref{general_HJBE}, we obtain seven labeled terms whose explicit forms are listed below:
\begin{align*}
\textcircled{0}
   &= \bar{P}^i \big(k_1^{i}(t)\,q -\bk_2^i(t)\,\bs - k_3^i(t)\big),\nonumber \\
    \textcircled{1} 
&=  -c_1^i\,k_1^i(t)k_4^i(t)\,q^2 
- c_1^i \big(k_1^i(t)\bk_5^i(t)-\bk_2^i(t)k_4^i(t)\big)\, \bs q + \bs^\top c_1^i \bk_2^i(t)^{\top}\bk_5^i(t)\, \bs 
 \\
&\quad + c_1^i \big( k_3^i(t)k_4^i(t)-k_1^i(t)k_6^i(t) \big)\, q + c_1^i  \big(\bk_2^i(t)k_6^i(t)+k_3^i(t)\bk_5^i(t) \big)\, \bs + c_1^i k_3^i(t)k_6^i(t) , \\
\textcircled{2}
&= c_2^i \big(k_1^i(t)\big)^2 q^2 
- 2 c_2^i k_1^i(t)\bk_2^i(t)\, \bs q + \bs^\top c_2^i\,\bk_2^i(t)^{\top}\bk_2^i(t)\, \bs 
 - 2 c_2^i k_1^i(t)k_3^i(t)\, q  + 2 c_2^i \bk_2^i(t)k_3^i(t) \bs \\
&\quad + c_2^i \big(k_3^i(t)\big)^2, \\
\textcircled{3} &= \bs^\top c_3^i \be_i \be_i^\top \bs 
- 2 c_3^i \zeta^i(t)\be_i^\top \bs 
+ c_3^i (\zeta^i(t))^2, \\
\textcircled{4} 
&= -2 \kappa \, p^{i}_{0}(t) \, q^2 
- 2 \kappa \, \bp^i(t)^\top \bs \, q
 + \kappa \big( 2 \theta(t) \, p^{i}_{0}(t) - r_0^i(t) \big) q + 2 \kappa \, \theta(t) \, \bp^i(t)^\top \bs 
+ \kappa \, r_0^i(t) \, \theta(t), \\
\textcircled{5} 
&=  2 {\bp^i(t)}^\top (\bk_1(t)+\bb(t)) \, q^2 
+ 2 \big( (\bk_1(t)+\bb(t))^\top  \bP^i(t) - {\bp^i(t)}^\top \bK_2(t) \big) \bs q - 2 \bs^\top { \bP^i(t)}^\top \bK_2(t) \bs \\ &\quad + \big( {\br^i(t)}^\top (\bk_1(t)+\bb(t)) - 2 {\bp^i(t)}^\top (\bk_3(t)-\ba(t)) \big) q 
- \big( {\br^i(t)}^\top \bK_2(t) + 2 (\bk_3(t)-\ba(t))^\top  \bP^i(t) \big) \bs \\
&\quad
- {\br^i(t)}^\top (\bk_3(t)-\ba(t)),  \\
\textcircled{6} 
&= \mathrm{Tr}\Big(\Sigma\Sigma^\top \big[\begin{smallmatrix}
p^{i}_{0}(t) & \bp^i(t)^\top \\
\bp^i(t) &  \bP^i(t)
\end{smallmatrix}\big]\Big).
\end{align*}
Collecting terms according to their polynomial degrees in $(q,\bs)$,
namely $q^2$, $\bs q$, $\bs^{\top}(\cdot)\bs$, $q$, $\bs$, and constants,
and matching coefficients on both sides of \eqref{general_HJBE},
we obtain the coupled ODE system \eqref{general_ode_sys}.

It remains to enforce the terminal condition. The terminal cost is given explicitly as
\begin{equation*}
V^i(T, q, \bs) 
= \bs^\top c_4^i \be_i \be_i^\top \bs 
- 2 c_4^i \zeta^i(T) \be_i^\top \bs 
+ c_4^i (\zeta^i(T))^2,
\end{equation*}
whereas evaluating the ansatz at $t=T$ gives
\begin{equation*}
V^i(T, q, \bs) =  p^{i}_{0}(T) q^2
+ 2 \bp^i(T)^\top \bs q 
+ \bs^\top  \bP^i(T) \bs
 + r_0^i(T) q 
+ \br^i(T)^\top \bs 
+ u^i(T).
\end{equation*}
Equating these two expressions yields \eqref{general_terminal}.
\end{proof}

\begin{proof}[Proof of Proposition \ref{general_wp}]\label{hetero_wp}

The coupled ODE system \eqref{general_ode_sys} has a triangular structure. Note that $\bP^i$, $i=1,\ldots,N$ form a closed Riccati system. After applying time reversal on $[0,T]$, this becomes initial-value problem. Once the existence solution of $\bP^i$ has been established on some interval in $[0,T]$, the equations for $p_0^i, \bp^i, r_0^i,$ and $u^i$ form an affine linear system with coefficients depending on $\bP^i$, and classical ODE theory yields existence and uniqueness on the same interval. Hence, the well-posedness of the full coefficient system reduces to that of the Riccati system for $\bP^i$, $i=1,\ldots,N$.

We introduce the block-diagonal matrix $\calP(t) := \text{diag}(\bP^1(t), \ldots, \bP^N(t)) \in \bbR^{N^2 \times N^2} $, whose $i$-th diagonal block is $\bP^i(t)$. By the ODE of $\bP^i(t)$ in \eqref{general_ode_sys}, we have 
\begin{equation}\label{calP_ode}
    -\dot \calP(t) = \calC_3 - \calP(t)\calJ\calS\calP(t)\calJ - \calJ\calP(t)\calS^\top\calJ\calP(t) +\calG(\calP(t)),\quad \calP(T) = \calC_4,
\end{equation}
where $\calJ = \bJ_N \otimes \bI_N$, $\calG(\calP) = \text{diag}(\sum_{j,l=1}^N \bP^j \bB^1\bP^l,\ldots,\sum_{j,l=1}^N \bP^j \bB^N\bP^l)$, $\calC_3$, $\calS$, $\{\bB^i\}_{i=1}^N$,and $\calC_4$ defined through \eqref{QSK_0} and \eqref{B_entries}.
 Following the idea of \citep{papavassilopoulos2003existence}, we first set $\bar\calP(t) = \calP(T-t)$, and we can convert \eqref{calP_ode} to the initial value problem on $[0,T]$ 
 \begin{equation}\label{cal_barP_ode}
     \dot{\bar{\calP}}(t) = \calC_3 - \bar\calP(t)\calJ\calS\bar\calP(t)\calJ - \calJ\bar\calP(t)\calS^\top\calJ\bar\calP(t) +\calG(\bar\calP(t)),\quad \bar\calP(0) = \calC_4,
 \end{equation}
 Introducing any matrix induced norm $\|\cdot\|$,  we can show that $\|\calP(t)\| = \max_{1\leq i \leq N}\|\bP^i(t)\|$ and also 
 \begin{align*}
     \|\calG(\calP(t))\|  = \max_{1\leq i\leq N}\| \sum_{j,l=1}^N \bP^j \bB^i\bP^l\|\leq N^2 \max_{1\leq i\leq N}\|\bB^i\|\|\calP\|^2.
 \end{align*}
 Then if $\bar\calP(t)$ is a solution of \eqref{cal_barP_ode} on $[0,t')$ where $t'\leq T$, then
 \begin{align}
     \|\dot{\bar{\calP}}(t)\| \leq N^2\beta \|\bar\calP(t)\|^2 + \|\calC_3\|,
 \end{align}
 where $\beta = \|\calS\|+\|\calS^\top\| + \max_{1\leq i\leq N}\| \bB^{i} \|$. By the third case of the Proposition in \citep{papavassilopoulos2003existence}, a sufficient condition of existence of the solution is that if $T < \frac{1}{2N\sqrt{\|\calC_3\|\beta}}\big[\pi - 2\tan^{-1}\big(\frac{N\beta\|\calC_4\|}{\sqrt{\|\calC_3\|\beta}}\big)\big]$, then the solution of \eqref{calP_ode} exists and moreover,
 \begin{equation*}
     \|\calP(t)\| \leq \frac{\sqrt{\|\calC_3\|\beta}}{N\beta}\tan\Big(N\sqrt{\|\calC_3\|\beta} \cdot (T-t)+\tan^{-1}\big(\frac{N\beta\|\calC_4\|}{\sqrt{\|\calC_3\|\beta}}\big)\Big), \quad t \in[0,T].
 \end{equation*}
 The uniqueness of the solution is a direct consequence of the local Lipschitz property of the vector field induced by the linear–quadratic structure of the ODE system. On a compact time interval, standard existence–uniqueness results for ODEs therefore imply that, once existence is established, the solution is necessarily unique.
\end{proof}

\begin{proof}[Proof of Theorem \ref{Homo_case}]
In the homogeneous-agent setting \eqref{eq:homo-condition}, the value functions and equilibrium controls are permutation invariant. More precisely, for any permutation \( \pi: \{1, \dots, N\} \to \{1, \dots, N\} \), the value function $\tilde V^i$ and the equilibrium control $\hat\alpha^i$ satisfy:
\[
\tilde V^i(t, q,\bs) = \tilde V^{\pi(i)}(t, q,\pi(\bs)), \quad \hat\alpha^i(t, q,\bs) = \hat\alpha^{\pi(i)}(t, q,\pi(\bs)),
\]
where \( \bs = ( s^1, \dots, s^N) \in \bbR^{N} \), and \( \pi(\bs) \) denotes the permuted vector $\pi(\bs) = (s^{\pi(1)}, \ldots, s^{\pi(N)})^\top$. Furthermore, fixing an index $i$ and splitting $\bs$ as $\bs = [s^i,\bs^{-i}]$, the symmetry implies that the remaining components $\bs^{-i}$ enter only through permutation-invariant combinations. Hence, for any permutation
\(\pi'\) acting on the indices $\{1, \ldots, N\} \setminus \{i\}$, we have 
\[
\tilde V^i(t, q, s^i, \bs^{-i}) = \tilde V^i(t, q, s^i, \pi'(\bs^{-i})), \quad
\hat\alpha^{i}(t, q, s^i, \bs^{-i}) = \hat\alpha^{i}(t, q, s^i, \pi'(\bs^{-i})),
\]
where
\[
\pi'(\bs^{-i}) := \big( s^{\pi'(1)}, \dots, s^{\pi'(i-1)}, s^{\pi'(i+1)}, \dots, s^{\pi'(N)} \big)^\top.
\]
This permutation symmetry motivates the reduced quadratic ansatz \eqref{homoHJBE}, which is a symmetric specialization of the general quadratic form \eqref{ansatz}.

We now compute the derivatives of $\tilde V^i$ given by \eqref{homoHJBE}. A direct calculation yields
\begin{equation*}
\begin{aligned}
    \partial_t \tilde V^i &= 
\dot{\tilde p}_1(t)\, q^2 + 2 \dot{\tilde p}_2(t)\, s^i q + 2 \dot{\tilde p}_3(t)\, \mathbf{1}_{N-1}^\top \bs^{-i} q + \dot{\tilde p}_4(t)\, (s^i)^2 
+ 2 \dot{\tilde p}_5(t)\, \mathbf{1}_{N-1}^\top s^i \bs^{-i} \\
&\quad + (\bs^{-i})^\top \dot{B}(t)\, \bs^{-i} + \dot{\tilde r}_1(t)\, q + \dot{\tilde r}_2(t)\, s^i + \dot{\tilde r}_3(t)\, \mathbf{1}_{N-1}^\top \bs^{-i} + \dot{\tilde{u}}(t), \\
\partial_q \tilde V^i &= 2 \tilde p_1(t)\, q + 2 \tilde p_2(t)\, s^i + 2 \tilde p_3(t)\, \mathbf{1}_{N-1}^\top \bs^{-i} + \tilde r_1(t), \\
\partial_{s^i} \tilde V^i &= 2 \tilde p_2(t)\, q + 2 \tilde p_4(t)\, s^i + 2 \tilde p_5(t)\, \mathbf{1}_{N-1}^\top \bs^{-i} + \tilde r_2(t), \\
\partial_{\bs^{-i}} \tilde V^i &= 2 \tilde p_3(t)\, \mathbf{1}_{N-1} q + 2 \tilde p_5(t)\, \mathbf{1}_{N-1} s^i + 2 ((\tilde p_7(t) - \tilde p_6(t)) \bI_{N-1} + \tilde p_6(t)\, \bJ_{N-1})\, \bs^{-i} \\ &\quad + \tilde r_3(t)\, \mathbf{1}_{N-1}, \\
\partial_{qq} \tilde V^i &= 2 \tilde p_1(t),\,
\partial_{s^i s^i} \tilde V^i = 2 \tilde p_4(t),\,
\partial_{\bs^{-i} \bs^{-i}} \tilde V^i = 2 ((\tilde p_7(t) - \tilde p_6(t)) \bI_{N-1} + \tilde p_6(t)\, \bJ_{N-1}).
\end{aligned}
\end{equation*}

Next, substituting $\bw = \mathbf{1}_N$ into \eqref{simp_M} and then inserting the resulting $\bM$ into equation \eqref{opt_alpha}, we obtain the equilibrium control in the homogeneous case:
\begin{align}\label{homo_opt}
     \hat{\alpha}_i &=  \frac{c_1 \big(q + \frac{1}{d} \sum_{k=1}^N \partial_{s^k} \tilde V^k + \eta_{0} \big)}{d\eta_{1}}- \frac{\partial_{s^i} \tilde V^i}{d} - \frac{\bar{P}}{d}.
\end{align}

Substituting \eqref{homo_opt} into \eqref{general HJBE-PDE}, the equation expands into the following labeled components:
\begin{align}
-\partial_t \tilde V^i
=\;&
\underbrace{
\bar{P}\Big(
\frac{c_1\big(q+\frac{1}{d}\sum_{k=1}^N \partial_{s^k}\tilde V^k+\eta_{0}\big)}{d\eta_{1}}
-\frac{\partial_{s^i}\tilde V^i}{d}
-\frac{\bar{P}}{d}
\Big)
}_{\footnotesize\textcircled{0}}
\nonumber\\
&-
\underbrace{
c_1\Big(
\frac{c_1\big(q+\frac{1}{d}\sum_{k=1}^N \partial_{s^k}\tilde V^k+\eta_{0}\big)}{d\eta_{1}}
-\frac{\partial_{s^i}\tilde V^i}{d}
-\frac{\bar{P}}{d}
\Big)
\Big(
\frac{q+\frac{1}{d}\sum_{k=1}^N \partial_{s^k}\tilde V^k+\eta_{0}}{\eta_{1}}
\Big)
}_{\footnotesize\textcircled{1}}
\nonumber\\
&+
\underbrace{
c_2\Big(
\frac{c_1\big(q+\frac{1}{d}\sum_{k=1}^N \partial_{s^k}\tilde V^k+\eta_{0}\big)}{d\eta_{1}}
-\frac{\partial_{s^i}\tilde V^i}{d}
-\frac{\bar{P}}{d}
\Big)^2
}_{\footnotesize\textcircled{2}} 
+
\underbrace{
c_3 (s^i-\zeta(t))^2
}_{\footnotesize\textcircled{3}}
+
\underbrace{
\partial_q \tilde V^i \,\kappa(\theta(t)-q)
}_{\footnotesize\textcircled{4}}
\nonumber\\
&+
\underbrace{
\sum_{j=1}^N \partial_{s^j}\tilde V^i
\Big(
\frac{c_1\big(q+\frac{1}{d}\sum_{k=1}^N \partial_{s^k}\tilde V^k+\eta_{0}\big)}{d\eta_{1}}
-\frac{\partial_{s^j}V^j}{d}
-\frac{\bar{P}}{d}
+a(t) + b(t)q
\Big)
}_{\footnotesize\textcircled{5}}
\nonumber\\
&+ \underbrace{\frac{1}{2}\mathrm{Tr}(\Sigma\Sigma^\top \partial_{q,\bs}^2\tilde V^i)}_{\footnotesize\textcircled{6}}
\end{align}
To simplify the algebra, we introduce the intermediate quantities $\{g_1, \ldots, g_8, \bh_1, \bh_2, \bh_4, \bh_5, \bh_6, \bh_8 \in \bbR^N\}$, rewriting the control and interaction terms in the compact forms
\begin{align*}
    \frac{c_1 \big(q + \frac{1}{d} \sum_{k=1}^N \partial_{s^k} \tilde V^k + \eta_{0} \big)}{d\eta_{1}}- \frac{\partial_{s^i} \tilde V^i}{d} - \frac{\bar{P}}{d}  
     &= g_1(t)q + g_2(t) s^i + g_3(t) \, \mathbf{1}_{N-1}^\top \bs^{-i} 
+ g_4(t),\\
    \frac{q + \frac{1}{d} \sum_{k=1}^{N} \partial_{s^k} \tilde V^k + \eta_{0}}{\eta_{1}}
&= g_5(t) q + g_6(t) s^i + g_7(t) \, \mathbf{1}_{N-1}^\top \bs^{-i} + g_8(t).
\end{align*}
With the additionally defined $\bH_3, \bH_7 \in \bbR^{N \times (N-1)}$:
\begin{equation*}
     \bH_3(t) := \begin{bmatrix}
2\tilde p_5(t)\, \mathbf{1}_{N-1}^\top \\ 2((\tilde p_7(t) - \tilde p_6(t)) \bI_{N-1} + \tilde p_6(t)\, \bJ_{N-1})
\end{bmatrix}, \; 
\bH_7(t) := \begin{bmatrix}
g_3(t)\, \mathbf{1}_{N-1}^\top \\
(g_2(t) - g_3(t)) \bI_{N-1} + g_3(t)\, \bJ_{N-1} 
\end{bmatrix},
\end{equation*}
the interaction term \textcircled{5} can be written compactly as
\begin{align*}
   \textcircled{5} &= \big(
\bh_1(t) q +
\bh_2(t) s^i +
\bH_3(t)
\bs^{-i}
 +
\bh_4(t)
\big)^\top
\big(
\bh_5(t)
q +
\bh_6(t)
s^i +
\bH_7(t)
\bs^{-i}
 +
\bh_8(t)
\big).
\end{align*}

Finally, substituting these expressions into the HJB equation,  grouping terms according to the monomials in $(q, s, \bs^{-i})$ and matching the corresponding coefficients, we obtain a system of ODEs for the time-dependent coefficients $\tilde p_i(t)$, $\tilde r_i(t)$, and $\tilde u(t)$, which is precisely the system \eqref{homo_ode_system}.
\end{proof}

\begin{proof}[Proof of Proposition \ref{homo_existence}]
Observe that in \eqref{homo_ode_system}, the equations for $\{\tilde p_4(t),\tilde p_5(t),\tilde p_6(t),\tilde p_7(t)\}$ form an autonomous sub-system, while those for $\{\tilde p_1(t),\tilde p_2(t),\tilde p_3(t),\tilde r_1(t), \tilde r_2(t), \tilde r_3(t),\tilde u(t)\}$ are affine in $\{\tilde p_4(t),\tilde p_5(t),\tilde p_6(t),\tilde p_7(t)\}$ once the latter are solved. Moreover, 
\begin{equation*}
    \dot{\tilde p}_6(t) - \dot{\tilde p}_7(t) = (g_3(t)-g_2(t))(\tilde p_6(t)-\tilde p_7(t)), \quad \tilde p_6(T) - \tilde p_7(T) = 0,
\end{equation*}
which gives $\tilde p_6(t) = \tilde p_7(t)$. We conclude that the triple $\tilde \bp(t) = [\tilde p_4(t),\tilde p_5(t),\tilde p_6(t)]$ forms a closed Ricatti system. Reversing time via $\bar \bp(t) := \tilde \bp(T-t)$, the triple $\bar \bp(t)$ satisfies an initial-value ODE
\begin{equation} \label{bar_bp}
    \dot{\bar \bp}(t) = \begin{bmatrix}
        \bar \bp(t)^\top \calM_1 \bar\bp(t) \\
        \bar \bp(t)^\top \calM_2 \bar\bp(t) \\
        \bar \bp(t)^\top \calM_3 \bar\bp(t) 
    \end{bmatrix} + \begin{bmatrix}
        c_3 \\
        0   \\
        0
    \end{bmatrix}, \qquad \bar \bp(0) = \begin{bmatrix}
        c_4\\
        0 \\
        0\\
    \end{bmatrix},
\end{equation}
where the coefficient matrices $\calM_1$, $\calM_2$ and $\calM_3$ take the form
\begin{equation}\label{calM_matrix}
    \calM_1 = \begin{bmatrix}
        m^{(1)}_{44} & m^{(1)}_{45} & 0 \\
        m^{(1)}_{45}& m^{(1)}_{55}&0 \\
        0 & 0& 0
    \end{bmatrix},\quad
    \calM_2 = \begin{bmatrix}
        m^{(2)}_{44} & m^{(2)}_{45} & m^{(2)}_{46} \\
        m^{(2)}_{45}& m^{(2)}_{55}& m^{(2)}_{56}\\
        m^{(2)}_{46} & m^{(2)}_{56} & 0
    \end{bmatrix},\quad
    \calM_3 = \begin{bmatrix}
        m^{(3)}_{44} & m^{(3)}_{45} & m^{(3)}_{46} \\
       m^{(3)}_{45}& m^{(3)}_{55}& m^{(3)}_{56} \\
       m^{(3)}_{46} & m^{(3)}_{56} & 0
    \end{bmatrix}.
\end{equation}
A direct computation of these coefficients from the expression of $\tilde \bp(t)$ in equation \eqref{homo_ode_system} and intermediate quantities $\{g_2(t), g_3(t), g_6(t), g_7(t), \bh_2(t), \bh_6(t)\}$ gives
\begin{equation}\label{calM}
    \left\{
    \begin{aligned}
     m^{(1)}_{44} &= -\frac{4(c_1+c_2)}{d^2}+\frac{4c_1}{d^2\eta_1}+\frac{4c_1^2c_2}{d^4\eta_1^2},\quad
     m^{(1)}_{45} = (N-1)\big(\frac{6c_1^2+8c_1c_2}{d^3\eta_1}-\frac{4c_1^3+4c_1^2c_2}{d^4\eta_1^2}\big); \\
     m^{(1)}_{55} &= -(N-1)\frac{4(c_1+c_2)}{d^2\eta_1}-(N-1)^2\frac{2c_1^3+2c_1^2c_2}{d^4\eta_1^2}; \\
     m^{(2)}_{44} &= \frac{4c_1^2+4c_1c_2}{d^3\eta_1}-\frac{4c_1^3+4c_1^2c_2}{d^4\eta_1^2}, \quad
     m^{(2)}_{45} = -\frac{4(c_1+c_2)}{d^2\eta_1} -4(N-1)\frac{c_1^3+c_1^2c_2}{d^4\eta_1^2}; \\
     m^{(2)}_{55} &= -\frac{4N^2(c_1+c_2)c_1}{d^3\eta_1^2}+ \frac{4N(c_1+c_2)(3c_1^2+2c_1c_2-4c_2^2)}{d^4\eta_1^2}+\frac{4(c_1+c_2)(6c_1c_2+8c_2^2)}{d^4\eta_1^2}; \\ m^{(2)}_{46} &= \frac{(N-1)c_1}{d^2\eta_1},\quad
     m^{(2)}_{56} = -\frac{2(N-1)(c_1+c_2)}{d^2\eta_1}; \\
      m^{(3)}_{44} &= -\frac{4c_1^3+4c_1^2c_2}{d^4\eta_1^2},\quad m^{(3)}_{45} = \frac{8c_1(c_1+c_2)^2}{d^4\eta_1^2},\quad m^{(3)}_{55} = -\frac{16(c_1+c_2)^3}{d^4\eta_1^2}; \\
        m^{(3)}_{46} &= -\frac{4(c_1+c_2)}{d^2\eta_1},\quad m^{(3)}_{56} = -\frac{2N}{d\eta_1} + \frac{6c_1+8c_2}{d^2\eta_1}.
    \end{aligned}
    \right.
\end{equation}
Applying the norm $\|\cdot\|_2$ on both sides of equation \eqref{bar_bp} and using  \eqref{gamma} we have
\begin{equation*}
    \|\dot{\bar \bp}(t)\|_2 \leq \sqrt{\|\calM_1\|_2^2 + \|\calM_2\|_2^2 +\|\calM_3\|_2^2} \cdot \|\bar \bp(t)\|_2^2 + c_3 = \beta_{Hom} \|\bar \bp(t)\|_2^2 + c_3, \qquad \|\bar \bp(0)\|_2 = c_4.
\end{equation*}
 The upper bound on $T$ follows by the same argument as in the proof of Proposition \ref{general_wp}, with $N^2\beta$ replaced by $\beta_{Hom}$ and $\|\calC_3\|$ replaced by $c_3$.
\end{proof}

\begin{proof}[Proof of Proposition~\ref{parameters_prop}]
From the proof of Proposition \ref{homo_existence}, we have established that the existence of the solution is determined by the existence of the system $\bar{\bp}(t)$ defined by \eqref{bar_bp}. We will prove the result by analyzing the bounded simplex-like region
\begin{align}\label{invariant region}
\calI \equiv \calI_N^{c_1,c_2,c_3,c_4}
:=
\Big\{
(x_1,x_2,x_3)\in\bbR^3 \ \big|\
&0\le x_1\le \max\big\{c_4,\sqrt{\frac{c_3}{-\lambda_N(c_1,c_2)}}\big\},\\
 & 0\le x_2\le \frac{c_1}{2(c_1+c_2)}x_1,
-\frac{c_1}{Nc_1+2c_2}x_2\le x_3\le0 \nonumber
\Big\},
\end{align}
where
\[
\lambda_N(c_1,c_2)
=
-\frac{1}{c_1+c_2}
-
\frac{4Nc_1^3}{(c_1+2c_2)^2((N+1)c_1+2c_2)^2}.
\]

We show that if $\bar\bp(0)\in \calI$ then $\bar\bp(t') = [\bar p_4(t'), \bar p_5(t'),\bar p_6(t')]^\top \in \calI$ for all $t' > 0$. To wit, it is enough to verify that the vector field $\dot{\bar\bp}$ is inward-pointing or tangent on each boundary face of the domain $\calI$. 
Geometrically, the latter boundary consists of the \textit{vertex} $(0,0,0)$, one \textit{line segment} $\{(x_1,0,0):0\le x_1\le \max\big\{c_4,\sqrt{\frac{c_3}{-\lambda_N(c_1,c_2)}}\big\}\}$, and four two-dimensional faces:  \textit{upper $x_1$-face} $x_1=\max\big\{c_4,\sqrt{\frac{c_3}{-\lambda_N(c_1,c_2)}}\big\}$, \textit{upper $x_2$-face} $x_2=\frac{c_1}{2(c_1+c_2)}x_1$, \textit{upper $x_3$-face} $x_3=0$, and \textit{lower $x_3$-face} $x_3=-\frac{c_1}{Nc_1+2c_2}x_2$. 
We verify each of these six boundaries utilizing equations \eqref{bar_bp}-\eqref{calM} which reduces to analyzing the gradient with respect to the coordinate(s)  corresponding to the boundary face under consideration; the remaining coordinates are either not constrained by this face or correspond to one of the other five boundary faces.

\medskip
\noindent
\textbf{Case 1: Vertex $(0,0,0)$.} This corresponds to $\bar p_4(t)=\bar p_5(t)=\bar p_6(t)=0$. Thus, at the vertex,
\begin{equation*}
    \dot{\bar p}_4(t)=c_3 \geq 0,
\end{equation*}
which is inward-pointing or tangent w.r.t.~$\calI$.
Hence the vector field has nonnegative component in the $x_1$-direction. The $x_2$-direction corresponding to $\bar p_5(t)$ and $x_3$-direction $\bar p_6(t)$ are treated in Case 2 and Case 5, respectively.

\medskip
\noindent
\textbf{Case 2: Line Segment $\big\{(x_1,0,0): 0 \le x_1 \le \max\big\{c_4,\sqrt{\frac{c_3}{-\lambda_N(c_1,c_2)}}\big\}\big\}$} corresponding to $\bar p_5(t)=0, \bar p_6(t)=0$. 
 On this line segment,
\begin{equation*}
    \dot{\bar p}_5(t)=m^{(2)}_{44}\bar p_4(t)^2, \qquad\text{where  }\quad m^{(2)}_{44}
=
\frac{4c_1(c_1+c_2)(Nc_1+2c_2)}{d^4\eta_1^2}>0.
\end{equation*}
Therefore, $\dot{\bar p}_5(t)\ge0$. The $x_1$-direction corresponding to $\bar p_4(t)$ is controlled at the endpoints of the line segment, which correspond to Case 1 and Case 3. The $x_3$-direction corresponding to $\bar p_6(t)$ is treated in Case 5. Thus, the vector field is inward-pointing or tangent on the line segment. 

\medskip
\noindent
\textbf{Case 3: Upper $x_1$-face $\bar p_4(t)=\max\big\{c_4,\sqrt{\frac{c_3}{-\lambda_N(c_1,c_2)}}\big\}$.} 
Let
\begin{equation}\label{xi}
    \xi=\frac{\bar p_5(t)}{\bar p_4(t)}.
\end{equation}
Then
\begin{equation*}
    \dot{\bar p}_4(t)
=
\bar p_4(t)^2\big(m^{(1)}_{44}+2m^{(1)}_{45}\xi+m^{(1)}_{55}\xi^2\big)+c_3.
\end{equation*}
The maximum of the quadratic-in-$\xi$ expression multiplying $\bar{p}_4(t)^2$ is $\lambda_N(c_1,c_2)$. Since by the definition of $\calI$, $0\le \xi\le \frac{c_1}{2(c_1+c_2)}$, we obtain
\begin{equation*}
    \dot{\bar p}_4(t)
\le
\big(\max\big\{c_4,\sqrt{\frac{c_3}{-\lambda_N(c_1,c_2)}}\big\}\big)^2\lambda_N(c_1,c_2)+c_3\le 0.
\end{equation*}
For the remaining coordinates: the $x_2$-direction corresponding to $\bar p_5(t)$ on the endpoints are treated by Case 2 and Case 4, while $x_3$-direction corresponding to $\bar p_6(t)$ on the endpoints are treated by Case 5 and Case 6.
Hence the vector field is inward-pointing or tangent on this boundary.

\medskip
\noindent
\textbf{Case 4: Upper $x_2$-face $x_2=\frac{c_1}{2(c_1+c_2)}x_1$.} This corresponds to $\bar p_5(t)=\frac{c_1}{2(c_1+c_2)}\bar p_4(t)$. Differentiating $F_{4,5}(t) :=\frac{c_1}{2(c_1+c_2)}\bar p_4(t) - \bar p_5(t)$ gives
\begin{equation*}
\dot{F}_{4,5}(t)
= \frac{c_1}{2(c_1+c_2)}\Big(c_3
+
\frac{c_1^2}{c_1+c_2}
\frac{
(N^2-4N-1)c_1^2+(2N^2-2N-4)c_1c_2+4(N-1)c_2^2
}{d^4\eta_1^2}\Big)\bar p_4(t)^2.
\end{equation*}
It follows that for $N\ge5$, $\dot{F}_{4,5}(t) \ge 0$ is inward-pointing or tangent. For the remaining coordinates: the $x_1$-direction corresponding to $\bar p_4(t)$ on the endpoints are treated by Case 1 and Case 3, while $x_3$-direction corresponding to $\bar p_6(t)$ on the endpoints are treated by Case 5 and Case 6.
Hence the vector field is inward-pointing or tangent on this boundary.

\medskip
\noindent
\textbf{Case 5: Upper $x_3$-face $x_3=0$.}  This corresponds to $\bar p_6(t)=0$. By equation \eqref{xi}
\begin{equation*}
\dot{\bar p}_6(t) =
\bar p_4(t)^2\big(m^{(3)}_{44}+2m^{(3)}_{45}\xi+m^{(3)}_{55}\xi^2\big)\\
=
-\frac{4(c_1+c_2)\big(2(c_1+c_2)\xi-c_1\big)^2}{d^4\eta_1^2}\bar p_4(t)^2
\le0.
\end{equation*}
For the remaining coordinates: the $x_1$-direction corresponding to $\bar p_4(t)$ on the endpoints are treated by Case 1 and Case 3, while $x_2$-direction corresponding to $\bar p_5(t)$ on the endpoints are treated by Case 2 and Case 4.
Hence the vector field is inward-pointing or tangent on this boundary.

\medskip
\noindent
\textbf{Case 6: Lower $x_3$-face $x_3=-\frac{c_1}{Nc_1+2c_2}x_2$.} This corresponds to $\bar p_6(t)=-\frac{c_1}{Nc_1+2c_2}\bar p_5(t)$. Since by equation \eqref{xi} $0 \le \xi  \le \frac{c_1}{2(c_1+c_2)}$, let $\xi = \frac{\phi c_1}{2(c_1+c_2)}$, $0\le \phi \le 1$. Differentiating $F_{5,6}(t) := \bar p_6(t)+\frac{c_1}{Nc_1+2c_2}\bar p_5(t)$ gives
\begin{equation*}
\dot F_{5,6}(t) =
\frac{c_1^2\phi H_N(c_1,c_2,\phi)}
{(c_1+c_2)^2(c_1+2c_2)(Nc_1+2c_2)^2((N+1)c_1+2c_2)}
\bar p_4(t)^2,
\end{equation*}
where the numerator is
\begin{align*}
H_N(c_1,c_2,\phi)
&=\big(3N+1-(2N+1)\phi \big)c_1^3
+\big(7N+9+\phi(N^2-5N-7) \big)c_1^2c_2\\
&+(4N+16-12\phi)c_1c_2^2+(8-4\phi)c_2^3.
\end{align*}
For $N\ge5$, we have $H_N(c_1,c_2,\phi)\ge 0$ and therefore $\dot F_{5,6}(t) \ge0$ is inward-pointing or tangent. For the remaining coordinates: the $x_1$-direction corresponding to $\bar p_4(t)$ on the endpoints are treated by Case 1 and Case 3, while $x_2$-direction corresponding to $\bar p_5(t)$ on the endpoints are treated by Case 2 and Case 4.

Since the vector field on all boundary components is inward-pointing or tangent, if at $t\ge 0$, $\bar \bp(t) \in \calI$ then  $\bar{\bp}(t') \in \calI$ for all $t' > t$. In particular, since the initial condition is $\bar\bp(0)=[c_4,0,0]^\top\in \calI$, $\bar{\bp}(t) \in \calI$ for all $t \ge 0$, establishing boundedness and therefore global well-posedness on time interval $[0,T]$ for every $T>0$.
\end{proof}

\subsection{Proofs for Section \ref{sec:comparative}}\label{appen:comparative}

\begin{proof}[Proof of Proposition~\ref{prop:homo_mean_control}]
Under the homogeneous-agent setting, the state dynamics are permutation invariant, and the distribution of $\hat S_t^i$ and $\hat \alpha_t^i$, $i=1,2,\ldots,N$ are identical across agents. In particular, their expectations do not depend on $i$. 

From the equilibrium control in \eqref{homo_optimal control}, $
\hat\alpha_t^i
= g_1(t)Q_t + g_2(t)\hat S_t^i + g_3(t)\sum_{j\ne i}\hat S_t^j + g_4(t)$,
taking expectations on both sides and using exchangeability gives \eqref{mean_control}. Next, using the pricing relation $
\hat P_t^i=\bar P-c_1\big(Q_t-\sum_{j=1}^N \hat\alpha_t^j\big)$ 
and again using homogeneity, we obtain $
\E[\hat P_t^i]
= \bar P-c_1\big(m_Q(t)-N\E[\hat\alpha_t^i]\big)$, which yields \eqref{mean_price} after substituting \eqref{mean_control}.

To compute the mean $m_Q(t)$ and variance $V_Q(t)$ of $Q_t$, we first write the explicit solution of $Q_t$ from \eqref{Q process}
\begin{equation}\label{Q_sol}
    Q_t = e^{-\kappa t}Q_0 + \kappa\int_0^t e^{-\kappa(t-s)}\theta(s)\,ds + \sigma^0 \int_0^t e^{-\kappa(t-s)}\,dW_s^0.
\end{equation}
Since $Q_0$ is deterministic, taking expectation in \eqref{Q_sol} gives \eqref{mean_Q}, while It\^o isometry yields the well-known expression for $V_Q(t)$.

We now derive $m_S(t)$. Taking expectation on both sides of \eqref{S process} and substitute \eqref{mean_control} for $\E[\hat\alpha_t^i]$, we obtain an ODE for $m_S(t)$:
\begin{equation}\label{mean_S_ODE}
    \dot{m_S}(t) = (g_2(t)+(N-1)g_3(t))m_S(t)+(b(t)+g_1(t))m_Q(t) + a(t)+g_4(t), 
\end{equation}
with initial condition $m_S(0) = S_0$. Solving this linear ODE gives \eqref{mean_S} for $m_S(t)$.

We next study the second moments of the state processes. Recall the average equilibrium SOC process defined by $\bar S_t = \frac{1}{N}\sum_{i=1}^N \hat S_t^i$. Averaging \eqref{S process} over \(i\) and using
$
\bar\alpha_t:=\frac1N\sum_{i=1}^N \hat\alpha_t^i
= g_1(t)Q_t+\bigl(g_2(t)+(N-1)g_3(t)\bigr)\bar S_t+g_4(t)$, we obtain
\begin{align}\label{bar_S_SDE}
    d\bar S_t &= (a(t)+b(t)Q_t+\bar\alpha_t)\,dt + \sigma\rho dW_t^0 + \sigma\sqrt{1-\rho^2}\frac{1}{N} \sum_{j=1}^N d W_t^j \nonumber \\
    &= ((g_2(t)+(N-1)g_3(t))\bar S_t + (b(t)+g_1(t))Q_t+ a(t)+g_4(t))\,dt \nonumber\\ 
    &\quad + \sigma\rho dW_t^0 + \sigma\sqrt{1-\rho^2}\frac{1}{N} \sum_{j=1}^N d  W_t^j.
\end{align}

Then we are ready to derive $\Cov(Q_t,\bar S_t) := C_{Q \bar S}(t)$. By It\^o's formula,
\begin{align*}
    d(Q_t\bar S_t) &= Q_t\,d\bar S_t + \bar S_t\,dQ_t + d\langle Q,\bar S\rangle_t \\ 
    &= \Big[(g_2(t)+(N-1)g_3(t)-\kappa)Q_t\bar S_t +(b(t)+g_1(t)) Q_t^2 + (a(t)+g_4(t))Q(t) \\
    &\quad + \kappa\theta(t)\bar S_t + \sigma\sigma^0\rho\Big]\,dt + \bar S_t\sigma^0dW_t^0 + Q_t\big(\sigma\rho\,dW_t^0+\sigma\sqrt{1-\rho^2}\frac{1}{N}\sum_{j=1}^N d W_t^j\big).
\end{align*}
Taking expectations on both sides, writing $\E[Q_t\bar S_t] := m_{Q\bar S}(t)$, and use the ODEs satisfied by $m_Q(t)$, namely, $\dot{m}_Q(t) = \kappa(\theta(t)-m_Q(t))$, and by $m_{\bar S}(t)$ (cf. \eqref{mean_S_ODE}), we obtain
\begin{align*}
  \dot{C}_{Q\bar S}(t) &= \dot{m}_{Q\bar S}(t) -\dot{m}_Q(t)m_{\bar S}(t) - m_Q(t)\dot{m}_{\bar S}(t) \\
  &= (g_2(t) + (N-1)g_3(t)-\kappa)C_{Q\bar S}(t) + (b(t)+g_1(t)) V_Q(t) + \sigma\sigma^0\rho.
  \end{align*}
Solving this linear ODE yields \eqref{cov_Q_bar_S}. 

Next, to compute $\Var(\bar S_t)$, we apply It\^{o}'s formula to $\bar S_t^2$ yields
\begin{align*}
    d \bar S_t^2 &= 2\bar S_t\,d\bar S_t + d \langle\bar S,\bar S\rangle_t \\
    &= \Big[2\bar S_t((g_2(t) + (N-1)g_3(t))\bar S_t+(b(t)+g_1(t))Q_t+a(t)+g_4(t))\\ 
    &\quad +\sigma^2\rho^2+\frac{\sigma^2(1-\rho^2)}{N}\Big]\,dt + 2\bar S_t\big(\sigma\rho\,dW_t^0 + \sigma\sqrt{1-\rho^2}\frac{1}{N}\sum_{j=1}^N d W_t^j\big).
\end{align*}
 Let $\E[\bar S_t^2] = m_{\bar S^2}(t)$. Then
\begin{align*}
    \dot{V}_{\bar S}(t) &=  \dot{m}_{\bar S^2}(t)- 2\dot{m}_{\bar S}(t)m_{\bar S}(t)\\
    &= 2(g_2(t) + (N-1)g_3(t))V_{\bar S}(t) + 2 (b(t)+g_1(t)) C_{Q\bar S}(t) + \sigma^2\rho^2 + \frac{1-\rho^2}{N}\sigma^2.
\end{align*}
Solving this ODE together with the expression for $C_{Q\bar S}(t)$ yields \eqref{Var_bar_S}. 

Using the price identity $\hat P_t^i = \bar P - c_1\big(Q_t - \sum_{i=1}^N \hat\alpha_t^i\big) = \bar P - c_1 (1-Ng_1(t)Q_t) + c_1(g_2(t) + (N-1)g_3(t))N\bar S_t + c_1Ng_4(t)$, we express $\Var (\hat P_t^i)$ in terms of $V_Q(t)$, $V_{\bar S}(t)$ and $C_{Q\bar S}(t)$, which gives \eqref{var_price}.

To derive $\Var(\hat \alpha_t^i)$, we define the deviation process $\tilde S_t^i = \hat S_t^i - \bar S_t$. Subtracting \eqref{bar_S_SDE} from \eqref{S process} gives the SDE
\begin{equation}\label{tilde_S_SDE}
    d\tilde S_t^i = d \hat S_t^i - d\bar S_t =(g_2(t)-g_3(t))\tilde S_t^i\,dt + \sigma \sqrt{1-\rho^2}\big(d W_t^i - \frac{1}{N}\sum_{j=1}^N d W_t^j\big).
\end{equation}
Solving it explicitly gives $
    \tilde S_t^i = \sigma\sqrt{1-\rho^2}\int_0^t e^{\int_s^t (g_2(u)-g_3(u))\,du}\,\big( dW_s^i-\frac{1}{N}\sum_{j=1}^N  dW_s^j\big)$,
and hence 
\begin{equation}\label{Var_tilde_S}
    \Var(\tilde S_t^i) = \sigma^2(1-\rho^2)\big(1-\frac{1}{N}\big)\int_0^t e^{2\int_s^t (g_2(u)-g_3(u))\,du}\,ds. 
\end{equation}
This definition yields the nice decomposition $\hat\alpha_t^i = g_1(t) Q_t + (g_2(t) + (N-1)g_3(t))\bar S_t + g_4(t) + (g_2(t) - g_3(t)) \tilde S_t^i$. Moreover, $\Cov(Q_t,\tilde S_t^i)=0$, since $Q_t$ depends only on $W^0$ while $\tilde S_t^i$ is driven only by the idiosyncratic noises $W^1, \ldots, W^N$. Also, \(\Cov(\bar S_t,\tilde S_t^i)=0\): indeed, $\E[\bar S_t \tilde S_t^i]= \E[\tilde S_t^i]=0$ due to the symmetry of the agents. Therefore, using the previous results together with $\Cov(Q_t,\tilde S_t^i)=0$, $\Cov(\bar S_t,\tilde S_t^i)=0$,
we obtain \eqref{var_control}.
\end{proof}

\subsection{Proofs for Section \ref{sec:asymptotics}}\label{appen:asymptotics}

\begin{proof}[Proof of Proposition~\ref{p0_sol}]

To identify the leading-order coefficients and the first-order corrections of $\tilde p_i(t, \epsilon)$ and $\tilde r_i(t, \epsilon)$, we first rewrite the ODE system~\eqref{homo_ode_system} in terms of $\epsilon = 1/N$:
\begin{align}\label{homo_p_eps}
\begin{cases}
\begin{aligned}
-\dot{\tilde p}_1(t,\epsilon)
&= -c_1\, g_1(t,\epsilon)\, g_5(t,\epsilon)
+ c_2\, g_1^2(t,\epsilon)
- 2\kappa\, \tilde p_1(t,\epsilon)
+ 2\tilde p_2(t,\epsilon)g_1(t,\epsilon)  \\
&\quad
+ 2\Big(
\frac{(1 - \epsilon)c_{1}}{\epsilon d + c_{1}}
+ \frac{2(\epsilon-1)}{\epsilon d + c_{1}}\,\tilde p_2(t,\epsilon)
\Big) \tilde p_3(t,\epsilon)
+ 2\tilde p_2(t,\epsilon)b(t)
+ \frac{2(1-\epsilon)}{\epsilon}\tilde p_3(t,\epsilon)b(t),
\\
-2\dot{\tilde p}_2(t,\epsilon)
&= -c_1 \Big(
g_1(t,\epsilon)\, g_6(t,\epsilon)
+ g_2(t,\epsilon)\, g_5(t,\epsilon)
\Big)
+ 2c_2\, g_1(t,\epsilon)\, g_2(t,\epsilon)
- 2\kappa\, \tilde p_2(t,\epsilon) \\ 
&\quad
+ 2\tilde p_2(t,\epsilon)g_2(t,\epsilon)
+ \frac{2(1 - \epsilon) c_{1}}{\epsilon d^{2} + d c_{1}}\,
  \tilde p_3(t,\epsilon)\tilde p_4(t,\epsilon)
- \frac{2(1 - \epsilon)(c_{1} + d)}{\epsilon d^{2} + d c_{1}}\,
  \tilde p_3(t,\epsilon)\tilde p_5(t,\epsilon)\\
&\quad
+ 2 \tilde p_4(t,\epsilon)g_1(t,\epsilon)
+ 2\Big(
\frac{(1 - \epsilon)c_{1}}{\epsilon d + c_{1}}
+ \frac{2(\epsilon-1)}{\epsilon d + c_{1}}\,\tilde p_2(t,\epsilon)
\Big) \tilde p_5(t,\epsilon)
+ 2\tilde p_4(t,\epsilon)b(t)
\\
&\quad+ \frac{2(1-\epsilon)}{\epsilon}\tilde p_5(t,\epsilon)b(t),
\\
-2\dot{\tilde p}_3(t,\epsilon)
&= -c_1 \Big(
g_1(t,\epsilon)\, g_7(t,\epsilon)
+ g_3(t,\epsilon)\, g_5(t,\epsilon)
\Big)
+ 2c_2\, g_1(t,\epsilon)\, g_3(t,\epsilon)
- 2\kappa\, \tilde p_3(t,\epsilon) \\
&\quad
+ \Big(
2\tilde p_2(t,\epsilon)\, g_3(t,\epsilon)
+ 2\tilde p_3(t,\epsilon)\, g_2(t,\epsilon)
\Big)
 \\
&\quad + 2\Big(
\frac{2(1 - 2\epsilon) c_{1}}{\epsilon d^{2} + d c_{1}}\, \tilde p_4(t,\epsilon)
- \frac{2(1 - 2\epsilon)(c_{1} + d)}{\epsilon d^{2} + d c_{1}}\, \tilde p_5(t,\epsilon)
\Big)\tilde p_3(t,\epsilon) \\
&\quad
+ \Big(
2g_1(t,\epsilon)\, \tilde p_5(t,\epsilon)
+ 2g_1(t,\epsilon)\, \tilde p_7(t,\epsilon)
\Big) \\
&\quad
+ 2\Big(
\frac{(1 - 2\epsilon)c_{1}}{\epsilon d + c_{1}}
- \frac{2(1 - 2\epsilon)}{\epsilon d + c_{1}}\, \tilde p_2(t,\epsilon)
\Big)\tilde p_6(t,\epsilon) \\
&\quad
+ 2\tilde p_5(t,\epsilon)b(t)
+ 2\tilde p_7(t,\epsilon)b(t)
+ 2(N-2)\tilde p_6(t,\epsilon)b(t),
\\
-\dot{\tilde p}_4(t,\epsilon)
&= -c_1\, g_2(t,\epsilon)\, g_6(t,\epsilon)
+ c_2\, g_2^2(t,\epsilon)
+ c_3
+ 2\tilde p_4(t,\epsilon)g_2(t,\epsilon) \\
&\quad
+ 2\Big(
\frac{2(1 - \epsilon) c_{1}}{\epsilon d^{2} + d c_{1}}\, \tilde p_4(t,\epsilon)
- \frac{2(1 - \epsilon)(c_{1} + d)}{\epsilon d^{2} + d c_{1}}\, \tilde p_5(t,\epsilon)
\Big)\,\tilde p_5(t,\epsilon),
\\
-2\dot{\tilde p}_5(t,\epsilon)
&= -c_1 \Big(
g_2(t,\epsilon)\, g_7(t,\epsilon)
+ g_3(t,\epsilon)\, g_6(t,\epsilon)
\Big)
+ 2c_2\, g_2(t,\epsilon)\, g_3(t,\epsilon)\\
&\quad
+ \Big(
2\tilde p_4(t,\epsilon)\, g_3(t,\epsilon)
+ 2\tilde p_5(t,\epsilon)\, g_2(t,\epsilon)
\Big) \\
&\quad
+ 2 \Big(
\frac{2(1 - 2\epsilon) c_{1}}{\epsilon d^{2} + d c_{1}}\, \tilde p_4(t,\epsilon)
- \frac{2(1 - 2\epsilon)(c_{1} + d)}{\epsilon d^{2} + d c_{1}}\, \tilde p_5(t,\epsilon)
\Big)\tilde p_5(t,\epsilon)
 \\
&\quad + \Big(
2g_2(t,\epsilon)\, \tilde p_5(t,\epsilon)
+ 2g_3(t,\epsilon)\, \tilde p_7(t,\epsilon)
\Big)  \\
&\quad
+ 2\Big(
\frac{2(1 - 2\epsilon) c_{1}}{\epsilon d^{2} + d c_{1}}\, \tilde p_4(t,\epsilon)
- \frac{2(1 - 2\epsilon)(c_{1} + d)}{\epsilon d^{2} + d c_{1}}\, \tilde p_5(t,\epsilon)
\Big) \tilde p_6(t,\epsilon),
\\
-\dot{\tilde p}_6(t,\epsilon)
&= -c_1\, g_3(t,\epsilon)\, g_7(t,\epsilon)
+ c_2\, g_3^2(t,\epsilon)
+ 2\tilde p_5(t,\epsilon)\, g_3(t,\epsilon)  \\
&\quad
+ 2\Big(
\tilde p_6(t,\epsilon)\, g_2(t,\epsilon)
+ \tilde p_7(t,\epsilon)\, g_3(t,\epsilon)
\Big) \\
&\quad
+ 2\Big(
\frac{2(1 - 3\epsilon) c_{1}}{\epsilon d^{2} + d c_{1}}\, \tilde p_4(t,\epsilon)
- \frac{2(1 - 3\epsilon)(c_{1} + d)}{\epsilon d^{2} + d c_{1}}\, \tilde p_5(t,\epsilon)
\Big)\,\tilde p_6(t,\epsilon),
\\
-\dot{\tilde p}_7(t,\epsilon)
&= -c_1\, g_3(t,\epsilon)\, g_7(t,\epsilon)
+ c_2\, g_3^2(t,\epsilon)
+ 2\tilde p_5(t,\epsilon)\, g_3(t,\epsilon)
+ 2\tilde p_7(t,\epsilon)\, g_2(t,\epsilon) \\
&\quad
+ 2 \Big(
\frac{2(1 - 2\epsilon) c_{1}}{\epsilon d^{2} + d c_{1}}\, \tilde p_4(t,\epsilon)
- \frac{2(1 - 2\epsilon)(c_{1} + d)}{\epsilon d^{2} + d c_{1}}\,  \tilde p_5(t,\epsilon)
\Big)\tilde p_6(t,\epsilon).
\end{aligned}
\end{cases}
\end{align}
Here $g_i(t,\epsilon)$, $i= 1, \ldots, 8$, are defined in Theorem~\ref{Homo_case} and admit the following expansions in powers of $\epsilon$:
\begin{align*}
g_1(t,\epsilon)
&= \Big(1-\frac{2\tilde p_2^{(0)}(t)}{c_1}\Big)\epsilon + \cO(\epsilon^2),
\\
g_2(t,\epsilon)
&= \frac{2}{d}\big(\tilde p_5^{(0)}(t)-\tilde p_4^{(0)}(t)\big)
+\Big(
\frac{2}{d}\Big(\tilde p_5^{(1)}(t)-\tilde p_4^{(1)}(t)\Big)
-\frac{2}{d}\Big(\tilde p_5^{(0)}(t)-\tilde p_4^{(0)}(t)\Big)
-\frac{2}{c_1}\tilde p_5^{(0)}(t)
\Big)\epsilon
+ \cO(\epsilon^2),
\\
g_3(t,\epsilon)
&= \Big(\frac{2\tilde p_4^{(0)}(t)}{d}- \frac{2(c_1+d)}{dc_1}\tilde p_5^{(0)}(t)\Big)\epsilon
\\
&\quad
+\Big(
\frac{2\tilde p_4^{(1)}(t)}{d}
-\frac{2(c_1+d)\tilde p_5^{(1)}(t)}{dc_1}
-\Big(\frac{2\tilde p_4^{(0)}(t)}{c_1}-\frac{2(c_1+d)\tilde p_5^{(0)}(t)}{c_1^2}\Big)
\Big)\epsilon^2
+ \cO(\epsilon^3),
\\
g_4(t,\epsilon)
&= -\Big(\frac{1}{c_1}\tilde r_2^{(0)}(t) + \frac{\bar{P}}{c_1}\Big)\epsilon + \cO(\epsilon^2),
\\
g_5(t,\epsilon)
&= \frac{2}{c_1}\tilde p_2^{(0)}(t)
+\Big(\frac{2\tilde p_2^{(1)}(t)}{c_1}
- \frac{2d}{c_1^2}\tilde p_2^{(0)}(t)
+\frac{d}{c_1}\Big)\epsilon
+ \cO(\epsilon^2),
\\
g_6(t,\epsilon)
&= \frac{2\tilde p_5^{(0)}(t)}{c_1}
+\Big(\frac{2\tilde p_4^{(0)}(t)}{c_1}
+\frac{2\tilde p_5^{(1)}(t)}{c_1}
- \frac{2(d+c_1)}{c_1^2}\tilde p_5^{(0)}(t)\Big)\epsilon
+ \cO(\epsilon^2),
\\
g_7(t,\epsilon)
&= \frac{2\tilde p_5^{(0)}(t)}{c_1}
+\Big(\frac{2\tilde p_4^{(0)}(t)}{c_1}
+ \frac{2\tilde p_5^{(1)}(t)}{c_1}
- \frac{2(d+c_1)}{c_1^2}\tilde p_5^{(0)}(t)\Big)\epsilon
+ \cO(\epsilon^2),
\\
g_8(t,\epsilon)
&= \Big(\frac{\tilde r_2^{(0)}(t)}{c_1}+\frac{\bar{P}}{c_1}\Big)
-\Big(\frac{d\tilde r_2^{(0)}(t)}{c_1^2}
+\frac{d\bar{P}}{c_1^2}
-\frac{\tilde r_2^{(1)}(t)}{c_1}\Big)\epsilon
+ \cO(\epsilon^2).
\end{align*}

We next expand $\tilde p_i(t, \epsilon)$ according to \eqref{p_expansion}. Collecting terms of $\mathcal{O}(1)$ and matching the terminal conditions~\eqref{homo_terminal} gives the following ODE system for the leading-order terms $\tilde p_i^{(0)}$, $i = 4, 5, 6, 7$:
    \begin{align*}
\left\{
\begin{aligned}
-\dot{\tilde p}_4^{(0)}(t)
&=
-\frac{4}{d}\big(\tilde p_5^{(0)}(t)-\tilde p_4^{(0)}(t)\big)\tilde p_5^{(0)}(t)
+ c_2\Big(\frac{2}{d}(\tilde p_5^{(0)}(t)-\tilde p_4^{(0)}(t))\Big)^2 
+ c_3
-\frac{4}{d}\big(\tilde p_4^{(0)}(t)\big)^2\\
&\quad
-\frac{4(c_1+d)}{dc_1}\big(\tilde p_5^{(0)}(t)\big)^2,
\\
-\dot{\tilde p}_5^{(0)}(t)
&=
-\frac{2}{c_1}\big(\tilde p_5^{(0)}(t)\big)^2
+\frac{2}{d}\tilde p_4^{(0)}(t)\tilde p_6^{(0)}(t)
-\frac{2(c_1+d)}{dc_1}\tilde p_5^{(0)}(t)\tilde p_6^{(0)}(t),
\\
-\dot{\tilde p}_6^{(0)}(t)
&=
-\frac{4}{c_1}\tilde p_5^{(0)}(t)\tilde p_6^{(0)}(t),
\\
-\dot{\tilde p}_7^{(0)}(t)
&=
\frac{4}{d}\tilde p_7^{(0)}(t)\big(\tilde p_5^{(0)}(t)-\tilde p_4^{(0)}(t)\big)
+\frac{4}{d}\tilde p_4^{(0)}(t)\tilde p_6^{(0)}(t)
-\frac{4(c_1+d)}{dc_1}\tilde p_5^{(0)}(t)\tilde p_6^{(0)}(t),
\end{aligned}
\right.
\end{align*}
with terminal conditions $\tilde  p_4^{(0)}(T) = c_4$ and $\tilde p_5^{(0)}(T) =\tilde p_6^{(0)}(T) = \tilde p_7^{(0)}(T) = 0$.

By standard existence and uniqueness theory for linear ODE, we immediately obtain $\tilde p_6^{(0)}(t) \equiv 0$, then $\tilde p_5^{(0)}(t) \equiv 0$, and finally $\tilde p_7^{(0)}(t) \equiv 0$. Substituting these into the first equation reduces the system to the Riccati equation for  $\tilde p_4^{(0)}(t)$
\begin{equation}\label{p_4^{0}}
    \dot{\tilde p}_4^{(0)}(t) = \Big(\frac{4}{d}-\frac{4c_2}{d^2}\Big)(\tilde p_4^{(0)}(t))^2 - c_3, \quad 
    \tilde p_4^{(0)}(T) = c_4, 
\end{equation}
whose explicit solution is given in \eqref{p4^0 sol}.

Next, collecting terms of $\mathcal{O}(\epsilon)$ in \eqref{homo_p_eps} yields the following coupled ODE system for first-order corrections $\tilde p_i^{(1)}(t)$, $i = 4, 5, 6, 7$:
\begin{align*}
\left\{
\begin{aligned}
-\dot{\tilde p}_4^{(1)}(t)
&=
\Big(\frac{8}{d}-\frac{8c_2}{d^2}\Big)\big(\tilde p_4^{(0)}(t)\big)^2
+ \Big(\frac{12}{d}-\frac{8c_2}{d^2}\Big)
  \tilde p_4^{(0)}(t)\tilde p_5^{(1)}(t) 
+ \Big(\frac{8}{d}-\frac{8c_2}{d^2}\Big)
  \tilde p_4^{(0)}(t)\tilde p_4^{(1)}(t),
\\
-\dot{\tilde p}_5^{(1)}(t)
&=
\Big(\frac{4}{d}-\frac{4c_2}{d^2}\Big)\big(\tilde p_4^{(0)}(t)\big)^2
+ \frac{2}{d}\tilde p_4^{(0)}(t)\tilde p_6^{(1)}(t),
\\[1ex]
-\dot{\tilde p}_6^{(1)}(t)
&= 0,
\\[1ex]
-\dot{\tilde p}_7^{(1)}(t)
&=
\frac{4}{d}\tilde p_4^{(0)}(t)\tilde p_6^{(1)}(t)
- \frac{4}{d}\tilde p_7^{(1)}(t)\tilde p_4^{(0)}(t),
\end{aligned}
\right.
\end{align*}
with terminal conditions $\tilde p_4^{(1)}(T)  = \tilde p_5^{(1)}(T) = \tilde p_6^{(1)}(T) = 0 =  \tilde p_7^{(1)}(T) \equiv 0$. 

Using the same argument as for the leading-order system, we deduce $\tilde p_6^{(1)}(t) \equiv 0$, and then $\tilde p_7^{(1)}(t) \equiv 0$. Consequently, the system reduces to the couple equations for $\tilde p_4^{(1)}(t)$ and $\tilde p_5^{(1)}(t)$ givein in \eqref{p1_sys_p4}, which can be solved either numerically or analytically. In particular, combining equations~\eqref{p_4^{0}} with \eqref{p1_sys_p4} shows that the ODE for $\tilde p_5^{(1)}(t)$ simplifies to $\dot{\tilde{p}}_5^{(1)}(t) = -\dot{\tilde p}_4^{(0)}(t)-c_3$. Integrating from $t$ to $T$ and using the terminal conditions $\tilde p_4^{(0)}(T) = c_4$ and $\tilde p_5^{(1)}(T) = 0$, we obtain
\begin{equation}\label{p_4p_5}
    \tilde p_4^{(0)}(t) + \tilde p_5^{(1)}(t) = c_3(T-t)+c_4.
\end{equation}

We next consider the ODEs for the leading-order terms $\tilde p^{(0)}_2$ and $\tilde p^{(0)}_3$, obtained by collecting $\mathcal{O}(1)$ terms in \eqref{homo_p_eps}:
\begin{align*}
\left\{\begin{aligned} 
-\dot{\tilde{p}}_3^{(0)}(t) &= -\kappa \tilde p_3^{(0)}(t) - \frac{2}{c_1}\tilde p_3^{(0)}(t)\tilde p_5^{(0)}(t) + \tilde p_6^{(0)}(t) - \frac{2}{c_1}\tilde p_2^{(0)}(t)\tilde p_6^{(0)}(t)\\
& \quad + (\tilde p_5^{(0)}(t)+\tilde p_7^{(0)}(t) +\tilde p_6^{(1)}(t))b(t), \\
  -\dot{\tilde{p}}_2^{(0)}(t) &= -\kappa \tilde p_2^{(0)}(t) + \frac{1}{d} \tilde p_3^{(0)}(t)\tilde p_4^{(0)}(t) - \frac{c_1+d}{dc_1}\tilde p_3^{(0)}(t)\tilde p_5^{(0)}(t) + \tilde p_5^{(0)}(t) - \frac{2}{c_1}\tilde p_2^{(0)}(t)\tilde p_5^{(0)}(t) \\
    &\quad + (\tilde p_4^{(0)}(t)+\tilde p_5^{(1)}(t))b(t), 
\end{aligned}\right.
\end{align*}
with terminal conditions $\tilde p_2^{(0)}(T) = \tilde p_3^{(0)}(T) \equiv 0$.
Terms such as $\frac{2(1-2\epsilon)}{\epsilon}\tilde p_6^{(0)}(t)$ and $\frac{2(1-\epsilon)}{\epsilon}\tilde p_5^{(0)}(t)b(t)$ vanish because we have already shown that $\tilde p_5^{(0)}(t) = \tilde p_6^{(0)}(t) \equiv 0$. Moreover, using the fact $\tilde p_5^{(0)}(t)=\tilde p_6^{(0)}(t)=\tilde p_6^{(1)}(t)=\tilde p_7^{(0)}(t) \equiv 0$, we easily conclude $\tilde p_3^{(0)}(t) = 0$. Finally, using \eqref{p_4p_5}, the ODE for $\tilde p_2^{(0)}(t)$ simplifies to
\begin{align*}
    \dot{\tilde{p}}_2^{(0)}(t) = \kappa \tilde p_2^{(0)}(t) - (c_3(T-t)+c_4)b(t),
\end{align*}
whose solution is given in \eqref{p2^0 sol}. 

When rewriting \eqref{homo_ode_system} in terms of $\epsilon = 1/N$, the subsystem for $\tilde r_i(t,\epsilon)$ becomes:
\begin{align}\label{homo_r_eps}
\left\{
\begin{aligned}
-\dot{\tilde r}_1(t,\epsilon)
&=
\bar{P}\, g_1(t,\epsilon)
- c_1 \big(
g_4(t,\epsilon) g_5(t,\epsilon)
+ g_1(t,\epsilon) g_8(t,\epsilon)
\big)
+ 2c_2\, g_1(t,\epsilon) g_4(t,\epsilon)
\\
&\quad
+ \big(
2\kappa\, \tilde p_1(t,\epsilon)\theta(t)
- \kappa\, \tilde r_1(t,\epsilon)
\big)
+ 2\tilde p_2(t,\epsilon)\big(g_4(t,\epsilon)+a(t)\big)
\\
&\quad
- 2\Big(
\frac{1-\epsilon}{\epsilon d+c_1}\tilde r_2(t,\epsilon)
+ \frac{\bar{P}(1-\epsilon)}{\epsilon d+c_1}
\Big)\tilde p_3(t,\epsilon)
+ \frac{2(1-\epsilon)}{\epsilon}\tilde p_3(t,\epsilon)a(t)
\\
&\quad
+ \tilde r_2(t,\epsilon)\big(g_1(t,\epsilon)+b(t)\big)
+ \Big(
\frac{(1-\epsilon)c_1}{\epsilon d+c_1}
+ \frac{2\epsilon-2}{\epsilon d+c_1}\tilde p_2(t,\epsilon)
+ \frac{1-\epsilon}{\epsilon}b(t)
\Big)\tilde r_3(t,\epsilon),
\\
-\dot{\tilde r}_2(t,\epsilon)
&=
\bar{P}\, g_2(t,\epsilon)
- c_1 \big(
g_4(t,\epsilon) g_6(t,\epsilon)
+ g_2(t,\epsilon) g_8(t,\epsilon)
\big)
+ 2c_2\, g_2(t,\epsilon) g_4(t,\epsilon)
- 2c_3\, \zeta(t)
\\
&\quad
+ 2\kappa\, \tilde p_2(t,\epsilon)\theta(t,\epsilon)
+ 2\tilde p_4(t,\epsilon)\big(g_4(t,\epsilon)+a(t)\big)
\\
&\quad
- 2\Big(
\frac{1-\epsilon}{\epsilon d+c_1}\tilde r_2(t,\epsilon)
+ \frac{\bar{P}(1-\epsilon)}{\epsilon d+c_1}
\Big)\tilde p_5(t,\epsilon)
+ \frac{2(1-\epsilon)}{\epsilon}\tilde p_5(t,\epsilon)a(t)
\\
&\quad
+ \tilde r_2(t,\epsilon) g_2(t,\epsilon)
+ \Big(
\frac{2(1-\epsilon)c_1}{\epsilon d^2+dc_1}\tilde p_4(t,\epsilon)
- \frac{2(1-\epsilon)(c_1+d)}{\epsilon d^2+dc_1}\tilde p_5(t,\epsilon)
\Big)\tilde r_3(t,\epsilon),
\\
-\dot{\tilde r}_3(t,\epsilon)
&=
\bar{P}\, g_3(t,\epsilon)
- c_1 \big(
g_4(t,\epsilon) g_7(t,\epsilon)
+ g_3(t,\epsilon) g_8(t,\epsilon)
\big)
+ 2c_2\, g_3(t,\epsilon) g_4(t,\epsilon) \\
&\quad
+ 2\kappa\, \tilde p_3(t,\epsilon)\theta(t)
+ \big(g_4(t,\epsilon)+a(t)\big)
  \big(2\tilde p_5(t,\epsilon)+2\tilde p_7(t,\epsilon)\big)
\\
&\quad
- 2\Big(
\frac{1-2\epsilon}{\epsilon d+c_1}\tilde r_2(t,\epsilon)
+ \frac{\bar{P}(1-2\epsilon)}{\epsilon d+c_1}
\Big)\tilde p_6(t,\epsilon)
+ \frac{2(1-2\epsilon)}{\epsilon}\tilde p_6(t,\epsilon)a(t)
\\
&\quad
+ \tilde r_2(t,\epsilon) g_3(t,\epsilon)
+ \tilde r_3(t,\epsilon) g_2(t,\epsilon)
\\
&\quad
+ \Big(
\frac{2(1-2\epsilon)c_1}{\epsilon d^2+dc_1}\tilde p_4(t,\epsilon)
- \frac{2(1-2\epsilon)(c_1+d)}{\epsilon d^2+dc_1}\tilde p_5(t,\epsilon)
\Big)\tilde r_3(t,\epsilon).
\end{aligned}
\right.
\end{align}
The  expansion and coefficient-matching procedure is analogous to that for the $\tilde p_i(t,\epsilon)$ system. At the leading order, the equation for $\tilde r_3^{(0)}(t)$ reduces $-\dot{\tilde{r}}_3^{(0)}(t) = 0$ with terminal condition $\tilde r_3^{(0)}(T) = 0$, and hence $\tilde r_3^{(0)}(t) \equiv 0$.
For $\tilde r_2^{(0)}(t)$, we have 
\begin{align*}
    -\dot{\tilde{r}}_2^{(0)}(t) &= -2c_3\zeta(t) + 2\kappa\tilde p_2^{(0)}(t)\theta(t) + 2(\tilde p_5^{(1)}(t)+\tilde p_4^{(0)}(t))a(t) + \frac{2}{d}\tilde p_4^{(0)}(t)\tilde r_3^{(0)}(t) \\
    &= -2c_3\zeta(t) + 2\kappa\tilde p_2^{(0)}(t)\theta(t) + 2(c_3(T-t)+c_4)a(t),
\end{align*}
with terminal condition $\tilde r_2^{(0)}(T) = -2c_4\zeta(T)$, where we have used $\tilde r_3^{(0)}(t) \equiv 0$ and the relation~\eqref{p_4p_5}.
Solving this ODE and substituting $\tilde p_2^{(0)}(t)$ from \eqref{p2^0 sol} yields \eqref{r_2^0 sol}. 
\end{proof}

\begin{proof}[Proof of Theorem \ref{regional_control}]
    Using the expansions of $g_i(t,\epsilon)$, $i=1,2,3,4$, together with \eqref{homo_optimal control}, the equilibrium control function $\hat{\alpha}^{i,\epsilon}(t,q,\bs)$ admits the expansion
\begin{align*}
\hat{\alpha}^{i,\epsilon}(t,q,\bs)
&=
\epsilon
\Big(1-\frac{2\tilde p_2^{(0)}(t)}{c_1}\Big) q
+ \frac{2}{d}\tilde p_4^{(0)}(t)\big(\bar s-s^i\big)
+ \epsilon\,
\frac{2}{d}\big(\tilde p_5^{(1)}(t)-\tilde p_4^{(1)}(t)\big)s^i
- \epsilon\Big(\frac{1}{c_1}\tilde r_2^{(0)}(t)+\frac{\bar{P}}{c_1}\Big)
\\
&\quad
+ \epsilon^2
\Big(
\frac{2}{d}\tilde p_4^{(1)}(t)
- \frac{2(c_1+d)}{dc_1}\tilde p_5^{(1)}(t)
- \frac{2}{c_1}\tilde p_4^{(0)}(t)
\Big)\mathbf 1_{N-1}^{\top}\bs^{-i}
+ \cO(\epsilon^2).
\end{align*}
where we use the equality $\epsilon\mathbf{1}_{N-1}^{\top} \bs^{-i} = \bar s - \epsilon s^i$ since $\epsilon = \frac{1}{N}$. This gives \eqref{regional_optimal} after noticing $\hat\alpha_t^{i,\epsilon} = \hat\alpha^{i,\epsilon}(t, Q_t, \hat \bS_t)$.

To derive the aggregate control function, we sum  \eqref{regional_optimal} over $i = 1, \ldots, N$. Using $\sum_{i=1}^N \epsilon q = q$ and $\sum_{i=1}^N s^i = N \bar s$, we obtain 
  \begin{align*}
\sum_{i=1}^N \hat\alpha^{i,\epsilon}(t,q,\bs)
&=
\sum_{i=1}^N \Big(1-\frac{2\tilde p_2^{(0)}(t)}{c_1}\Big)\epsilon q
+ \Big(
-\frac{2}{d}\tilde p_4^{(0)}(t)
+ \big(
\frac{2}{d}\big(\tilde p_5^{(1)}(t)-\tilde p_4^{(1)}(t)\big)
+ \frac{2}{d}\tilde p_4^{(0)}(t)
\big)\epsilon
\Big)s^{i}
\nonumber\\
&\quad \qquad
+ \Big(
\frac{2}{d}\tilde p_4^{(0)}(t)\epsilon
+ \Big(
\frac{2}{d}\tilde p_4^{(1)}(t)
- \frac{2(c_{1}+d)}{d c_{1}}\tilde p_5^{(1)}(t)
- \frac{2}{c_{1}}\tilde p_4^{(0)}(t)
\Big)\epsilon^{2}
\Big)\mathbf{1}^{\top}\bs^{-i} \nonumber\\
&\quad \qquad - \Big(\frac{1}{c_{1}}\tilde r_2^{(0)}(t)+\frac{c_{0}}{c_{1}}\Big)\epsilon
+ \cO(\epsilon^2)
\nonumber\\
&=
\Big(1-\frac{2\tilde p_2^{(0)}(t)}{c_1}\Big)q
+ \frac{2}{d}\big(\tilde p_5^{(1)}(t)-\tilde p_4^{(1)}(t)\big)\bar{s}
\nonumber\\
&\quad
+ \sum_{i=1}^{N}
\Big(
\frac{2}{d}\tilde p_4^{(1)}(t)
- \frac{2(c_{1}+d)}{d c_{1}}\tilde p_5^{(1)}(t)
- \frac{2}{c_{1}}\tilde p_4^{(0)}(t)
\Big)\epsilon^{2}
\big(\sum_{j=1}^{N} s_{j} - s_{i}\big) \nonumber\\
&\quad
- \Big(\frac{1}{c_{1}}\tilde r_2^{(0)}(t)+\frac{c_{0}}{c_{1}}\Big)
+ \cO(\epsilon)
\nonumber\\
&=
\Big(1-\frac{2\tilde p_2^{(0)}(t)}{c_1}\Big)q
+ \frac{2}{d}\big(\tilde p_5^{(1)}(t)-\tilde p_4^{(1)}(t)\big)\bar{s}
+ \Big(
\frac{2}{d}\tilde p_4^{(1)}(t)
- \frac{2(c_{1}+d)}{d c_{1}}\tilde p_5^{(1)}(t)
- \frac{2}{c_{1}}\tilde p_4^{(0)}(t)
\Big)\bar{s}
\nonumber\\
&\quad
- \Big(\frac{1}{c_{1}}\tilde r_2^{(0)}(t)+\frac{c_{0}}{c_{1}}\Big)
+ \cO(\epsilon)
\nonumber\\
&=
\Big(1-\frac{2\tilde p_2^{(0)}(t)}{c_1}\Big)q
- \frac{2}{c_{1}}\big(\tilde p_5^{(1)}(t)+\tilde p_4^{(0)}(t)\big)\bar{s}
- \Big(\frac{1}{c_{1}}\tilde r_2^{(0)}(t)+\frac{c_{0}}{c_{1}}\Big)
+ \cO(\epsilon),
\end{align*}
which yields \eqref{agg_optimal}.
\end{proof}

\begin{proof}[Proof of Theorem \ref{bar_S_t_limit}]
By the definition of $S_t^i$ and Theorem~\ref{regional_control}, the averaged SOC process $\bar S_t$ satisfies the SDE 
\begin{equation}\label{bar S_t}
d\bar S_t
=
\big(
a(t)
+ b(t) Q_t
+ \frac{1}{N} G(t,Q_t,\bar S_t) 
+ R_N
\big)\, dt
+ \sigma\rho\, dW_t^0
+ \frac{\sigma}{\sqrt N}\sqrt{1-\rho^2}\, d\overline W_t,
\end{equation}
with initial condition $\bar S_0 = S_0$. Here $\overline W_t$ is a standard Brownian motion independent of $W_t^0$, and $G(t,q ,\bar s)
=
\big(1-\frac{2\tilde p_2^{(0)}(t)}{c_1}\big)q
- \frac{2}{c_1}\big(c_3(T-t)+c_4\big)\bar s
- \big(\frac{1}{c_1}\tilde r_2^{(0)}(t)+\frac{\bar{P}}{c_1}\big)$ represents the leading-order contribution of the aggregate control. The remainder term $R_N \sim \cO(\frac{1}{N^2})$. 

We next introduce the limiting process $\bar S_t^*$ defined by
\begin{align}\label{bar S_t star}
    d\bar S_t^* = (a(t) + b(t)Q_t)\,dt + \sigma\rho\,dW_t^0, \qquad \bar S_0^* = S_0.
\end{align}
Our goal is to prove that, for any fixed $T$, there exists a constant $C_T>0$ independent of $N$, such that $\sup_{0\leq t \leq T}\mathbb{E}\big[|\bar S_t - \bar S_t^*|^2\big] \leq \frac{C_T}{N}$.

First, observe that the coefficients for the coupled system for $Q_t, \bar S_t, \bar S_t^\ast$ (cf. \eqref{Q process}, \eqref{S process} and \eqref{bar S_t star}) are global Lipchitz with linear growth, the system admits a unique strong solution. Standard moment estimates then yield the uniform second-moment bound
\begin{equation}\label{second moment bound}
    \sup_{N\geq1}\sup_{0\leq t \leq T} \mathbb{E}[|Q_t|^2+ |\bar S_t|^2 + |\bar S_t^\ast|^2] < \infty.
\end{equation}

Define the difference process $X^N_t = \bar S_t - \bar S_t^*$. Subtracting \eqref{bar S_t star} from \eqref{bar S_t}, we obtain that $X^N_t$ satisfies
\begin{equation*}
dX_t^N
=
\Big[
\frac{1}{N}G(t,Q_t,\bar S_t)
+ R_N
\Big]\,dt
+ \frac{\sigma}{\sqrt N}\sqrt{1-\rho^2}\, d\overline W_t, \quad X_0^N = 0.
\end{equation*}
By the linear growth of $G$ in $(Q_t,\bar S_t)$ and the estimate $R_N \sim \mathcal{O}\big(\frac{1}{N^2}\big)$, there exist constants $C_G, C_R>0$ such that
\begin{equation}\label{X^N_t drift bound}
\Big|
\frac{1}{N}G(t,Q_t,\bar S_t)+R_N
\Big|
\leq
\frac{1}{N}\big|G(t,Q_t,\bar S_t)\big|
+ |R_N|
\leq
\frac{C_G}{N}\big(1+|Q_t|+|\bar S_t|\big)
+ \frac{C_R}{N^2}.
\end{equation}
Applying Ito's formula to $(X_t^N)^2$ gives
\begin{equation*}
d(X_t^N)^2
=
\Big(
2X_t^N\Big(\frac{1}{N}G(t,Q_t,\bar S_t)+R_N\Big)
+ \frac{\sigma^2(1-\rho^2)}{N}
\Big)\, dt
+ 2X_t^N\frac{\sigma\sqrt{1-\rho^2}}{\sqrt N}\, d\overline W_t.
\end{equation*}
Integrating from 0 to $t$, using $X_t^N = 0$ and taking expectations on both sides yields
\begin{equation*}
\mathbb{E}\big[(X_t^N)^2\big] =
\int_0^t
\mathbb{E}\Big[
2X_s^N\Big(\frac{1}{N}G(s,Q_s,\bar S_s)+R_N\Big)
+ \frac{\sigma^2(1-\rho^2)}{N}
\Big]\, ds,
\end{equation*}
where the It\^o integral term varnishes since $X_t^N$ is squared integrable on $[0,T]$ which follows from \eqref{second moment bound}. 

We now estimate the drift term $2X_s^N\big(\frac{1}{N}G(s,Q_s,\bar S_s) + R_N\big)$. Using Young's inequality $2ab \leq \frac{1}{2}a^2+2b^2$, together with \eqref{X^N_t drift bound}, we deduce
\begin{align*}
2\Big|
X_s^N\Big(\frac{1}{N}G(s,Q_s,\bar S_s)+R_N\Big)
\Big|
&\leq
2|X_s^N|
\Big(
\frac{C_G}{N}\big(1+|Q_s|+|\bar S_s^N|\big)
+ \frac{C_R}{N^2}
\Big)
\\
&\leq
\frac{1}{2}|X_s^N|^2
+ \frac{4C_G^2}{N^2}\big(1+|Q_s|+|\bar S_s^N|\big)^2
+ \frac{4C_R^2}{N^4}. 
\end{align*}
Using the uniform second-moment bound \eqref{second moment bound}, there exists a constant \(C_1>0\) such that,
\begin{equation*}
\mathbb{E}\Big[
2X_s^N\Big(\frac{1}{N}G(s,Q_s,\bar S_s)+R_N\Big)
\Big]
\leq
\frac{1}{2}\mathbb{E}\big[|X_s^N|^2\big]
+ \frac{C_1}{N^2}.
\end{equation*}
Consequently, for some constant $C_2 > 0$,
\begin{equation*}
\mathbb{E}\big[(X_t^N)^2\big]
\leq
\frac{1}{2}\int_0^t \mathbb{E}\big[(X_s^N)^2\big]\, ds
+ \frac{C_1 t}{N^2}
+ \frac{\sigma^2(1-\rho^2)t}{N}\leq
\frac{1}{2}\int_0^t \mathbb{E}\big[(X_s^N)^2\big]\, ds
+ \frac{C_2 T}{N}.
\end{equation*}
since $t \le T$ and $N \ge 1$. By Gronwall's inequality, for any $0\leq t \leq T$,
\begin{align*}
   \mathbb{E}[(X_t^N)^2] &\leq \frac{C_2 T e^{t/2}}{N} \leq \frac{C_2 T e^{T/2}}{N},
\end{align*}
Setting $C_T = C_2 T e^{T/2}$ completes the proof.
\end{proof}

\end{document}